\documentclass[12pt]{amsart}

\usepackage{etex}
\usepackage{amsmath, amssymb}
\usepackage{array}
\usepackage[frame,cmtip,arrow,matrix,line,graph,curve]{xy}
\usepackage{graphpap, color, paralist, pstricks}
\usepackage[mathscr]{eucal}
\usepackage[pdftex]{graphicx}
\usepackage[pdftex,colorlinks,backref=page,citecolor=blue]{hyperref}
\usepackage{tikz}
\usepackage{subcaption}
\usepackage{rotating}

\numberwithin{equation}{section}
\numberwithin{figure}{section}

\newtheorem{theorem}{Theorem}[section]
\newtheorem{theoremletter}{Theorem}

\newtheorem{proposition}[theorem]{Proposition}

\newtheorem{lemma}[theorem]{Lemma}
\newtheorem{observation}[theorem]{Observation}

\newtheorem{conjecture}[theorem]{Conjecture}

\theoremstyle{definition}
\newtheorem{definition}[theorem]{Definition}

\newtheorem{remark}[theorem]{Remark}

\newtheorem{algorithm}[theorem]{Algorithm}

\newcommand{\CC}{\mathbb{C} }

\newcommand{\ZZ}{\mathbb{Z} }

\newcommand{\RR}{\mathbb{R} }
\newcommand{\NN}{\mathbb{N} }

\newcommand{\cA}{\mathcal{A} }

\newcommand{\cO}{\mathcal{O} }

\newcommand{\cS}{\mathcal{S} }
\newcommand{\cT}{\mathcal{T} }

\newcommand{\be}{\mathbf{e} }
\newcommand{\bm}{\mathbf{m} }

\newcommand{\bv}{\mathbf{v} }

\newcommand{\ri}{\mathrm{i} }
\newcommand{\rf}{\mathrm{f} }

\def\SL{\mathrm{SL}}

\def\spec{\mathrm{Spec}\, }

\def\RMC{\mathsf{RMC}}
\def\RML{\mathsf{RML}}
\def\NMC{\mathsf{NMC}}
\def\NML{\mathsf{NML}}
\def\type{\mathrm{type}}

\def\Star{\mathrm{Star}}

\def\Teich{\mathscr{T}^d(\Sigma)}
\def\Curve{\mathscr{C}(\Sigma)}
\def\Arc{\mathscr{S}_{h}(\Sigma)}
\def\Muller{\mathrm{Sk}_{q}(\hat{\Sigma})}
\def\Mullerone{\mathrm{Sk}_{1}(\hat\Sigma)}

\hypersetup{%
pdftitle={On the injectivity conjecture},%
pdfauthor={Han-Bom Moon, Helen Wong},%
pdfkeywords={skein algebra},%
citecolor=blue,%
linkcolor=blue,%
}

\begin{document}

\title[Roger-Yang skein algebra]{The Roger-Yang skein algebra and the decorated Teichm\"uller space}
\date{\today}

\author{Han-Bom Moon}
\address{Department of Mathematics, Fordham University, New York, NY 10023}
\email{hmoon8@fordham.edu}

\author{Helen Wong}
\address{Department of Mathematical Sciences, Claremont McKenna College, Claremont, CA 91711}
\email{hwong@cmc.edu}

\begin{abstract}
Based on hyperbolic geometric considerations, Roger and Yang introduced an extension of the Kauffman bracket skein algebra that includes arcs.  In particular, their skein algebra is a deformation quantization of a certain commutative curve algebra, and there is a Poisson algebra homomorphism between the curve algebra and the algebra of smooth functions on decorated Teichm\"uller space.  In this paper, we consider surfaces with punctures which is not the 3-holed sphere and which have an  ideal triangulation without self-folded edges or triangles.  For those surfaces, we prove that Roger and Yang's Poisson algebra homomorphism is injective, and the skein algebra they defined have no zero divisors.   A section about generalized corner coordinates for normal arcs may be of independent interest.  
\end{abstract}

\maketitle

\section{Introduction}
 Let $\Sigma$ be a closed surface with finitely many punctures.  Defined by Penner \cite{Penner87}, the decorated Teichm\"uller space $\Teich$ consists of complete, finite area hyperbolic metrics on $\Sigma$ that come with a choice of horoball around each puncture.    This paper describes progress in a program initiated by Roger and Yang in \cite{RogerYang14} to establish a certain skein algebra $\Arc$ as a quantization of $\Teich$.  
  
One of Roger and Yang's objective was to extend the rich body of work showing that the Kauffman bracket skein algebra is a quantization of the usual Teichm\"uller space, via the $\SL_{2}$-character variety \cite{HostePrz90, Turaev91, Bullock97,  BullockFrohmanJKB99,  PrzytyckiSikora00, BullockFrohmanJKB02, CharlesMarche12}. In particular, they introduced an extension $\Arc$ of the Kauffman bracket skein algebra to the case of punctured surfaces that uses arcs.  Observe that, in contrast to the usual Teichm\"uller space,  in  $\Teich$ one can assign a length to arcs that go from puncture to puncture, by truncating at the horoballs.  This way of assigning lengths leads to the so-called lambda-length functions which parametrize $\Teich$ (\cite{Penner87}). Roger and Yang's skein algebra $\Arc$ is generated by both framed arcs and loops, and  an indeterminate variable for each of the punctures that accounts for the size of the horoballs decorations.  Besides the two usual Kauffman bracket skein relations, $\Arc$ has two more which, by  design, match the relations between lambda-length functions of arcs. For the definition of $\Arc$, see Section~\ref{sec:arcalgebra}.

Roger and Yang also define a commutative curve algebra $\Curve$ generated by loops and arcs in $\Sigma$ (see Section~\ref{sec:curvealgebra} for the relations), and they show that it has a Poisson bracket that generalizes the Goldman bracket formula for the Kauffman bracket skein algebra.  Furthermore, there is a Poisson algebra homomorphism 
\begin{equation}\label{eqn:phi}
	\Phi : \Curve \to C^{\infty}(\Teich),
\end{equation}
 where $ C^{\infty}(\Teich)$ is the algebra of $\CC$-valued smooth functions on $\Teich$.   The Poisson structure on $ C^{\infty}(\Teich)$ is the Weil-Petersson symplectic structure \cite{Penner92} whose action on lambda-length functions was explicitly computed by \cite{Mondello09}.    
 
 Roger and Yang show that the skein algebra, $\Arc$ is a deformation quantization of $\Curve$ (\cite[Theorem 1.1]{RogerYang14}).  It then follows that $\Arc$ seems a likely candidate for quantization of $\Teich$. However, there remain several technical hurdles to the program Roger and Yang sketched out.

\subsection{Main results}

The main purpose of this paper is to present progress toward the biggest obstacle, which Roger and Yang conjectured, as below.  

\begin{conjecture}[\protect{\cite[Conjecture 3.17]{RogerYang14}}]\label{conj:injectivity}
The Poisson algebra homomorphism $\Phi$ in \eqref{eqn:phi} is injective. 
\end{conjecture}

In this paper, we provide an overall strategy for proving the conjecture,  and carry it out in full for \emph{locally planar} surfaces, which are surfaces that have an ideal triangulation without self-folded edges or triangles (and is not the three-puncture sphere). Note that for any surface, if we drill enough extra points, then it becomes locally planar. The following two theorems are the main results of this paper. 

\begin{theoremletter}[\protect{Theorem \ref{thm:nonzerodivisorimpliesinjectivity}}]\label{thm:nonzerodivisorimpliesinjectivityintro} 
 If  $\Sigma$ has an ideal triangulation such that no edge of the triangulation is a zero divisor in $\Curve$, then $\Phi$ in \eqref{eqn:phi} is injective.
\end{theoremletter}

\begin{theoremletter}[\protect{Theorem \ref{thm:locallyplanarzerodivisor}}]\label{thm:mainthmintro} 
If $\Sigma$ is locally planar, then no edge of a locally planar triangulation is a zero divisor in $\Curve$. In particular, $\Phi$ in \eqref{eqn:phi} is injective. 
\end{theoremletter}

An interesting algebraic consequence of the injectivity of $\Phi$ is the following, which is proved in Section~\ref{sec:consequences}.

\begin{theoremletter} [\protect{Theorem~\ref{thm:nozerodivisors}}]\label{thm:nozerodivisorsintro} 
If Conjecture~\ref{conj:injectivity} is true, then $\Curve$ and its quantization $\Arc$ are domains. In particular, if $\Sigma$ is locally planar, $\Curve$ and $\Arc$ are domains. 
\end{theoremletter} 

A similar statement for the Kauffman bracket skein algebra $\cS^{A}(\Sigma)$ appeared in \cite{PrzytyckiSikora00, PrzytyckiSikora19}, and was a necessary step in showing that when $A=\pm1$, $\cS^{\pm1}(\Sigma)$ is isomorphic with the coordinate ring of the $\SL_{2}$-character variety \cite{Bullock97, CharlesMarche12}.

In addition, we developed a generalization of the theory of normal curves on a surface (as in \cite{Matveev07}) to describe normal arcs.  Whereas the corner coordinates of a normal loop is an integer, the generalized corner coordinate of a normal arc ending at a puncture is the negative fraction $-\frac{1}{2}$.  The generalized corner coordinates satisfy the usual matching equation at edges.  See Section~\ref{sec:generalizedcornercoord}, which may be of independent interest in combinatorial topology.

\subsection{Summary of the proof}

We give here a brief summary of the main points of the proofs of Theorems \ref{thm:nonzerodivisorimpliesinjectivityintro} and \ref{thm:mainthmintro}. 

The key insight for Theorem~\ref{thm:nonzerodivisorimpliesinjectivityintro}  is to consider the  localization  $S^{-1}\Curve$ by the multiplicative set, $S$, that is generated by edges of an ideal triangulation.  If $\lambda_i$ denotes the lambda-length function of the $i$-th edge of the triangulation, we show that $\Phi : \Curve \to C^{\infty}(\Teich)$  factors through $\CC[\lambda_{i}^{\pm}]$ and its localized map $\Psi : S^{-1}\Curve \to \CC[\lambda_{i}^{\pm}]$ is an isomorphism.  Furthermore, we show that if none of the edges are zero divisors, then the localization map $L : \Curve \to S^{-1}\Curve$ is injective.  This implies the injectivity of $\Phi$.  See Section~\ref{sec:outline}.  

The proof of Theorem~\ref{thm:mainthmintro} is significantly more complicated, and we only mention some interesting points here.   The proof is outlined in Section~\ref{sec:outlinemainthm} and takes up Sections~\ref{sec:generalizedcornercoord}--\ref{sec:nonzerodivisor}.  The goal is to show that given any edge $e$ in a locally planar triangulation, $\beta \neq 0$  implies $e \beta \neq 0$ for every  $\beta \in \Curve$.  When $\beta = \alpha_i$, representing a  single reduced multicurve class (no self-crossings or turn-backs inside a triangle, and no component is a  trivial loop or  loop around a puncture), that $e \alpha_i \neq 0$ is fairly obvious, since $ e \alpha_{i}$ is a linear combination of distinct, linearly independent resolutions.  However, it is not so obvious when $\beta = \sum_{j \in I}f_{j}\alpha_{j}$ is a  $\CC[v_{i}^{\pm}]$-linear combination of reduced multicurves $\alpha_i$.  In particular, we must understand the various ways that  resolutions of $e \alpha_i$ and $e \alpha_j$ could cancel out in $e \beta$, in order to rule out the scenario where \emph{all} the resolutions cancel each other out in $e \beta$. 
 
Our solution is to define an order on the reduced multicurves and to consider resolutions that produce ``leading terms'' according to that order.  In particular, we consider the two resolutions of $e \alpha_{i}$ without turnbacks, the so-called \emph{positive} and \emph{negative} resolutions, $P_{e}(\alpha_{i})$ and $N_{e}(\alpha_{i})$, respectively.  Our strategy is to explicitly find an $\alpha_{i}$ component of $\beta$ so that the positive resolution $P_{e}(\alpha_{i})$ becomes the leading term of $ e \beta$.    Although this strategy is very much inspired by similar results and techniques developed for the Kauffman bracket skein algebra, e.g. recently in \cite{PrzytyckiSikora19, FrohmanJKBLe19}, multiplying by arcs leads to complications not present when only looking at loops.  For example, in the Roger-Yang skein algebra, there are numerous cases where  $\alpha_{i} \neq \alpha_{j}$ but  $P_{e}(\alpha_{i})= N_{e}( \alpha_{j})$, even when $\alpha_{i}$ and $\alpha_{j}$ have the same order.  We found that most natural and reasonably simple choices of order produced such examples, so  cancellations in $e \beta$ were consistently an issue.   See Remark \ref{rem:algviewpoint} below.  

To understand when cancellations happen, we needed  a very precise description of the positive and negative resolutions, which we found tractable in the locally planar case.  In the larger non-locally planar examples that we examined, the positive and negative resolutions can simplify in very unexpected ways, and finding explicit formulas for them seemed to require ad hoc methods.  We nonetheless believe that our method should still work; namely, that it is possible to show that no edge of a triangulation zero divisor, even when the triangulation is not locally planar. 

We close this section with a few remarks. 
\begin{remark} \label{rem:algviewpoint}
In the algebraic viewpoint, a natural way to attack  Theorem \ref{thm:mainthmintro} is as follows: First, introduce a total order $\succ$ on the generating set of multicurves of $\Curve$. Next, establish a particular resolution $R$ which is a $\succ$-preserving map, i.e.,  so that $\alpha \succ \beta$ implies $R \alpha \succ R \beta$, and $R \alpha$ is the leading term in $e \alpha$. Finally, prove that for any $\beta = \sum_{j \in I}f_{j}\alpha_{j}$, if $\alpha_{0}$ is the leading term, then $R\alpha_{0}$ is the nonzero leading term of $e\beta$, thus $e\beta \ne 0$.  For example, such an algebraic scheme was successfully implemented for the Kauffman skein algebra, \cite{PrzytyckiSikora19, FrohmanJKBLe19}.  

In our context, there are a number of natural candidates for $R$. However, as we mentioned briefly above, we were unable to find a total order $\succ$ satisfying the algebraic scheme just described.  Various, different issues arose, mainly because of the existence of arc classes.  Thus we decided to use a partial order, and relied on some extra tie-breaking conditions when necessary. 
\end{remark}

\begin{remark}
Conjecture~\ref{conj:injectivity} is verified for the non-locally planar cases of the three-puncture sphere and one-punctured torus. One can  directly compute, or use the presentations of the Roger-Yang skein algebra from \cite{BKPW16Involve}, to show that no edge is a zero divisor.  
\end{remark} 

\begin{remark}
A natural way to extend Theorem \ref{thm:mainthmintro} to arbitrary $\Sigma$ is to drill new punctures and get another pointed surface $\Sigma'$ which is locally planar, and then compare $\Curve$ and $\mathscr{C}(\Sigma')$. However, the lack of functorial morphisms makes comparing $\Curve$ and $\mathscr{C}(\Sigma')$ difficult. 
\end{remark}

\begin{remark}
The proof of the three main theorems are completely independent from the choice of base ring. So one may use arbitrary commutative ring $A$ instead of $\CC$, with a replacement of $C^{\infty}(\Teich)$ by the ring of $A$-valued functions on $\Teich$. 
\end{remark}

\subsection{An extended remark about the Muller skein algebra} \label{ssec:Muller}

At about the same time as Roger and Yang used hyperbolic geometry to motivate the algebras $\Curve$ and $\Arc$ for punctured surfaces, Muller \cite{Muller16} used the theory of cluster algebras to define a different set of algebras for surfaces with marked points on its boundary.  Both theories relate to the decorated Teichm\"uller space $\Teich$, and so they are expected to be parallel in many ways.  However, explicit connections between the two points of view are still lacking.

To start, the algebra generated by lambda-length functions of edges of an ideal triangulation forms a cluster algebra $\cA_{1}(\Sigma)$, so that $\Teich$ has a cluster manifold structure (\cite{GekhtmanShapiroVainshtein05}, see also \cite{FominShapiroThurston08}). This result applies for any surface with markings.  This includes the case of a surface with punctures  (the $\Sigma$ studied in this paper), and  a surface with non-empty boundary and marked points on the boundary (which, to contrast, we denote by $\hat{\Sigma}$).   

In the latter case, Muller in \cite{Muller16} defined three related algebras related to $\mathscr{T}^{d}(\hat{\Sigma})$. Based on the work of \cite{BerensteinZelevinsky05}, he defined a quantum cluster algebra $\mathcal{A}_{q}(\hat{\Sigma})$ and  an upper quantum cluster algebra $\mathcal{U}_{q}(\hat{\Sigma})$ associated to $\hat{\Sigma}$. When $q=1$, the quantum cluster algebra corresponds to $\mathcal{A}_{1}(\hat{\Sigma})$ in the previous paragraph. In addition, Muller also defined a skein algebra $\Muller$ that is generated by framed loops and arcs which end at the marked points on the boundary components.   
Muller showed that there are natural inclusions
\begin{equation} \label{eqn:Mulleralgebras}
	\mathcal{A}_{q}(\hat{\Sigma}) \subseteq T^{-1}\Muller \subseteq \mathcal{U}_{q}(\hat{\Sigma})
\end{equation}
where $T^{-1}\Muller$ is the non-commutative localization of $\Muller$ by $T$, the set containing the boundary parallel curves. When there are  two or more marked points on each boundary component, the inclusions are equalities, so that the skein algebra is identical to the two quantum cluster algebras. Up to localization, $T^{-1}\Mullerone$ becomes isomorphic to the algebraic coordinate ring $\cO(\mathscr{T}^{d}(\hat{\Sigma}))$.

Returning to the case of punctured surfaces that we study in this paper, at least for the classical case ($q = 1$ or $h = 0$), we expect that Roger and Yang's skein algebra $\Arc$ fits into a similar framework.  If $\Sigma$ is locally planar, we obtain natural inclusions
\begin{equation}\label{eqn:RYalgebras}
	\cA_{1}(\Sigma) \subseteq \Curve = \mathscr{S}_{0}(\Sigma) \subseteq \mathcal{U}_{1}(\Sigma)
\end{equation}
(see Remark \ref{rem:clusteralgebra}).   However, except for the simplest cases, we do not know if the inclusions are equalities.  Note that in contrast to Equation \eqref{eqn:Mulleralgebras}, in \eqref{eqn:RYalgebras} there is no further localization. 

The analogy between the Roger-Yang and Muller algebras may extend also to the more algebraic geometric approach of Fock and Goncharov.  In \cite{FockGoncharov06}, they described how to understand $\cA_{1}(\hat{\Sigma})$ as a coordinate ring of a certain moduli space of decorated $\SL_{2}$ local systems. It would be really interesting to have an analogous statement for $\Sigma$, i.e., find a moduli space $B(\Sigma)$ whose coordinate ring (or  Cox ring) $\cO(B(\Sigma))$
is isomorphic to $\Curve$.

\subsection{Structure of the paper}

Section ~\ref{sec:curvealgebra} reviews the main points of \cite{RogerYang14}.  In particular, we define the curve algebra $\Curve$,  the decorated Teichmuller space $\Teich$, and the map $\Phi$.  Theorem \ref{thm:nonzerodivisorimpliesinjectivityintro} is proven in Section~\ref{sec:outline}.  Section \ref{sec:locallyplanar} is a very short section introducing locally planar surfaces.  The proof of Theorem~\ref{thm:mainthmintro} is outlined in Section \ref{sec:outlinemainthm}, and the details appear in Sections Sections~\ref{sec:generalizedcornercoord}--\ref{sec:nonzerodivisor}.  Note that in Section~\ref{sec:generalizedcornercoord}  we generalize the theory of normal curves on surfaces for normal arcs, and this may be of independent interest.  In Section \ref{sec:arcalgebra} we define Roger-Yang's skein algebra  $\Arc$, and we prove Theorem \ref{thm:nozerodivisorsintro}.

\subsection*{Acknowledgements}
The authors met while both were members at the Institute for Advanced Study in 2017-18, and we gratefully acknowledge the IAS's financial support, hospitality, and childcare throughout this collaboration.  H.M. was partially supported by the Minerva Research Foundation while he was staying at IAS. In addition, H.W. was partially funded by NSF DMS-1841221 and DMS-1906323.   We would like to also thank Tian Yang for pointing out this research problem.  \\



\section{Background: Roger and Yang's curve algebra and  decorated Teichm\"uller space} \label{sec:curvealgebra}

\subsection{Triangulation}

We begin with some notation for a surface with triangulation. 

Let $\overline{\Sigma}$ be a compact Riemann surface and $V = \{v_{i}\}$ be a finite set of points in $\overline{\Sigma}$. Then $\Sigma := \overline{\Sigma} \setminus V$ is a \emph{punctured surface} and $V$ is the set of its \emph{punctures}. 

For a triangulation $\cT = (V, E, T)$ of a compact Riemann surface $\overline{\Sigma}$, let $V$ be the set of vertices, $E$ be the set of edges, and $T$ be the set of triangles. A triangulation for a punctured surface $\Sigma = \overline{\Sigma} \setminus V$ is a triangulation of $\overline{\Sigma}$ whose vertex set is $V$. 

A \emph{corner} of $\cT$ is a pair $(v, \Delta)$ where $\Delta \in T$ is a triangle and $v \in V$ is a vertex of $\Delta$. Let $C$ be the set of all corners of $\cT$. 

\subsection{Curve classes on a punctured surface}\label{ssec:curve}

Let $\Sigma = \overline{\Sigma} \setminus V$ be a punctured surface. A \emph{loop} in $\Sigma$ is an immersion of a circle into $\overline \Sigma$ that is disjoint from $V$.  An \emph{arc} in $\Sigma$  is a map of a closed finite interval into $\overline \Sigma$ such that the interior of the interval is immersed into $\overline \Sigma \setminus V$, and the endpoints of the interval are mapped to (one or two points in) $V$. A \emph{generalized multicurve} in $\Sigma$  is a union of finitely many loops and arcs in $\Sigma$. Note that more than one component of a multicurve may have endpoints at the same puncture. If $\alpha$ and $\beta$ are two multicurves, then we denote their union, which is again a multicurve, by $\alpha \cdot \beta$ or $\alpha \beta$.  

We will consider multicurves up to regular homotopy, as defined in detail in \cite{RogerYang14, Whitney37}. We may thus assume that multicurves are in general position and although many arcs can end at a vertex, only double points occur in the interior.

Let $\cT$ be a fixed triangulation of $\Sigma$. We may further assume that our curve class $\alpha$ is general with respect to this triangulation.  By this we mean that for any edge $e \in E$, if $\alpha$ intersects the relative interior of $e$ then the intersection is transversal and if $\alpha$ ends at a vertex $v$, then any component of $\alpha$ does not tangent to any edge $e \in E$ at $v$, except the case that $e$ is a component of $\alpha$. 

A \emph{trivial loop} in $\Sigma$ is a loop bounding a disk in $\overline \Sigma$ that contains no punctures, and a \emph{puncture loop} in $\Sigma$ is a loop bounding a disk in $\overline \Sigma$ that contains exactly one puncture.  

\begin{definition}\label{def:RMCRML}
Let $\Sigma$ be a surface with a triangulation $\cT$. We say that a general multicurve $\alpha$ on $\Sigma$ is \emph{normal} if the map $\alpha: (\sqcup S^1) \sqcup (\sqcup I) \to \overline{\Sigma}$ is an embedding and there are no turn-backs inside any triangle. If all of its components are loops, we call the normal multicurve also by the term \emph{normal multiloops}. Let $\NMC$ be the set of isotopic classes of normal multicurves and $\NML$ be the set of isotopic classes of normal multiloops. 

A normal multicurve is \emph{reduced} if no component is a trivial loop or a puncture loop. By convention, the empty set $\emptyset$ is a reduced multicurve. Let $\RMC$ be the set of isotopic classes of reduced multicurves on $\Sigma$. Let $\RML$ be the subset of isotopic classes of reduced multiloops. 
\end{definition}

A reduced multicurve has no crossings and at each vertex $v$, there is at most one arc connecting $v$. 

\begin{remark}
Note that the definition of a normal multiloop is same to that of normal curve in \cite[Section 3.2]{Matveev07}. 
\end{remark}

\subsection{Decorated Teichm\"uller space}\label{ssec:teichmuller}

Suppose for the moment that $\chi(\Sigma)<0$ so that the surface $\Sigma$ with punctures $V$ admits a hyperbolic metric.  In \cite{Penner87}, Penner introduced the \emph{decorated Teichm\"uller space} $\Teich$ to be the space of pairs $(m, r)$ where $m$ is a complete hyperbolic metric on  $\Sigma$ with finite area, regarded up to an isometry that is isotopic to the identity map, and  $r: V \to \RR$ is a function which assigns a length $r(v)$ horocycle to each puncture $v \in V$. The decoration of horocycles $r$ allows us to measure the length of arcs in $\Sigma$.  In particular, given some $(m, r) \in \Teich$ and $\alpha$ an arc, the length $\ell(\alpha)$ is (up to sign) the length in the metric $m$ of the part of $\alpha$ between the horocycles described by $r$.  When $\alpha$ is a loop, its length  $\ell(\alpha)$ is the usual one determined by $m$. 

Define the  \emph{lambda-length} to be $	\lambda(\alpha) = e^{\ell(\alpha)/2}$ when $\alpha$ is an arc class,
and 	$\lambda(\alpha) = 2\cosh \frac{\ell(\alpha)}{2}$ when $\alpha$ is a loop.  The decorated Teichm\"uller space is parametrized by the lambda-length functions; more specifically we have the following theorem due to Penner.  

\begin{theorem}\cite[Theorem 3.1]{Penner87}\label{thm:homeomorphism}
Let $\{e_{1}, \cdots, e_{n}\}$ be the set of edges of a triangulation $\cT$  of $\Sigma$.  Then there is a homeomorphism
$
	\lambda : \Teich \to \RR_{> 0}^{n}
$
which maps each edge $e_{i}$ to its lambda-length $\lambda_{i} = \lambda(e_{i})$.
\end{theorem}

Note that $\Teich$ is a Zariski-dense semi-algebraic set in the complex $n$-dimensional torus $\spec \CC[\lambda_{i}^{\pm}]_{1 \le i \le n} \cong (\CC^{*})^{n}$. More precisely, $\Teich$ is the set of positive real points $(\CC^{*})^{n}(\RR_{> 0})$ in $(\CC^{*})^{n}$. There are no algebraic relations between the $\lambda_{i}$'s.

\subsection{The curve algebra $\Curve$ of loops and arcs in a punctured surface}\label{ssec:curvealgebra}

To define the curve algebra $\Curve$, we associate an indeterminate $v_i$ for each puncture in $V$, and further assume that the formal inverse $v_i^{-1}$ exists.  (Note, by a slight abuse of notation, we use $v_i$ for both a puncture and its associated indeterminate variable.)  Let $\CC[v_i^{\pm1}]$ denote that $\CC$-algebra generated by $\{v_i^{\pm 1}\}$.  

\begin{definition}\label{def:curvealgebra}
The (classical) \emph{curve algebra $\Curve$} is the $\CC[v_i^{\pm1}]$-algebra freely generated by by the generalized multicurves in $\Sigma$ modded out by the following relations
\begin{align*}
&1)
\quad
\begin{minipage}{.5in}\includegraphics[width=\textwidth]{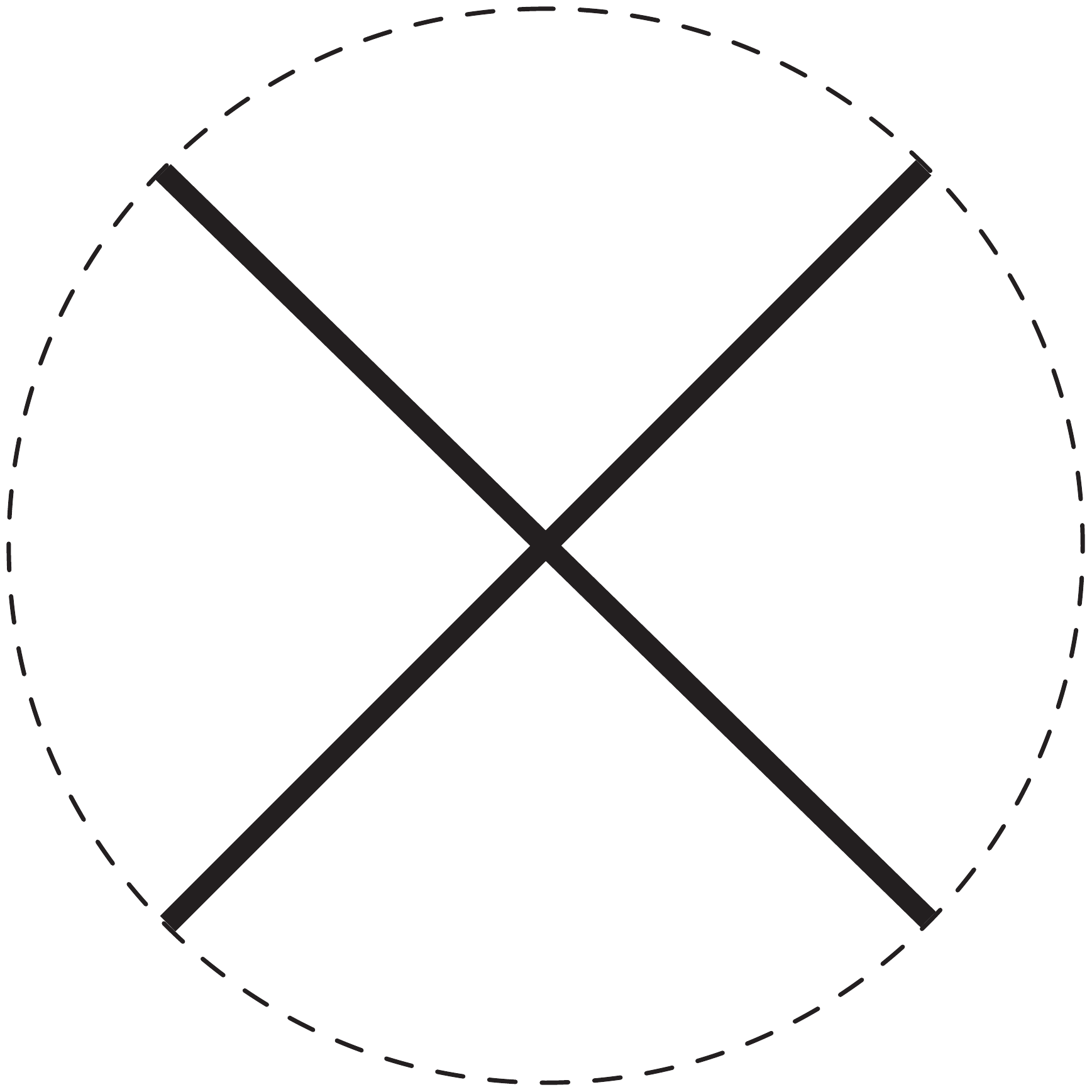}\end{minipage} 
-  \left( \begin{minipage}{.5in}\includegraphics[width=\textwidth]{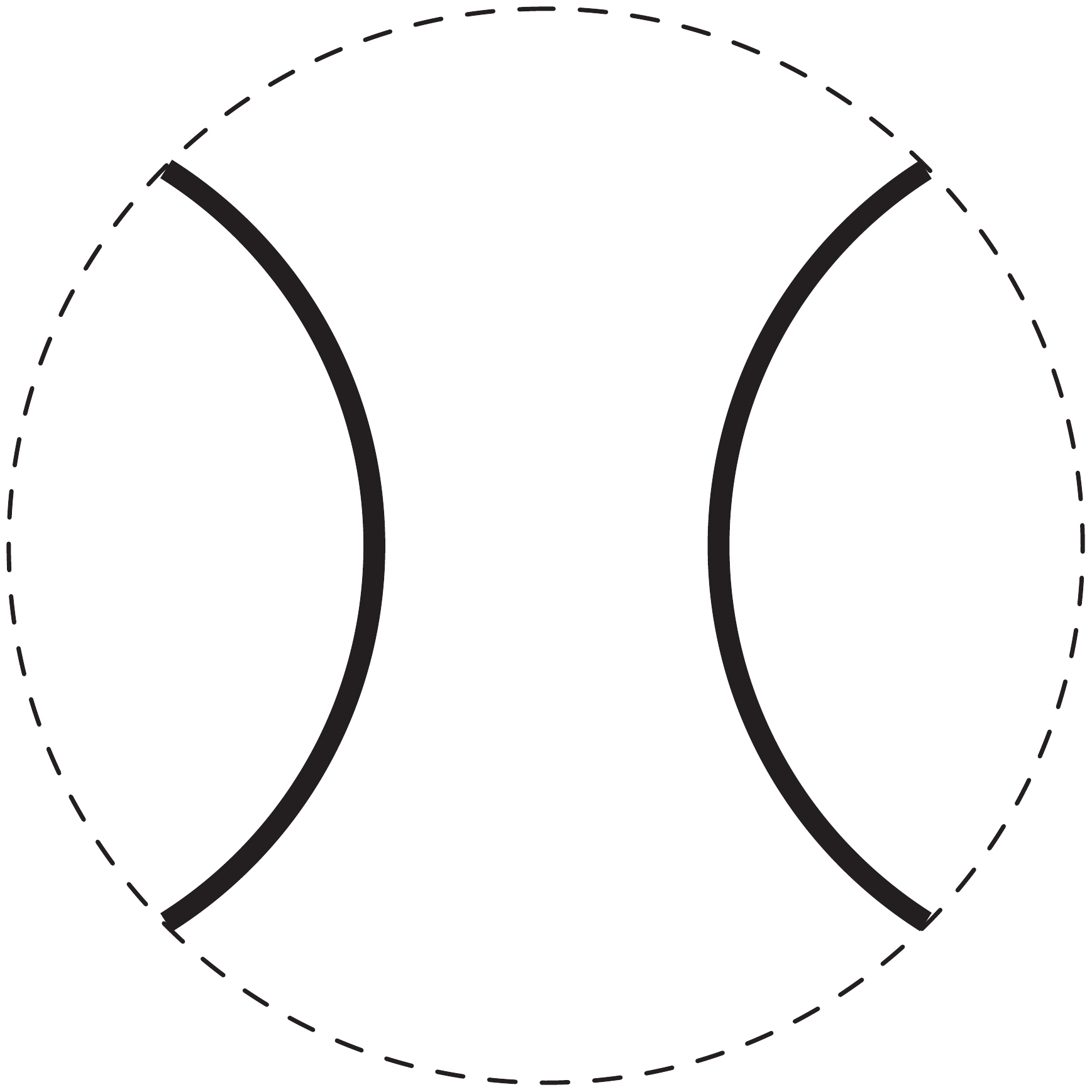}\end{minipage} 
+\begin{minipage}{.5in}\includegraphics[width=\textwidth]{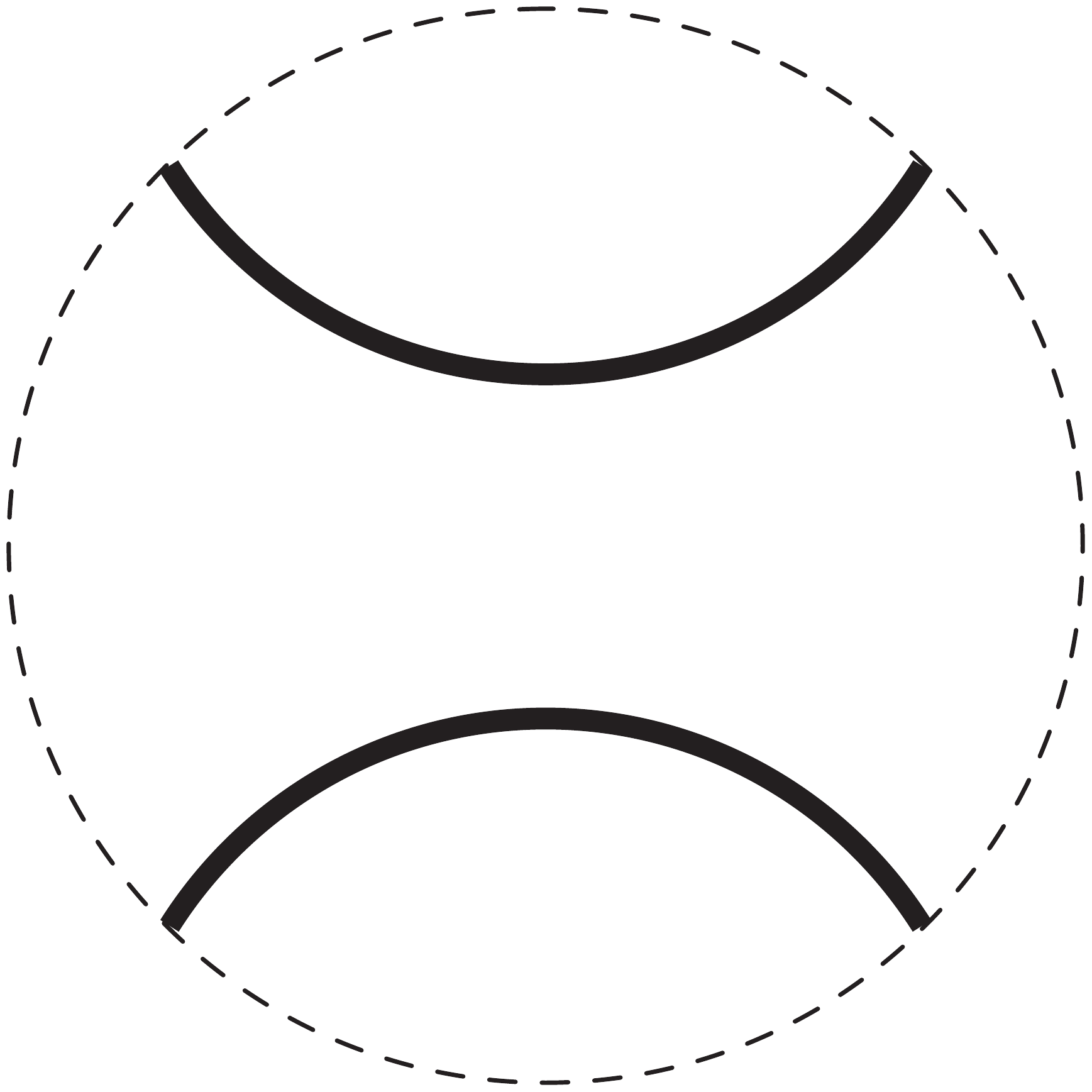}\end{minipage}  \right)\\
&2)
\quad 
v_i \begin{minipage}{.5in}\includegraphics[width=\textwidth]{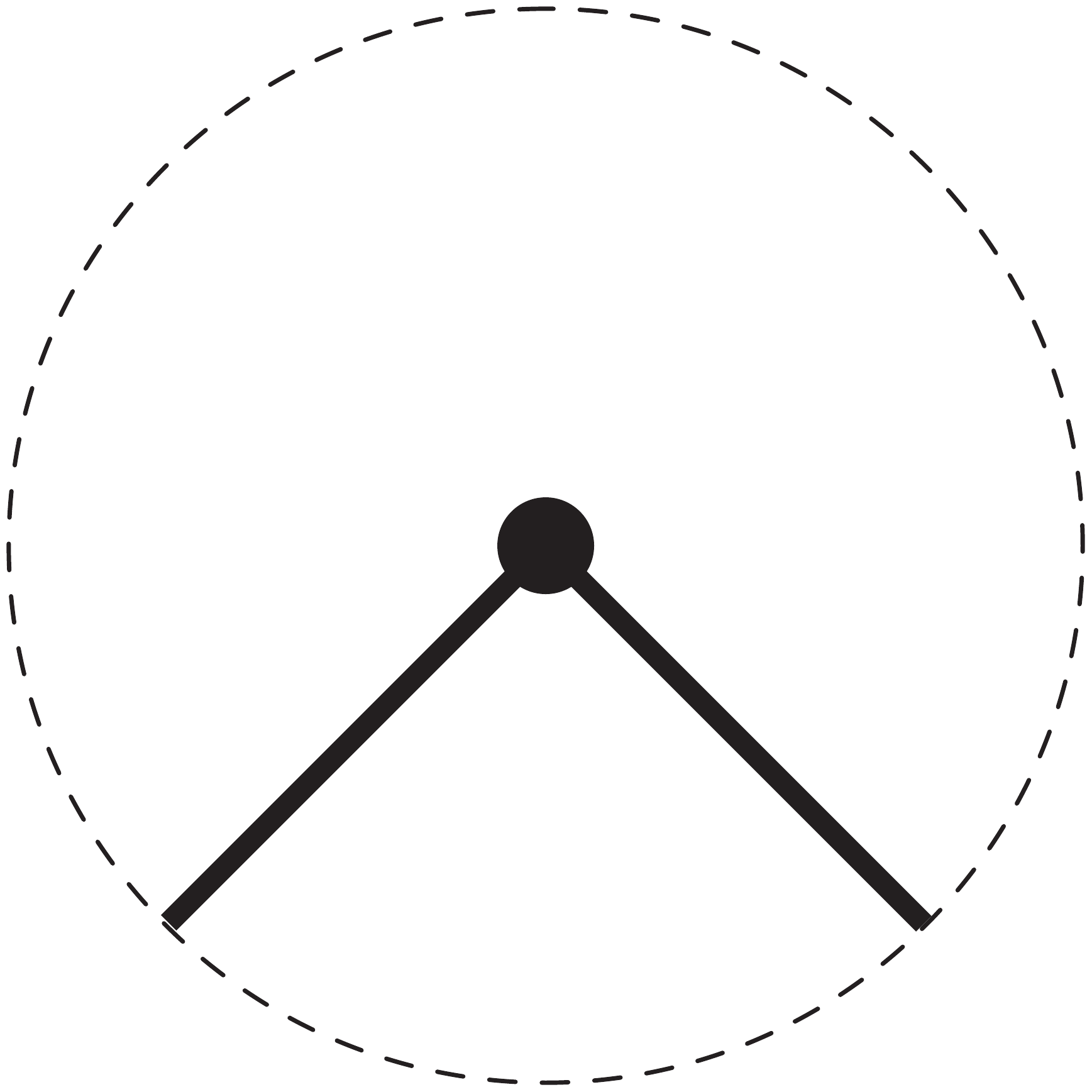}\end{minipage} 
- \left( \begin{minipage}{.5in}\includegraphics[width=\textwidth]{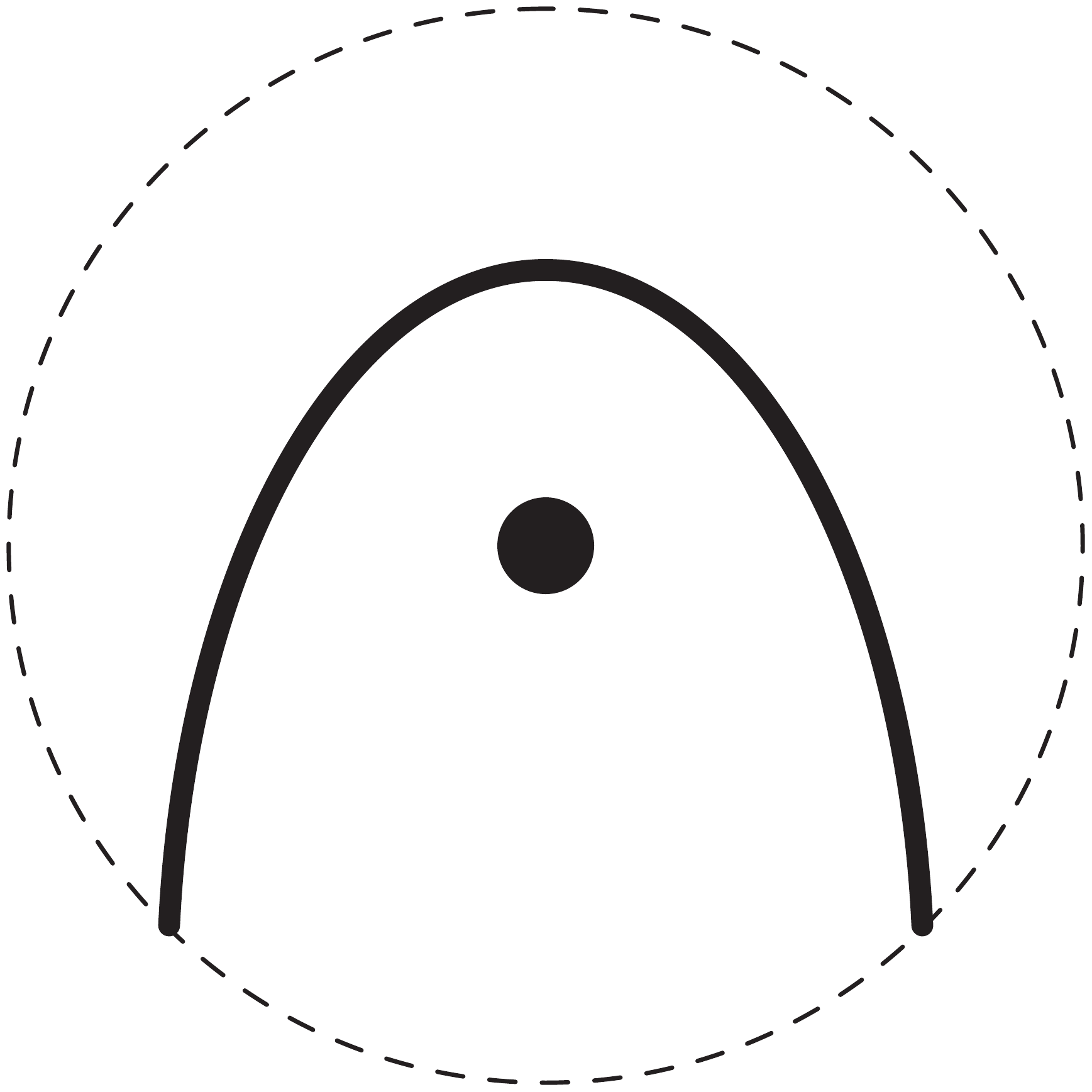}\end{minipage} 
+\begin{minipage}{.5in}\includegraphics[width=\textwidth]{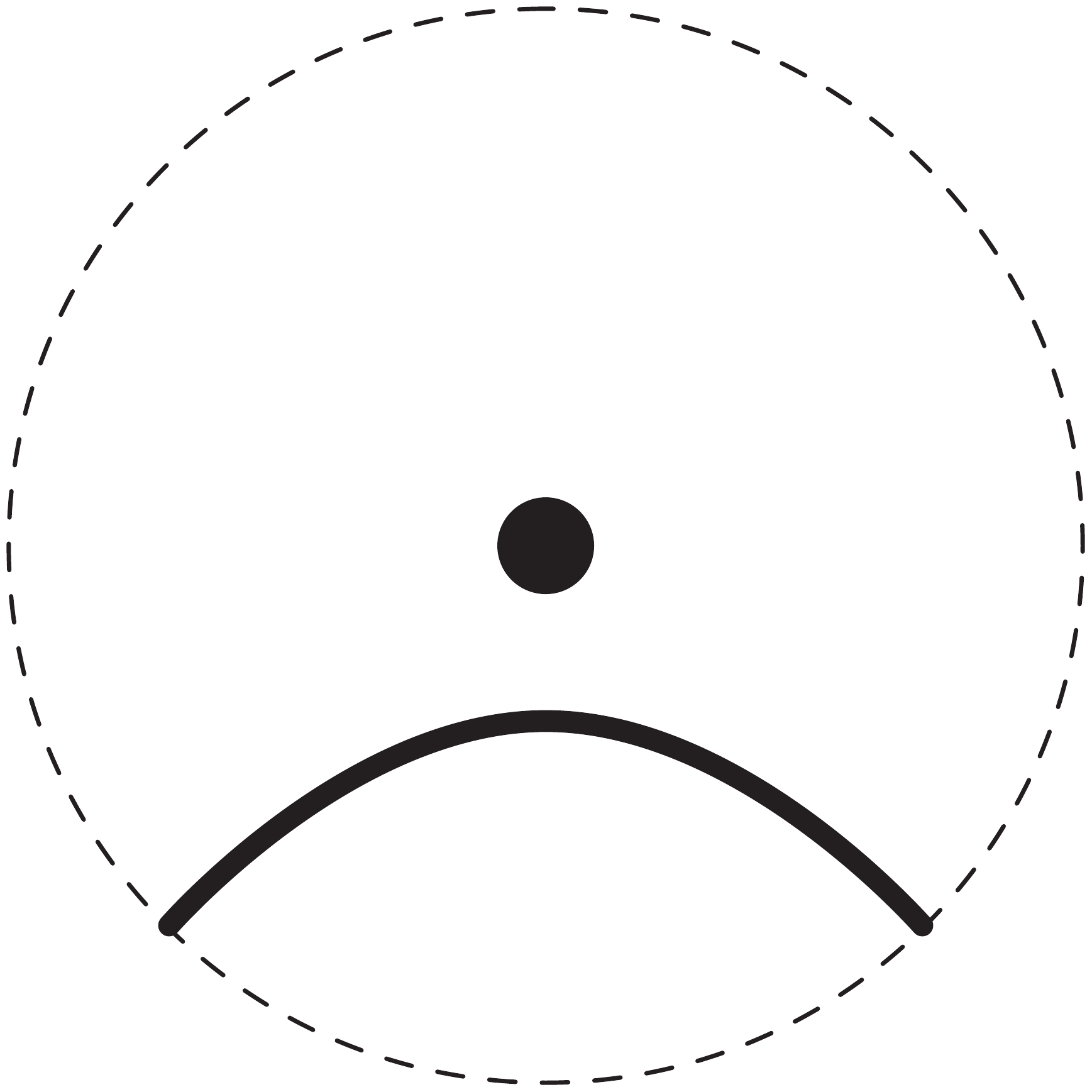}\end{minipage}  \right)\\
&3)
\quad 
\begin{minipage}{.5in}\includegraphics[width=\textwidth]{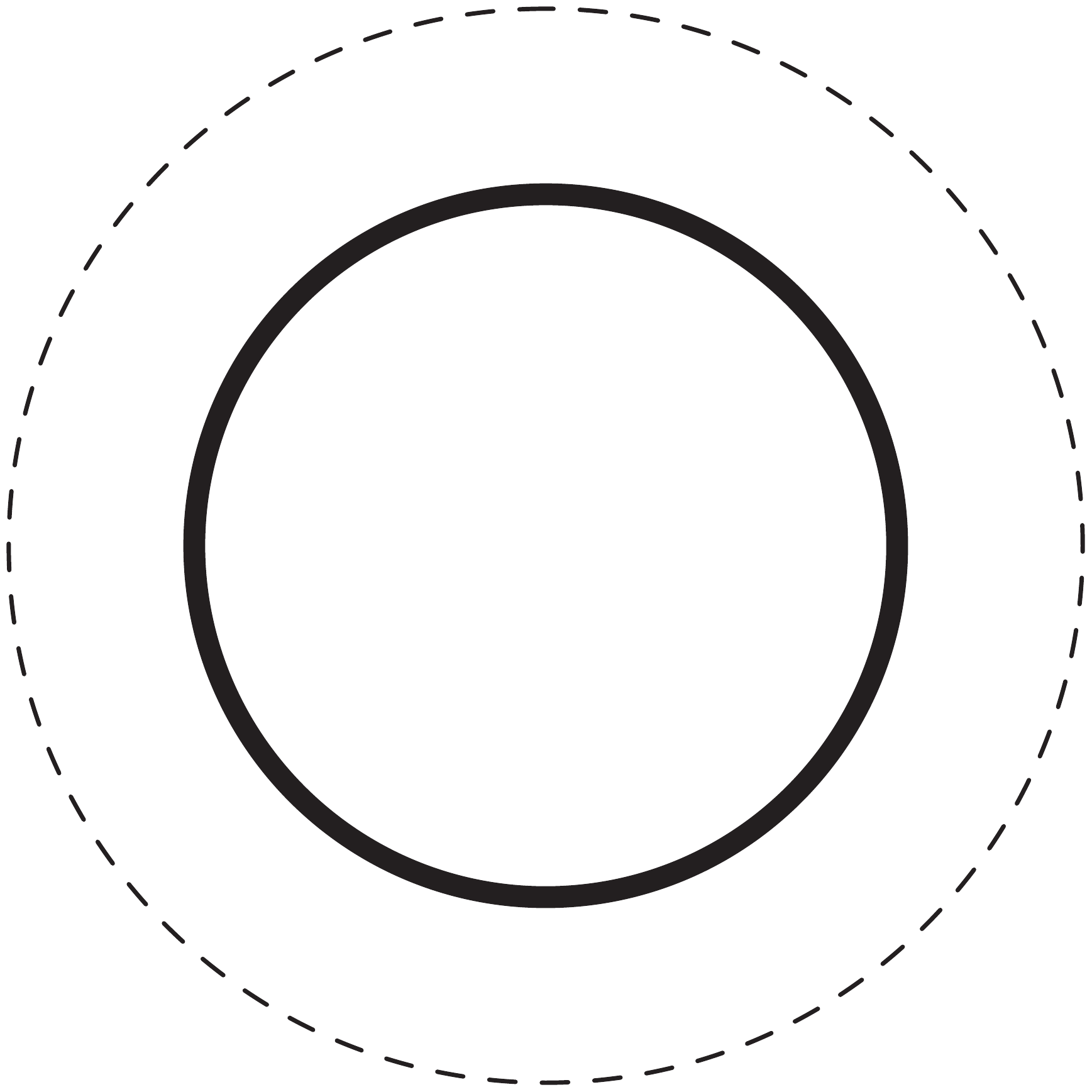} \end{minipage} 
+ 2\\
&4)
\quad 
\begin{minipage}{.5in}\includegraphics[width=\textwidth]{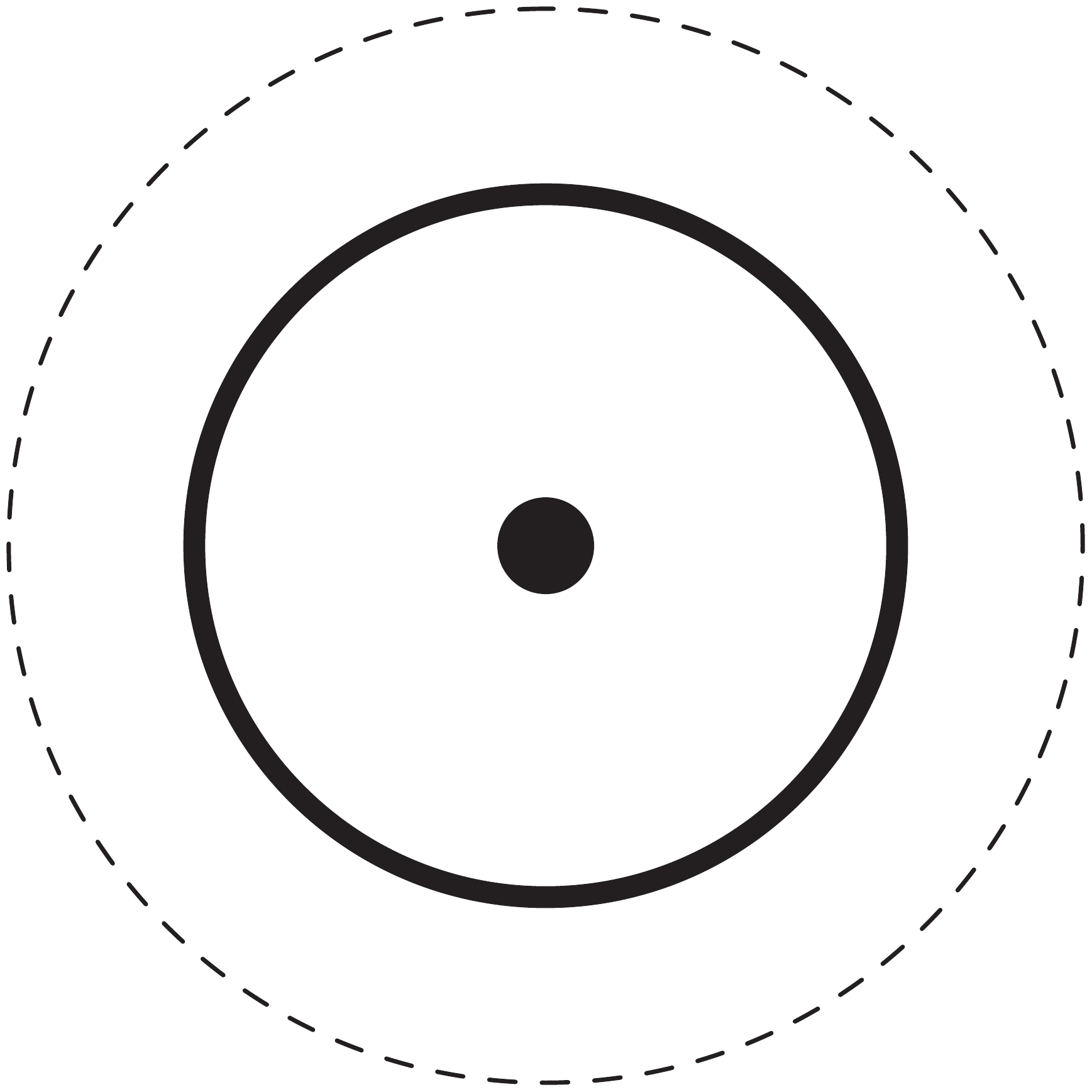} \end{minipage} 
-2\\ 
\end{align*}
where the diagrams in the relations are assumed to be identical outside of the small balls depicted.  Multiplication of elements in $\Curve$ is the one induced by taking the union of generalized curves in $\Sigma$, and the unit is the empty curve $\emptyset$. 
\end{definition}

\begin{remark}[On vector notation]
We will make the following notational convention, to use in the following proposition and throughout the rest of the paper.  
When we need to describe a tuple of objects, we use a boldface letter. For instance, for  a finite subset $\{v_{1}, v_{2}, \cdots, v_{n}\}$ of a commutative algebra $R$ and an integral vector $\bm = (m_{1}, m_{2}, \cdots, m_{n})$, $\bv^{\bm} = \prod_{i=1}^{n}v_{i}^{m_{i}}$. \\
\end{remark}

\begin{proposition}[\protect{\cite[Remark 2.5, Proposition 2.10]{RogerYang14}}]\label{prop:generatorsofcurvealgebra}
Let $W$ be the $\CC$-vector space generated by $\RMC$. Then $\Curve \cong W \otimes \CC[v_{i}^{\pm}]$. In other words, any $\beta \in \Curve$ can be written uniquely as a finite sum
\[
	\sum_{\bm \in \ZZ^{C}}\beta_{\bm}\bv^{\bm},
\]
where $\beta_{\bm}$ is a $\CC$-linear combination of elements in $\RMC$ and $\bv^{\bm}$ is a monomial in $\CC[v_{i}^{\pm}]$. And $\beta \in \Curve$ can also be written uniquely as a finite sum
\[
	\sum_{j \in I}f_{j}(v_{i}^{\pm})\alpha_{j}
\]
where $\alpha_{j} \in \RMC$ and $f_{j}(v_{i}^{\pm}) \in \CC[v_{i}^{\pm}]$. 
\end{proposition}

Although $\Curve$ is in itself interesting from the algebraic point of view, it is its relationship with hyperbolic geometry that is most intriguing.  Indeed, its definition grew out of a study of the decorated Teichm\"uller space of the punctured surface $\Sigma$, and we will describe this relationship in the next section.

\subsection{The Poisson algebra structures of $\Teich$ and $\Curve$}\label{ssec:Poissonalgebra}

Let $C^\infty(\Teich)$ be the space of $\CC$-valued smooth functions on the decorated Teichm\"uller space. Like for the usual Teichm\"uller space, $C^\infty(\Teich)$ admits a Weil-Petersson symplectic structure \cite{Penner92}.  To describe the symplectic form, one can use the \emph{lambda-length functions} $\lambda_{i}$ (or equivalently, the length functions $\ell(\alpha)$) that we introduced in Section \ref{ssec:teichmuller} . 

The Poisson bracket for the lambda-length functions was explicitly computed in \cite{Mondello09}. Although we will not need it in this paper, we include formulas of it here for the sake of completeness.
Fix a triangulation $\cT$ on $\Sigma$. For notational simplicity, assume that two ends of any edge are different vertices. For an edge $\alpha$, let $\ell(\alpha)$ be the normalized length of $\alpha$. For two edges $\alpha, \beta \in E$ which meet at $v$, let $\theta_{v}$ be the generalized angle (equal to the length of the part of the horocycle) from $\alpha$ to $\beta$ in the positive direction, and $\theta_{v}'$ be the generalized angle from $\beta$ to $\alpha$. Then the following bi-vector field
\[
	\Pi_{WP} = \frac{1}{4}\sum_{v \in V}\sum_{\stackrel{\alpha, \beta \in E}{\alpha \cap \beta = v}}\frac{\theta_{v}' - \theta_{v}}{r(v)}\frac{\partial}{\partial \ell(\alpha)}\wedge \frac{\partial}{\partial \ell(\beta)}
\]
defines the Poisson bracket on $C^{\infty}(\Teich)$.

On the other hand, Roger and Yang in \cite{RogerYang14} show that the curve algebra $\Curve$ admits a Poisson structure using a bracket  $\{\, , \, \}$ that generalizes Goldman's construction for loops on a closed surface. The \emph{generalized Goldman bracket} on $\Curve$ is a bilinear map $\{\, , \, \}: \Curve \times \Curve\to \Curve$ satisfying:
\begin{enumerate}
\item For any $v \in V$ and $\beta \in \Curve$, $\{v, \beta\} = 0$;
\item For $\alpha, \beta \in \RMC$, 
\[
	\{\alpha, \beta\} := \frac{1}{2}\sum_{p \in \alpha \cap \beta \cap \Sigma}\left((\alpha\beta)_{p}^{+} - (\alpha\beta)_{p}^{-}\right) + \frac{1}{4}\sum_{v \in \alpha \cap \beta \cap V}\frac{1}{v}\left((\alpha\beta)_{v}^{+} - (\alpha\beta)_{v}^{-}\right),
\]
where $(\alpha\beta)_{x}^{\pm}$ denotes two resolutions (called positive/negative resolutions (Definition \ref{def:PandN})) of $\alpha\beta$ at the point $x$. 
\end{enumerate}

Roger and Yang were able to show that the lambda-length functions satisfy the skein relations of the curve algebra $\Curve$, and moreover, there is a map $\Phi$ which respects the Poisson brackets of $\Curve$ and  $C^\infty(\Teich)$.

\begin{theorem}[\protect{\cite[Theorem 3.4]{RogerYang14}}] \label{thm:Phi}
For any vertex $v$, set
	$\Phi(v) = r(v)$, 
the length of the horocycle around $v$, and for any non-self intersecting arc or loop $\alpha$, set 
$	\Phi(\alpha) = \lambda(\alpha)$, the lambda-length function of $\alpha$.  

Then there exists a well-defined map $\Phi: \Curve \to C^\infty(\Teich)$ which extends linearly the map $\Phi$ before. Furthermore, $\Phi$ is a Poisson algebra homomorphism with respect to the generalized Goldman bracket on $\Curve$ and the Weil-Petersson Poisson bracket on  $C^\infty(\Teich)$. 
\end{theorem}

\section{Proof of Theorem \ref{thm:nonzerodivisorimpliesinjectivityintro}} \label{sec:outline}


The following appeared as Theorem~\ref{thm:nonzerodivisorimpliesinjectivityintro} in the introduction. 

\begin{theorem}\label{thm:nonzerodivisorimpliesinjectivity}
If $\Sigma$ has an ideal triangulation such that no edge of the triangulation is a zero divisor in $\Curve$, then $\Phi$ in \eqref{eqn:phi} is injective.
\end{theorem}

Several lemmas will build up to the proof.  Let $\cT = (V, E, T)$ be a  (not necessarily locally planar) triangulation on $\Sigma$ with edges $\{e_i\}_{1 \leq i \leq n}$.   For any vector $\mathbf{m} = (m_{i})_{1 \le i \le n} \in (\ZZ_{\geq 0})^n$, we define $\be^{\mathbf{m}}$ to be the monomial $ e_1^{m_1} e_2^{m_2} \cdots e_n^{m_n}$.   
 Let $S = \{ \be^\bm \; | \;  \bm \in (\ZZ_{\geq 0})^n \} $ be the set of all monomials with variables in $\{e_{i}\}$. Then $S$ is a multiplicative subset of $\Curve$.  
Thus, we may consider the localization $S^{-1}\Curve$, consisting of formal fractions $\frac{\beta}{\be^\bm}$ where $\beta \in \Curve$ and $\be^\bm \in S$. Let $L$ denote the associated localization map $L : \Curve \to S^{-1}\Curve$.

\begin{lemma}\label{lem:Laurantpolynomialring}
The localization $S^{-1}\Curve$ is generated by the set $\{e_{i}^{\pm}\}$ of edges and their formal inverses. 
\end{lemma}
\begin{proof}
Let $R$ be the subring of $S^{-1}\Curve$ generated by $\{e_{i}^{\pm}\}$. 

First consider the case where $\alpha$ is a generalized multicurve.  Lemma 3.18 of \cite{RogerYang14} says that if $\bm$ is the vector whose $i$-th coordinate is the intersection number of $\alpha$ and the edge $e_{i}$, then the product $ \be^{\bm} \alpha \in \Curve$ can be expressed as a polynomial with variables in  $\{e_{i}\}$.  It follows that $\alpha$, when regarded as an element of  $S^{-1}\Curve$, can be expressed as a polynomial with respect to $\{e_{i}^{\pm}\}$. Therefore $\alpha$ is in $R$. 

Next consider the case of a vertex $v$.  Let $e_{1}, e_{2}$ be two (not necessarily distinct) edges ending at $v$. Then by using the second relation in Definition \ref{def:curvealgebra}, we obtain $v e_{1}e_{2}$ as a linear combination of generalized multicurves.  By the previous case,  we have $v e_1 e_2 \in R$. Since edges $e_1, e_2 \in S$ are invertible, then $v \in R$ as well.
\end{proof}

\begin{remark}\label{rem:clusteralgebra}
Indeed, $\be^{\bm}\alpha$ in the proof of Lemma \ref{lem:Laurantpolynomialring} is a polynomial with \emph{positive} coefficients.  This was proved in \cite[Theorem 3.22]{RogerYang14}, but using lambda-length functions. In proving Theorem \ref{thm:mainthmintro}, we will show that $\Phi$ is injective for a locally planar surface, so the same formula is valid for edge classes. So $\alpha$ is a Laurant polynomial with respect to $\{e_{i}\}$ with positive coefficients. Similarly, by applying the second skein relation in Definition \ref{def:curvealgebra}, we may conclude that each $v_{i}e_{1}e_{2}$ is a positive linear combination of arc classes, so $v_{i}$ is also a Laurent polynomial with respect to $\{e_{i}\}$ with positive coefficients. Thus $\Curve$ is generated by the positive Laurent polynomials with respect to $\{e_{i}\}$.  Following on from our brief discussion in Section~\ref{ssec:Muller}, we've thus showed that $\Curve$ is a subalgebra of the upper cluster algebra $\mathcal{U}_{1}(\Sigma)$ of the cluster algebra $\cA_{1}(\Sigma)$ from $\Teich$, where the definitions of the cluster algebras are as from  \cite{GekhtmanShapiroVainshtein05}.
\end{remark}

Returning to the proof of Theorem~\ref{thm:nonzerodivisorimpliesinjectivity}, recall that $\Phi: \Curve \to C^{\infty}(\Teich)$ denotes the Poisson algebra homomorphism from \cite{RogerYang14}, which we introduced in Section \ref{ssec:Poissonalgebra}. For every edge $e_i$, let us denote $\Phi(e_i) = \lambda(e_i)$  by $\lambda_i$.  

\begin{lemma}\label{lem:factoring}
$\Phi$ factors through $\CC[\lambda_{i}^{\pm}] \subset C^{\infty}(\Teich)$. 
\end{lemma}
\begin{proof}
By  Lemma \ref{lem:Laurantpolynomialring}, for  any  linear combination of generalized curves and vertex classes $\beta \in \Curve$, there is some $\be^{\bm}  \in S$ such that $\be^{\bm} \beta$ is a polynomial with respect to $\{e_{i}\}$. Therefore $\Phi(\be^\bm \beta)  \in \CC[\lambda_{i}^{\pm}]$.  But $\Phi(\be^\bm \beta) = \Phi(\be^\bm) \Phi(\beta)$, and  $\Phi(\be^\bm) \in \CC[\lambda_{i}^{\pm}]$ too.  Thus $\Phi(\beta) \in \CC[\lambda_{i}^{\pm}]$, as desired.
\end{proof}

Notice that $\Phi$ maps every edge $e_i$ to a unit in $\CC[\lambda_{i}^{\pm}]$.   Hence $\Phi$ also maps every element of $S$ to a unit in $\CC[\lambda_{i}^{\pm}]$.  By the Universal Property of Localization,  there is a unique homomorphism $\Psi: S^{-1} \Curve \to \CC[\lambda_{i}^{\pm}]$ so that the following diagram commutes:   

\begin{equation}\label{eqn:diagram}
\xymatrix{\Curve \ar[r] \ar[d]_{L} \ar[rd]^{\Phi} 
& C^{\infty}(\Teich)\\
S^{-1}\Curve \ar[r]^{\Psi} &\CC[\lambda_{i}^{\pm}] 
\ar@{^{(}->}[u]_{id}
}
\end{equation}

\begin{lemma}\label{lem:localmapisom}
The localized map $\Psi$ is an isomorphism. 
\end{lemma}

\begin{proof}
Consider the map $\Xi : \CC[\lambda_{i}^{\pm}] \to S^{-1}\Curve$ which sends $\lambda_{i}$ to $e_{i}$.  Then $\Xi$ is an algebra homomorphism such that $\Psi \circ \Xi = \mathrm{id}_{\CC[\lambda_{i}^{\pm}]}$, implying that $\Xi$ is injective. By Lemma \ref{lem:Laurantpolynomialring}, any generator of $S^{-1}\Curve$ can be written as a Laurent polynomial with respect to $\{e_{i}\}$. Thus $\Xi$ is surjective. Hence $\Xi$ and $\Psi$ are bijective. 
\end{proof}

The proof now falls easily from the previous lemmas. 
\begin{proof}[Proof of Theorem~\ref{thm:nonzerodivisorimpliesinjectivity}]
If no $e_i$ is a zero divisor, then $S$ will not contain any zero divisors.  It follows that the localization map $L : \Curve \to S^{-1}\Curve$ is injective. Indeed, suppose that $L(\alpha) = L(\beta)$. By the definition of localization, there is $\be^\bm \in S$ such that $\be^\bm(\alpha - \beta) = 0$. But because $\be^\bm$ is not a zero-divisor, $\alpha = \beta$. By Lemma \ref{lem:localmapisom}, $\Phi$ is injective because it is a composition of injective morphisms. 
\end{proof}


\section{Locally planar surfaces} \label{sec:locallyplanar}

\begin{definition}
A \emph{locally planar triangulation} of $\Sigma$ is a triangulation  $\cT = (V, E, T)$ of  $\overline{\Sigma} $, where $\Sigma = \overline{\Sigma} \setminus V$ such that
\begin{enumerate}
\item there is no one-cycle or two-cycle of edges in $\cT$;
\item $\Sigma$ is not the three-punctured sphere.
\end{enumerate}
In that case, we also say that a surface $\Sigma$ is \emph{locally planar}.
\end{definition}

One simple class of examples of non-locally planar surfaces is one or two punctured surfaces. On the other hand, the four-punctured sphere, which has a tetrahedral triangulation, is locally planar.  It is evident that any triangulation can be refined to a locally planar one by introducing more vertices.

Locally planar triangulations $T$ satisfy the following properties:

\begin{enumerate}
\item There is no self-folded triangle in $T$. 
\item For any $v \in V$, the star $\Star(v) := \bigcup_{v \in \Delta \in T}\Delta$ has at least three triangles. 
\item For any edge $e$, the relative interior of the star $\Star(e) := \bigcup_{e \cap \Delta \ne \emptyset}\Delta$ (see Figure \ref{fig:localpicture}) is contractible. No two triangles can be identified because otherwise there must be a two-cycle connecting $v$ and $w$. Also any internal edge cannot be identified with a boundary edge. Thus the map
$
	\mathrm{int}\; \Star(e) \to \overline{\Sigma}
$
is a continuous embedding.
\end{enumerate}
 However, note that it is possible that two boundary edges of $\partial \Star(e)$ are identified in a locally planar triangulation.

\section{Outline of the proof of Theorem \ref{thm:mainthmintro}} \label{sec:outlinemainthm}

For the convenience of the reader, we here give an outline and purpose of each of Sections \ref{sec:generalizedcornercoord}--\ref{sec:nonzerodivisor} in proving the next main result of this paper.    

\begin{theorem}\label{thm:locallyplanarzerodivisor} 
If $\Sigma$ is locally planar, then no edge of a locally planar triangulation is a zero divisor in $\Curve$. In particular, $\Phi$ in \eqref{eqn:phi} is injective. 
\end{theorem}

Before proving Theorem~\ref{thm:locallyplanarzerodivisor}, in Section~\ref{sec:generalizedcornercoord} we introduce an extension of the theory of normal curves on surfaces, as in Chapter~3 of \cite{Matveev07}.  In the usual theory, normal loops on a surface correspond to integer-valued corner coordinates that satisfy a matching condition at each edge.   In the extended theory, normal multicurves (recall Definition \ref{def:RMCRML}) now correspond to integer- and half-integer-valued \emph{generalized corner coordinates} that also satisfy the same matching conditions at each edge, and some other obvious conditions. 

Let $e$ be an edge of a locally planar triangulation of $\Sigma$,  and $\gamma \in \Curve$.   We here outline a proof that $e \gamma \neq 0$.  We refer the reader to the appropriate sections, as mentioned below, for formal definitions and detailed proofs.

\begin{enumerate}
\item Since the set of reduced multicurves $\RMC$ generates $\Curve$, in Section~\ref{sec:resolutions} we focus first on the resolutions  of $e \cup \alpha$ where $\alpha \in \RMC$.  We pick out the two, the \emph{positive resolution} $P_{e}(\alpha)$ and the \emph{negative resolution} $N_{e}(\alpha)$, where all the intersections are resolved in the same direction.  We give explicit formulas for their generalized corner coordinates.  

\item In Section~\ref{sec:injectivity} we define an ordering on $\RMC$, which we call \emph{edge degree} $\deg_{e}$.  Given a finite set of reduced multicurves, we may now determine which are leading terms with respect to $\deg_{e}$, i.e. which have the highest degree.     
We apply this to the case when $\alpha \in \RMC$, and $e \alpha$ is a finite linear combination of resolutions.   Lemma~\ref{lem:PandNareleadingterms} shows that,  with respect to  $\deg_{e}$,  the resolutions $P_{e}(\alpha)$ and $N_{e}(\alpha)$ are the only possible leading terms of $e \alpha$. 

\item Next, we consider $\CC$-linear combinations of reduced multicurves, say some non-zero $\beta = \sum_{k \in I} c_{k} \alpha_{k}$ and $c_{k} \in \CC$.  Then $e \beta = \sum_{k \in I} c_{k} (e\alpha_{k}) $ is a $\CC[v_{i}^{\pm}]$-linear combination of resolutions of all the $e \alpha_{k}$.  We wish to understand when resolutions of the $e \alpha_{k}$ could cancel each other out in $e \beta$.  

\begin{enumerate}
\item Proposition \ref{prop:Pisinjective} shows that for $\alpha \in \RMC$, the maps $\alpha \mapsto P_{e}(\alpha)$ and $\alpha \mapsto N_{e}(\alpha)$ are injective.  In  other words, $P_{e}(\alpha_{i}) \neq P_{e} (\alpha_{j})$ and  $N_{e}(\alpha_{i}) \neq N_{e} (\alpha_{j})$ for $\alpha_{i} \neq \alpha_{j} \in \RMC$.  So a positive resolution cannot cancel out with another positive resolution in $e \beta$, and similarly for the negative resolutions.   

\item In Section \ref{sec:nonzerodivisor}, we encounter examples where $P_{e}(\alpha_{i}) = N_{e} (\alpha_{j})$ for $\alpha_{i} \neq \alpha_{j} \in \RMC$, as illustrated in Figures~\ref{fig:cancellationRMC1}  and \ref{fig:Pai=Naj}.   Thus there can be situations where cancellations between resolutions of different components occur in $e \beta$.
\end{enumerate}

\item Proposition~\ref{prop:multiplyingeisinjective} proves $ e \beta \neq 0$ for  $\beta = \sum_{k \in I} c_{k} \alpha_{k}$ and $c_{k} \in \CC$.  The proof gives instructions on how to  identify $ \alpha_{j}$ so that $P_{e}(\alpha_{j}) \neq N_e(\alpha_{k})$ for all $j \neq l \in I$.   So $P_{e}(\alpha_{j})$ is a resolution that does not cancel with any other resolutions in $e \beta$, and hence it is a non-zero leading term of $e \beta$.  

\item Proposition~\ref{prop:deltanozerodivisor} finishes with the most general case, for $\gamma = \sum_{k \in I}f_{k}(v_{i}^{\pm})\alpha_{k}$ with $f_{k}(v_{i}^{\pm}) \in \CC[v_{i}^{\pm}]$.  The result follows from Proposition~\ref{prop:multiplyingeisinjective} and a short algebraic argument based on Proposition~\ref{prop:generatorsofcurvealgebra}.  
\end{enumerate}

We remark that, with the exception of Proposition~\ref{prop:deltanozerodivisor}, all the statements mentioned in the outline above require the locally planarity assumption.  In particular, our analysis relies on having exact formulas for the generalized corner coordinates of the positive and negative resolutions.


\section{Generalized corner coordinates}\label{sec:generalizedcornercoord}

Let $\Sigma = \overline{\Sigma} \setminus V$ be a punctured surface with a triangulation $\cT = (V, E, T)$. The set of corners $C$ consists of  pairs $(v, \Delta)$ where $\Delta \in T$ and $v \in V$ is a vertex of $\Delta$.  Recall Definition~\ref{def:RMCRML} of a normal and reduced multicurves. In this section, we describe how to  uniquely describe them using a tuple of numbers that encodes essential geometric information.  

\subsection{Edge and corner coordinates}

The well-known theory of normal curves on surfaces applies when our normal multicurve $\alpha$ has no arc components (see for example \cite[Section 3.2]{Matveev07}).  In that case, there are two equivalent ways to describe a normal multiloop as a tuple of integers.   One way is to record the intersection numbers of $\alpha$ with the edges.  For any normal multiloop $\alpha$ and edge $e$, let $\alpha(e)$ be the minimal number of transversal intersections of $\alpha$ with $e$.   Then the \emph{edge coordinates} of $\alpha$ are $\phi_E(\alpha) =  (\alpha(e))_{e \in E}$. 

Another way to coordinatize is to use the corners. For any normal multiloop $\alpha$ and corner $c = (v, \Delta)$, let $\alpha(c)$ be the number of components of $\alpha \cap \Delta$ that connect one edge adjacent to $v$ to the other edge adjacent to $v$.  Then the \emph{corner coordinates}  are  $\phi_C(\alpha) = (\alpha(c))_{c \in C}$.  

Clearly, the two coordinates maps $\phi_E(\alpha)$ and $\phi_C(\alpha)$ are related.   Let $c_0, c_1, c_2$ be the three corners of a triangle $\Delta$, and let $e_{i}$ be the edge opposite to $c_i$.  Then for $i$ taken modulo 3, 
\begin{eqnarray}
\alpha(e_{i}) &=& \alpha(c_{i+1}) + \alpha(c_{i+2}) \label{eqn:edgefromcorners}\\ 
\alpha(c_{i}) &=& \frac{1}{2}\left(\alpha(e_{i+1}) + \alpha(e_{i+2}) - \alpha(e_{i}\right)). \label{eqn:cornerfromedges}
\end{eqnarray}
From now on, we will use corner coordinates exclusively, as they are more suitable for arcs. 

To characterize the tuples of $\ZZ^C$ which are corner coordinates of some multiloop, we have the following definition. For any corner $c \in C$, we denote the $c$-th coordinate of some  $\mathbf{w} \in \ZZ^{C}$ by $\mathbf{w}(c)$.
\begin{definition} \label{def:matchingcondition}
Let $e$ be the common edge of two triangles $\Delta_{1}$ and $\Delta_{2}$.  Let $c_{1}$ and  $d_{1}$ be the corners of $\Delta_{1}$ adjacent to $e$, and  $c_{2}$ and $d_{2}$ be the corners of $\Delta_{2}$ adjacent to $e$.  Then a tuple $\mathbf{w} \in \ZZ^{C}$  satisfies the \emph{matching condition at $e$} if 
\[	\mathbf{w}(c_{1})+\mathbf{w}(d_{1}) = \mathbf{w}(c_{2}) + \mathbf{w}(d_{2}). \]
\end{definition}
A vector $\mathbf{w} \in \ZZ^{C}$ is the corner coordinates of some $\alpha \in \NML$ if and only if $\mathbf{w}$ has all non-negative coordinates and satisfies the matching condition at every edge.

\subsection{Generalized corner coordinates}

We now extend the corner coordinate map $\phi_C$ to normal multicurves, which include both arcs and loops.   Any normal multicurve $\alpha$ can intersect a triangle $\Delta \in \cT$ in one of three ways, as illustrated in Figure \ref{fig:basictriangles}. 
\begin{figure}[!ht]
    \begin{subfigure}[b]{0.15\textwidth}
        \includegraphics[width=.9\textwidth]{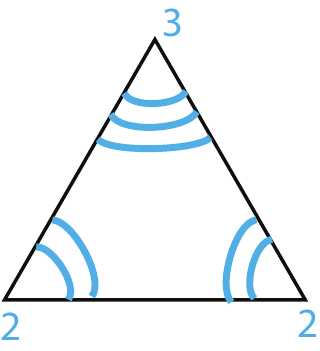}
        \caption{type I}
        \label{fig:typeI}
    \end{subfigure}
    \qquad \qquad
    \begin{subfigure}[b]{0.15\textwidth}
        \includegraphics[width=.9\textwidth]{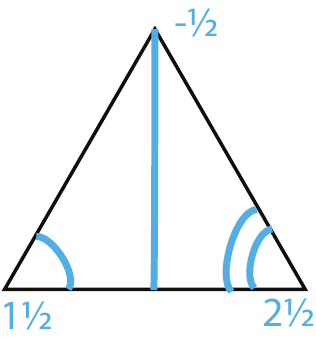}
        \caption{type II}
        \label{fig:typeII}
    \end{subfigure}
    \qquad \qquad 
    \begin{subfigure}[b]{0.15\textwidth}
        \includegraphics[width=.9\textwidth]{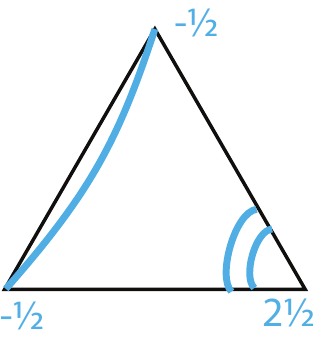}
        \caption{type III}
        \label{fig:typeIII}
    \end{subfigure}
    \caption{A normal multicurve $\alpha$ can intersect a triangle in one of three ways.  The corner coordinates of $\alpha$ are half-integers.}\label{fig:basictriangles}
\end{figure}

We say $\alpha$ in $\Delta$ is 
\begin{itemize}
\item type I,  if there is no component of $\alpha \cap \Delta$ is an arc connecting to a vertex of $\Delta$
\item type II,  if some component of $\alpha \cap \Delta $ is an arc from an edge to a vertex of $\Delta$
\item type III, if some component of $\alpha \cap \Delta $ is an arc from a vertex to another vertex of $\Delta$.  
\end{itemize}

Notice that in the case of type III, $\alpha$ contains an edge of the triangulation as a component.   We will also use the notation $\type_{\alpha}(c) \in \{ \mathrm{I}, \mathrm{II}, \mathrm{III}\} $ to describe the type of the triangle that contains the corner $c$ for $\alpha$.

\begin{definition}\label{def:cornercoordinatemap}
For each $\alpha \in \NMC$ and corner $c = (v, \Delta) \in C$, the \emph{generalized corner coordinate} $\alpha(c)$ is computed as follows.  Let  $a$ be the number of components of $\alpha \cap \Delta$ that connect the edges adjacent to $v$.  
\begin{itemize}
\item If $\type_{\alpha}(c)=\mathrm{I}$,  \, $\alpha(c) = a$. 
\item If $\type_{\alpha}(c)=\mathrm{II}$,  \,  $\alpha(c) = -\frac{1}{2}$ if there is an arc connecting $v$ and the opposite edge, and otherwise  $\alpha(c) = a+\frac{1}{2}$. 
\item If $\type_{\alpha}(c) = \mathrm{III}$, \,   $\alpha(c) = -\frac{1}{2}$ if $v$ is one end of $e$, and otherwise $\alpha(c) = a+\frac{1}{2}$.
\end{itemize}
The \emph{generalized corner coordinate} map is $\phi : \NMC \to \left(\frac{1}{2}\ZZ\right)^{C}$ such that $\phi(\alpha) = (\alpha(c))_{c \in C}$. 
\end{definition}

See Figure~\ref{fig:basictriangles} for some examples. It is evident that $\phi$ is well-defined and generalizes the corner coordinate map $\phi_C$ for normal multiloops.  For example, all of the generalized corner coordinates satisfy the matching condition  of  Definition \ref{def:matchingcondition}. Furthermore, two distinct normal multicurves must be assigned distinct generalized corner coordinates, so that the map $\phi$ is injective.   
We leave the proof of the following lemma as an exercise.  

\begin{lemma}\label{lem:matchingcondition}
For any $\alpha \in \NMC$, its generalized corner coordinates $\phi(\alpha) = (\alpha(c))_{c \in C}$ satisfy the following four conditions:
\begin{enumerate}
\item $\alpha(c) \ge -\frac{1}{2}$ for every $c \in C$, 
\item  \label{lem:matchingcondition2} either all three $\alpha(c_0), \alpha(c_1), \alpha(c_2) \in \ZZ$, or  all three $\alpha(c_0),\alpha(c_1), \alpha(c_2) \in \frac{1}{2}\ZZ$ whenever $c_0, c_1, c_2$ are the corners of a triangle,

\item $\alpha(c_i) \geq 0$ for at least one $i = 0, 1, 2$, whenever $c_0, c_1, c_2$ are the corners of a triangle,

\item $\phi(\alpha)$ satisfies the matching condition at each edge $e \in E$. 
\end{enumerate}
Conversely, if $\mathbf{w} \in \left(\frac{1}{2}\ZZ\right)^{C}$ satisfies (1)--(4) above, then there exists $\alpha \in \NMC$ such that $\phi(\alpha) = \mathbf{w}$. 
\end{lemma}

Because of the second condition in the lemma above, we will often say that a triangle of type I is \emph{integral}, whereas triangles of type II and III are \emph{fractional} or \emph{half-integral}. 

\begin{remark}
The inquisitive reader may ask whether it is possible to generalize edge coordinates instead of corner coordinates. For type I and II triangles, the corner numbers are easily deduced from edge numbers, and vice versa, by Formulas \eqref{eqn:edgefromcorners} and \eqref{eqn:cornerfromedges}. But for type III, it is unclear, at least to the authors, how to generalize the edge coordinates. 
\end{remark}

For the remainder of this article, we will focus on the reduced multicurves $\RMC$, since they generate the curve algebra $\Curve$. 
\begin{observation} \label{obs:matchingconditionRMC}
For any $\alpha \in \RMC$, its generalized corner coordinates $\phi(\alpha) = (\alpha(c))_{c \in C} \in \left(\frac{1}{2}\ZZ\right)^{C}$ satisfy the four conditions (1)--(4) of Lemma~\ref{lem:matchingcondition} and 
\begin{enumerate} \setcounter{enumi}{4}
\item at every vertex $v$, there must be at least one corner $c=(v, \Delta)$ such that $\alpha(c) \le 0$.  
\end{enumerate}
Conversely, if $\mathbf{w} \in \left(\frac{1}{2}\ZZ\right)^{C}$ is a vector satisfying (1)--(4) of Lemma~\ref{lem:matchingcondition} and (5) above, then there exists $\alpha \in \RMC$ such that $\phi(\alpha) = \mathbf{w}$. 
\end{observation}


\section{The positive and negative resolutions}\label{sec:resolutions}

Let $\alpha \in \RMC$, and let $e \in E$.  If $|\alpha \cap e| = n$, then their product $e\alpha \in \Curve$ can be decomposed into a $\CC[v_{i}^{\pm}]$-linear combination of $2^n$ crossingless multicurves.  Of those, there are two special ones.   

\subsection{Definition of $P\alpha$ and $N\alpha$} \label{def:PandN}

We say that a \emph{positive resolution of an interior crossing of $e$ and $\alpha$} is one that goes counterclockwise from $e$ to $\alpha$ as in Figure~\ref{fig:pos res int}. A \emph{positive resolution of an endpoint crossing of $e$ and $\alpha$} is one that goes clockwise from $e$ to $\alpha$ as in Figure~\ref{fig:posresend}, and when in the exceptional case where $e$ and $\alpha$ coincide (e.g. when $e$ is a component of $\alpha$),  is the one in Figure~\ref{fig:pos res e=alpha}. Let $P_e\alpha$ be the reduced multicurve isotopic to the one obtained by a positive resolution at every crossing of $e$ and  $\alpha$ and called the \emph{positive resolution} of $e\cup \alpha$ (or of $e \alpha$). For notational simplicity, we write $P\alpha$ instead of $P_e\alpha$.  
\begin{figure}[htpb]
    \centering
    \begin{subfigure}[b]{0.25\textwidth}
        \includegraphics[width=\textwidth]{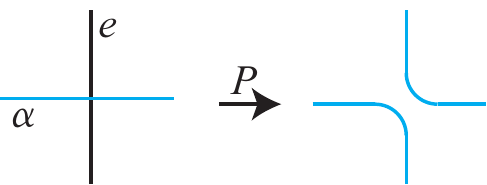}
        \caption{interior crossing}\label{fig:pos res int}
    \end{subfigure}
\quad \quad  \qquad
    \begin{subfigure}[b]{0.25\textwidth}
        \includegraphics[width=\textwidth]{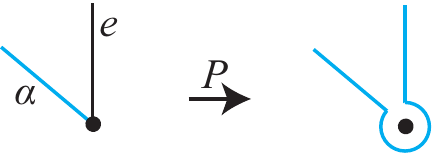}
        \caption{endpoint crossing}\label{fig:posresend}
    \end{subfigure} 
\quad \quad  \qquad
    \begin{subfigure}[b]{0.2\textwidth}
        \includegraphics[width=.9\textwidth]{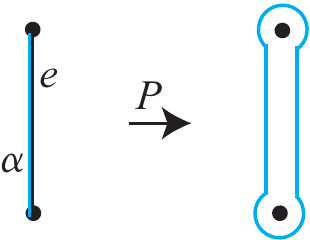}
        \caption{$P\alpha$ when $e \subseteq \alpha$}\label{fig:pos res e=alpha}
    \end{subfigure}
\caption{ $P \alpha$ is depicted in blue on the right side of the arrow.}
\end{figure}

We may similarly define the \emph{negative resolution} $N_{e}\alpha = N\alpha$ by taking the opposite resolution for every intersection. For an interior crossing of $e$ and $\alpha$, we go clockwise from $e$ to $\alpha$. For an endpoint crossing, we move counterclockwise from $e$ to $\alpha$.   And in the exceptional case where $e$ is a component of $\alpha$, so that $\alpha = e \sqcup (\alpha \setminus e)$,  then $N\alpha = -2(\alpha \setminus e)$. 

\begin{figure}[htpb]
    \centering
    \begin{subfigure}[b]{0.25\textwidth}
        \includegraphics[width=\textwidth]{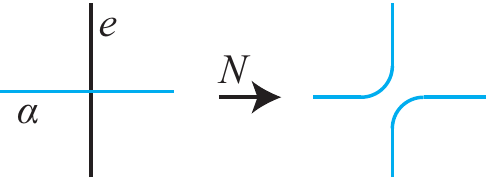}
        \caption{interior crossing}\label{fig:neg res int}
    \end{subfigure}
\quad \quad  \qquad
    \begin{subfigure}[b]{0.25\textwidth}
        \includegraphics[width=\textwidth]{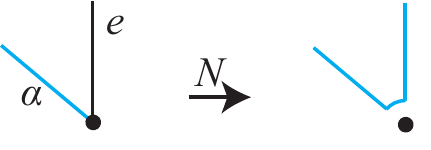}
        \caption{endpoint crossing}\label{fig:negresend}
    \end{subfigure} 
\quad \quad  \qquad
    \begin{subfigure}[b]{0.2\textwidth}
        \includegraphics[width=.9\textwidth]{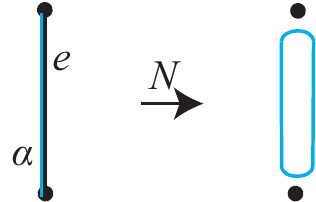}
        \caption{$N\alpha$  when $e \subseteq \alpha$}\label{fig:neg res e=alpha}
    \end{subfigure}
\caption{$N \alpha$  is depicted in blue on the right side of the arrow.}
\end{figure}

When $e$ and $\alpha$ do not intersect, there are no crossings to resolve; in that case $P \alpha = N \alpha = \alpha \sqcup e$.

By identifying $\alpha \in \RMC$ with its generalized corner coordinates $\phi(\alpha) \in \left(\frac{1}{2}\ZZ\right)^{C}$,  we may regard $P$ and $N$ as functions $\left(\frac{1}{2}\ZZ\right)^{C} \to \left(\frac{1}{2}\ZZ\right)^{C}$.  The goal of the next few sections is to write down formulas for these two resolution maps when $\Sigma$ is locally planar.


\subsection{Definition of $\RMC^{j}$} \label{sec:RMC} 
The formulas for the positive and negative resolution maps will usually require us to subdivide $\RMC$. For $j = 0, 1, 2$, define  $\RMC^{j}$ to be the set of reduced multicurves which share $j$ endpoints with $e$. Let $\Curve^{j}$ be the submodule of $\Curve$ generated by $\RMC^{j}$. 
Since the basis elements are disjoint, $\Curve = \Curve^{0} \oplus \Curve^{1} \oplus \Curve^{2}$.

When the edge $e$ has two distinct vertices $v_{0}$ and $v_{1}$, on occasion we will further refine $\RMC^{1}$ and $\Curve^{1}$ into two smaller sets, according to where $\alpha$ and $e$ intersect.  For $i = 0, 1$ let $\RMC^{1}_{v_{i}}$ be the set of reduced multicurves in $\RMC$ with an endpoint at $v_{i}$, and let $\Curve^{1}_{v_{i}}$ be the submodule of $\Curve$ generated by $\RMC^{1}_{v_{i}}$.  Clearly, $\Curve^{1} = \Curve^{1}_{v_{i}} \oplus \Curve^{1}_{v_{1-i}}$.

The next two lemmas are immediate from the definition of the curve algebra.  
\begin{lemma} \label{lem:jto2-j}  $e \cdot \Curve^{j} \subset \Curve^{2-j}$.  Moreover,  $e \cdot \Curve^{1}_{v_{i}} \subset \Curve^{1}_{v_{1-i}}$ . 
\end{lemma}

\begin{lemma} \label{lem:jto2-jPN}
 $P, N: \RMC^{j} \to \RMC^{2-j}$.   Moreover, $P, N : \RMC^{1}_{v_{i}} \to \RMC^{1}_{v_{1-i}}$. 
\end{lemma}


\subsection{ Notation for the corners of $\Star(e)$}

We assume from now on that $\Sigma$ admits a locally planar triangulation $\cT$.

Our first objective will be to describe the corner coordinates of $P \alpha$ and $N\alpha$ when $\Sigma$.  To do so, we need notation for the vertices and edges in a star neighborhood of the edge $e$.

 Let the two vertices of $e$ be denoted by $v$ and $w$.  We will think of $e$ as vertical, so that $v$ is at the top and $w$ is at the bottom. There is also an endpoint opposite $e$ to the left, and an endpoint opposite $e$ to the right, and these two and $v$ and $w$ are distinct from local planarity. Let $\Delta_{L}$ be the triangle on the left of $e$, and $\Delta_{R}$ be the triangle on the right.

Label the corners in $\Star(e)$ as in Figure \ref{fig:localpicture}. In particular, let $a_0$ be the corner at $v$ which is adjacent to and on the left of $e$.  Let $a_0, a_1, \ldots, a_s$ be the successive corners going counterclockwise around $v$, so that $a_s$ is the corner at $v$ which is adjacent to and on the right of $e$. For each $i = 0, \ldots, s$, let $a_i^L$ (resp. $a_i^R$) correspond to the corner opposite to and on the left side (resp. right side) of the triangle containing the corner $a_i$.  We say that $a_i$, $a_i^L$, $a_i^R$ for $i= 0, \ldots, s$ are the \emph{top corners},  near $v$.    Similarly, let $b_0, \ldots b_t$ be the corners going counterclockwise around $w$, starting at $e$, let $b_j^L$ and $b_j^R$ be the corners opposite $b_j$ for $j = 0, \ldots, t$, and we say that these near $w$ are the \emph{bottom corners}. 

\begin{figure}[!htpb]
   \includegraphics[width=.25\textwidth]{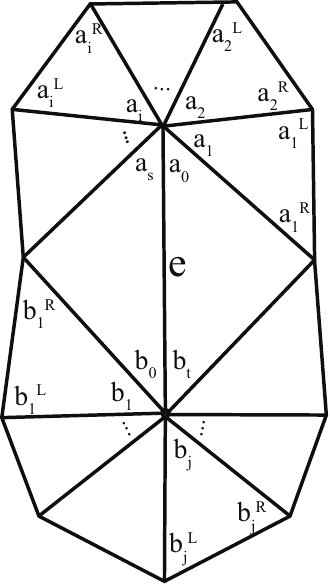}
\caption{Notation for vertices and edges in $\Star(e)$}\label{fig:localpicture}
\end{figure}

Note that by locally planarity assumption, $s, t \ge 2$ and all of the corners near $v$ and the corners near $w$ are distinct and the notation is well-defined, except for the six corners at the two triangles on either side of $e$.  The only identifications are $b_t = a_0^R$, $a_0 = b_t^L$, $a_0^L = b_t^R$, and $a_s = b_0^R$, $b_0 = a_s^L$, $b_s^R = a_0^L$.

\subsection{Corner cooordinates of $P\alpha$ and $N\alpha$ when $\Sigma$ is locally planar} \label{sec:algorithms}

By the local planarity assumption, any change in the corner coordinates between $\alpha$ and $P \alpha$ or $N\alpha$ can occur only on $\Star(e)$.  The changes are tractable, and in this section we describe the algorithms to compute them.  
Because there are no turnbacks in $P\alpha$ or $N \alpha$, the only place where a isotopy is needed to reduce them is near the vertices of $e$.  For example,  if $\alpha$ does not intersect $e$ at the vertex $v$, then $P \alpha$ and $N \alpha$ will end at $v$ and could wind around it.   See Figure~\ref{fig:RMC0Case1}, which we will discuss in more detail shortly. 

 As we shall see, the formulas obtained for $P\alpha$ and $N \alpha$ are relatively simple. Each formula is composed of two independent parts, one for each of vertices.  Each of those two parts depends only on  whether $\alpha$ ends at $v$, whether $\alpha$ ends at $w$, and on the placement of the first corner around the vertex such that the corner coordinate is $0$ or $-\frac{1}{2}$.
  
\subsubsection{Formulas when $\alpha \in \RMC^0$} \label{sec:rho0}

Figure \ref{fig:RMC0Case1} shows two sample computations for $P \alpha$ near the top vertex $v$.     
\begin{figure}[!htpb]
     \includegraphics[width=.7\textwidth]{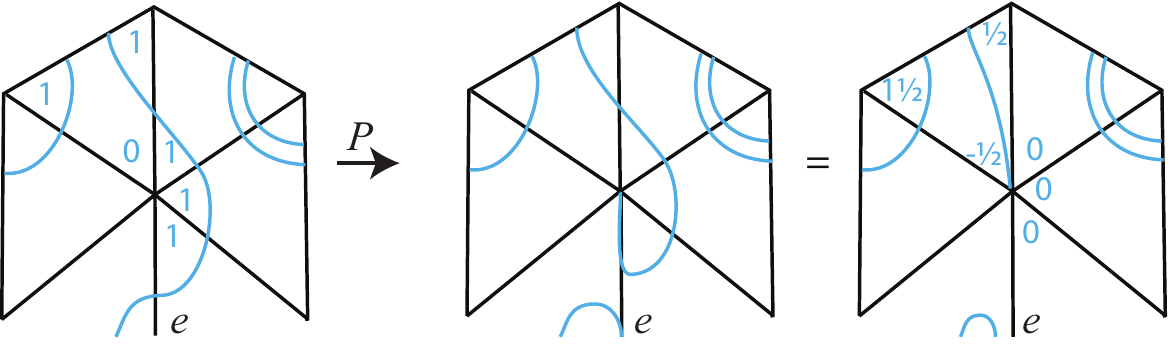}
     \includegraphics[width=.7\textwidth]{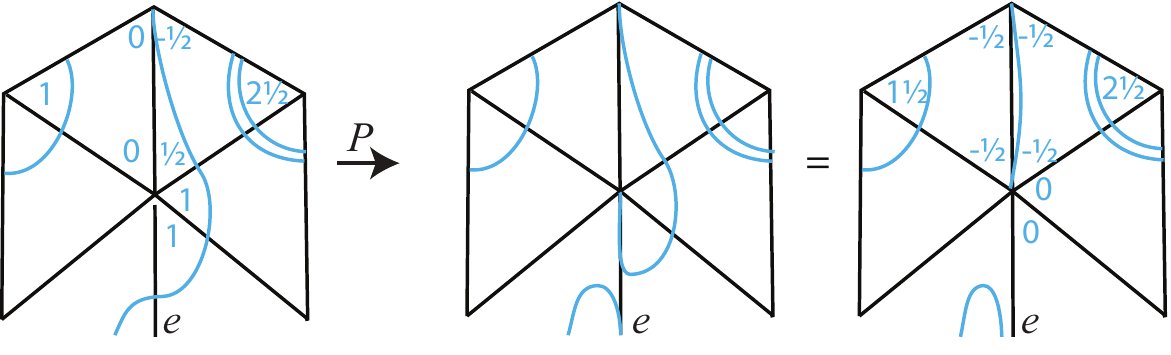}
\caption{Computations of $P \alpha$ for  $\alpha \in \RMC^{0}$} \label{fig:RMC0Case1}
\end{figure}
Notice that $P\alpha$ can be unwound around $v$, depending exactly on how far the innermost strand of $\alpha$ winds around $v$ to begin with.  Because there are no turnbacks in $P\alpha$, no other simplifications are possible, except at the bottom vertex where similar unwinding can occur.  If the reader so wishes, they may rotate each of the pictures in Figure \ref{fig:RMC0Case1} by $180^\circ$ to see examples of calculations near the bottom vertex.   Any unwinding around the top vertex does not interfere with unwinding around the bottom vertex.

\begin{algorithm}\label{alg:PRMC0}[Coordinates for $P\alpha$ when $\alpha \in \RMC^{0}$]

 Let $a_m$ be the first corner of $v$ (going counterclockwise, starting at $e$) where  $\alpha$ has corner number zero.  
Define the \emph{top change} $P_{0}^{t}$ linearly as follows: $P_{0}^{t}\alpha(a_{i}) = \alpha(a_{i}) -1$ if $i < m$, $P_{0}^{t}\alpha(a_{m}) = a_{m} -\frac{1}{2}$, $P_{0}^{t}\alpha(a_{m}^{R}) = \alpha(a_{m}^{R}) - \frac{1}{2}$, and $P_{0}^{t}\alpha(a_{m}^{L}) = \alpha(a_{m}^{L}) + \frac{1}{2}$. The other corners are same. 

Let $b_n$ be the first corner of $w$ (going counterclockwise, starting at $e$) where $\beta$ has corner number zero.  
Define the \emph{bottom change} $P_{0}^{b}$  linearly as follows: $P_{0}^{b}\alpha(b_{i}) = \alpha(b_{i}) - 1$ if $i < n$,  $P_{0}^{b}\alpha(b_{n}) = \alpha(b_{n}) - \frac{1}{2}$, $P_{0}^{b}\alpha(b_{n}^{R}) = \alpha(b_{n}^{R}) - \frac{1}{2}$, and $P_{0}^{b}\alpha(b_{n}^{L}) = \alpha(b_{n}^{L}) + \frac{1}{2}$. The other corners are same. 

When $\Sigma$ is locally planar and $\alpha \in \RMC^0$,  $P\alpha = P_{0}^{t}P_{0}^{b}\alpha = P_{0}^{b}P_{0}^{t}\alpha$. 
\end{algorithm}

For example, if $m \neq 0, s$ and $n = 0$, then because $a_s = b_0^R$, $P \alpha(a_s) = P_1^t (P_0^b \alpha(a_s) )= P_1^t(\alpha(a_s) - \frac{1}{2}) = \alpha(a_s) - \frac{1}{2}$.  However, in this case and in all others we subsequently describe, the corners $a_0$, $a_s$, $b_0$, and $b_t$ are the only ones that could possibly be affected by both the top change and the bottom change.  

 We leave verification of the algorithm as an exercise for the interested reader.   Notice that  for $\alpha \in \RMC^{0}$, the case $m = s$ and $n= 0$ is impossible.  Furthermore, the top change and bottom change do not affect each other, whence it follows that $P_{0}^{t}P_{0}^{b}= P_{0}^{b}P_{0}^{t}$.

By reflecting the figure horizontally, we get the computation for $N$.   Thus, in contrast, the formulas for $N\alpha$ go clockwise.

\begin{algorithm}\label{alg:NRMC0}[Coordinates for $N\alpha$ when $\alpha \in \RMC^{0}$]
Let $a_{m}$  be the first corner of $v$ (going clockwise, starting at $e$) where $\alpha$ has corner number zero. 
Define a \emph{top change} $N_{0}^{t}$ linearly as follows: $N_{0}^{t}\alpha(a_{i}) = \alpha(a_{i}) -1$ if $i > m$,  $N_{0}^{t}\alpha(a_{m}) = a_{m} -\frac{1}{2}$, $N_{0}^{t}\alpha(a_{m}^{R}) = \alpha(a_{m}^{R}) + \frac{1}{2}$, and $N_{0}^{t}\alpha(a_{m}^{L}) = \alpha(a_{m}^{L}) - \frac{1}{2}$. The other corners are same. 

Let $b_{n}$  be the first corner of $w$ (going clockwise, starting at $e$) where $\alpha$ has corner number zero. 
Define a \emph{bottom change} $N_{0}^{b}$ linearly as follows: $N_{0}^{b}\alpha(b_{i}) = \alpha(b_{i}) - 1$ if $i < n$,  $N_{0}^{b}\alpha(b_{n}) = \alpha(b_{n}) - \frac{1}{2}$, $N_{0}^{b}\alpha(b_{n}^{R}) = \alpha(b_{n}^{R}) + \frac{1}{2}$, and $N_{0}^{b}\alpha(b_{n}^{L}) = \alpha(b_{n}^{L}) - \frac{1}{2}$. The other corners are same. 

When $\Sigma$ is locally planar and $\alpha \in \RMC^0$,   $N\alpha = N_{0}^{t}N_{0}^{b}\alpha = N_{0}^{b}N_{0}^{t}\alpha$. 
\end{algorithm}

\subsubsection{Formulas when $\alpha \in \RMC^2$} \label{sec: rho2}
In this case, $\alpha$ intersects $e$ at both endpoints.  In Figure~\ref{fig:RMC2Case}, we've illustrated how to compute $P \alpha$  where $\alpha$ approaches the top vertex in a type II triangle, and another where $\alpha$ approaches in a type III triangle.  
\begin{figure}[!htpb]
      \includegraphics[width=.7\textwidth]{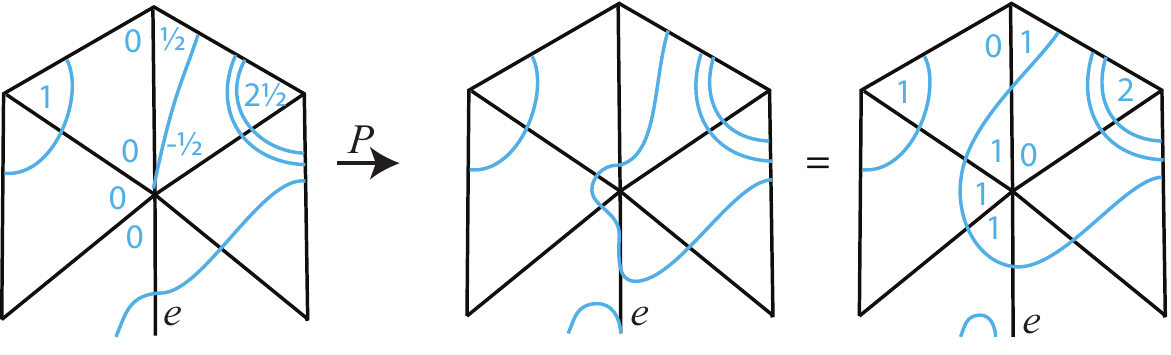}
     \includegraphics[width=.7\textwidth]{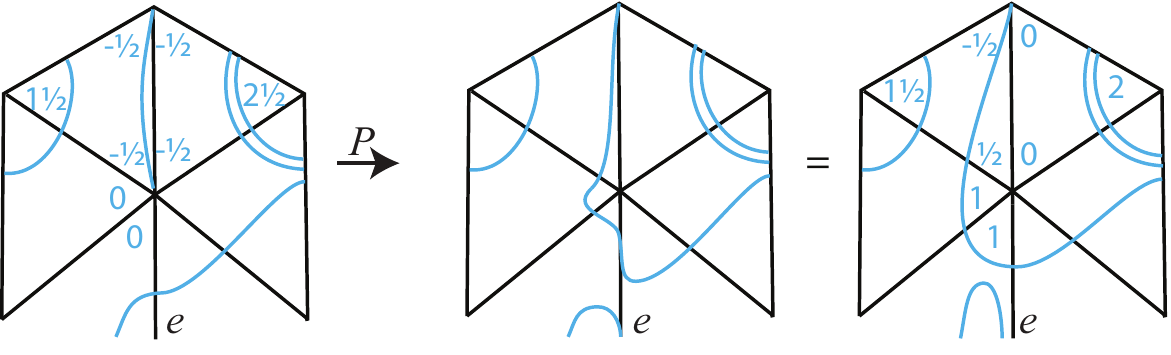}
\caption{Computing $P \alpha$ for  $\alpha \in \RMC^{2}$} \label{fig:RMC2Case}
 \label{RMC2Case3}
\end{figure}
In both examples, the $a_2$ is the first corner of $v$ where $\alpha$ has corner number $-\frac{1}{2}$.  Moreover, we see that the affected corners are $a_2$, $a_{2}^{R}$, $a_{2}^{L}$, and $a_3, \ldots, a_s$, and that the changes are by the same amount for both.  
This generalizes to the following algorithm.

\begin{algorithm}\label{alg:PRMC2}[Coordinates for $P\alpha$ when $\alpha \in \RMC^{2}$]
 Let $a_{m}$ be the first corner of $v$ (going counterclockwise, starting at $e$) where $\alpha$ has corner number $-\frac{1}{2}$.  Define a \emph{top change} $P_{1}^{t}$ linearly as follows: $P_{1}^{t}\alpha(a_{i}) = \alpha(a_{i}) +1$ if $i > m$,  $P_{1}^{t}\alpha(a_{m}) = a_{m} + \frac{1}{2}$, $P_{1}^{t}\alpha(a_{m}^{R}) = \alpha(a_{m}^{R}) - \frac{1}{2}$, and $P_{1}^{t}\alpha(a_{m}^{L}) = \alpha(a_{m}^{L}) + \frac{1}{2}$. The other corners are same. 

Let $b_{n}$ be the first corner of $w$ (going counterclockwise , starting at $e$)
where $\alpha$ has corner number $ -\frac{1}{2}$. 
Define a \emph{bottom change} $P_{1}^{b}$ linearly as follows: $P_{1}^{b}\alpha(b_{i}) = \alpha(b_{i}) + 1$ if $i > n$,  $P_{1}^{b}\alpha(b_{n}) = \alpha(b_{n}) + \frac{1}{2}$, $P_{1}^{b}\alpha(b_{n}^{R}) = \alpha(b_{n}^{R}) - \frac{1}{2}$, and $P_{1}^{b}\alpha(b_{n}^{L}) = \alpha(b_{n}^{L}) + \frac{1}{2}$. The other corners are same. 

When $\Sigma$ is locally planar,  $P\alpha = P_{1}^{t}P_{1}^{b}\alpha = P_{1}^{b}P_{1}^{t}\alpha$. 
\end{algorithm}

As before, by reflecting the picture horizontally, we obtain the algorithm for $N$.

\begin{algorithm}\label{alg:NRMC2}[Coordinates for $N\alpha$ when $\alpha \in \RMC^{2}$]
Let $a_{m}$ be the first corner of $v$ (going clockwise, starting at $e$) where $\alpha$ has corner number $-\frac{1}{2}$.  
Define a \emph{top change} $N_{1}^{t}$ linearly as follows: $N_{1}^{t}\alpha(a_{i}) = \alpha(a_{i}) +1$ if $i < m$,  $N_{1}^{t}\alpha(a_{m}) = a_{m} + \frac{1}{2}$, $N_{1}^{t}\alpha(a_{m}^{R}) = \alpha(a_{m}^{R}) + \frac{1}{2}$, and $N_{1}^{t}\alpha(a_{m}^{L}) = \alpha(a_{m}^{L}) - \frac{1}{2}$. The other corners are same. 

Let $b_{n}$ be the first corner of $v$ (going clockwise, starting at $e$) where $\alpha$ has corner number $-\frac{1}{2}$. 
Define a \emph{bottom change} $N_{1}^{b}$ linearly as follows: $N_{1}^{b}\alpha(b_{i}) = \alpha(b_{i}) + 1$ if $i < n$,  $N_{1}^{b}\alpha(b_{n}) = \alpha(b_{n}) + \frac{1}{2}$, $N_{1}^{b}\alpha(b_{n}^{R}) = \alpha(b_{n}^{R}) + \frac{1}{2}$, and $N_{1}^{b}\alpha(b_{n}^{L}) = \alpha(b_{n}^{L}) - \frac{1}{2}$. The other corners are same. 

Then $P\alpha = N_{1}^{t}N_{1}^{b}\alpha = N_{1}^{b}N_{1}^{t}\alpha$. 
\end{algorithm}


\subsubsection{Formulas when $\alpha \in \RMC^1$}
 The top change and bottom change are as described previously.  
 
\begin{algorithm}\label{alg:PNRMC1}[Coordinates for $P\alpha$ when $\alpha \in \RMC^{1}$] 
\begin{enumerate} 
\item $P\alpha = P_{1}^{b}P_{0}^{t}\alpha = P_{0}^{t}P_{1}^{b}\alpha$ if $\alpha \in \RMC^{1}_{w}$.  
\item $P\alpha = P_{0}^{b}P_{1}^{t}\alpha = P_{1}^{t}P_{0}^{b}\alpha$ if $\alpha \in \RMC^{1}_{v}$.  
\end{enumerate}
\end{algorithm}

\begin{algorithm}\label{alg:PNRMC1b}[Coordinates for $N\alpha$ when $\alpha \in \RMC^{1}$]
\begin{enumerate}
\item $N\alpha = N_{1}^{b}N_{0}^{t}\alpha = N_{0}^{t}N_{1}^{b}\alpha$ if $\alpha \in \RMC^{1}_{w}$. \item $N\alpha = N_{0}^{b}N_{1}^{t}\alpha = N_{1}^{t}N_{0}^{b}\alpha$ if $\alpha$ has an endpoint at $v$ if $\alpha \in \RMC^{1}_{v}$
\end{enumerate}
\end{algorithm}

We leave the verification to the interested readers.
\\


\section{Injectivity of the two resolutions on $\RMC$}\label{sec:injectivity}

As in the previous section, assume that $\Sigma$ is locally planar, and fix an edge $e$.  The goal of this section is to show that  for $\alpha \in \RMC$, the resolution maps $\alpha \mapsto P_e \alpha$ and $\alpha \mapsto N_e\alpha$ are injective.  We first introduce an important notion, the edge degree, that will be crucial to the proof of injectivity. We use the notation established in  Sections \ref{sec:generalizedcornercoord} and \ref{sec:resolutions}; see especially Figure \ref{fig:localpicture}. 

\subsection{Edge degree and leading terms}\label{ssec:localdegree}

\begin{definition}\label{def:localdegree}
For any $\alpha \in \RMC$, the \emph{edge degree} of $\alpha$ with respect to $e$ is defined to be
\[
	\deg_{e}(\alpha) := \frac{1}{2}\left(\alpha(a_{0}) + \alpha(a_{s}) + \alpha(b_{0}) + \alpha(b_{t})\right). 
\]
\end{definition}

If there is no chance of confusion, we will drop the subscript $e$. 
Note that $\deg(\alpha) \in \ZZ$, and by the matching condition at the edge $e$, 
\begin{equation}\label{eqn:locdeg=sum}
	\deg(\alpha) = \alpha(a_{0}) + \alpha(b_{t}) = \alpha(a_{s}) + \alpha(b_{0}) .
\end{equation}

\begin{definition} \label{def:leadingterm}
For any element  $\beta \in \Curve(\Sigma)$,  write it as a $\beta = \sum_{i \in I} c_i \alpha_{i}$, where $c_i \in \CC[v_{i}^{\pm}]$ and $\alpha_{i}$ is a reduced multicurve.  A \emph{leading term with respect to $\deg_{e}$} of $\beta$ is a component $\alpha_{j}$ with maximal edge degree, i.e., such that $\deg_{e}(\alpha_{j}) = \max \{ \deg_{e}(\alpha_{i} \;|\; i \in I \}$.  
\end{definition}

We will use $\deg(\alpha)$ to compare the different resolutions of $e \alpha$.  In particular, by the following lemma,  it allows us to pick out the positive and negative resolutions.  

\begin{lemma}\label{lem:PandNareleadingterms}
Let $\alpha \in \RMC$. When computing $e \alpha \in \Curve$, there are at most two leading terms with respect to $\deg_{e}$, and they are $P\alpha$ and $N\alpha$.  
\end{lemma}

\begin{proof}
Recall that $e\alpha$ can be written as a $\CC[v_i^{\pm}]$-linear combination of $2^n$ crossingless curves.  Each crossingless curve comes from a choice of a positive or negative resolution at each of the $n$ intersections between $e$ and $\alpha$.  The distinguished resolutions $P\alpha$ and $N\alpha$ are the only two without any turn-backs (which occur when adjacent intersections are resolved in opposite ways).  For all the others, the existence of a turnback implies that their degree  is strictly smaller. 
\end{proof}

By the formulas for $P \alpha$ and $N\alpha$ in Section \ref{sec:algorithms}, we have the following lemma.
\begin{lemma}\label{lem:Pisorderpreserving}
Let $\alpha \in \RMC^{j}$. Then $\deg(P\alpha) = \deg(N\alpha) = \deg(\alpha) + j - 1$.
\end{lemma}

\subsection{Visualizing the $P$ and $N$ maps for degree $d$ curves}\label{ssec:changelc}

For any $\alpha \in \RMC$, let 
\[
	\pi(\alpha) := (\alpha(a_{s}), \alpha(b_{t})). 
\]
Then the algorithms from Section~\ref{sec:algorithms} imply the following lemma.  There is a similar one for the negative resolution, which we omit for brevity.

\begin{lemma}\label{lem:locdegcomputation}
\begin{enumerate}
\item For $\alpha \in \RMC^{0}$, with $m = \mathrm{min}\{i\;|\; \alpha(a_{i}) = 0\}$ and $n = \mathrm{min}\{i \;|\; \alpha(b_{i}) = 0\}$,  
\[ \pi(P \alpha) = \pi(\alpha) + 
\begin{cases} 
(0,0) 			& \mbox{ if } m \ne 0, s ,  n \ne 0, t,\\
(-\tfrac{1}{2}, 0)	 & \mbox{ if } m \ne 0, s,   n = 0, 
			 \mbox{ or  } m = s,   n \ne 0, t  \\
(0, -\tfrac{1}{2})	& \mbox{ if } m \ne 0, s,  n = t, 
			 \mbox{ or  } m = 0,  n \ne 0, t,\\
(-\tfrac{1}{2}, -\tfrac{1}{2})	& \mbox{ if } m = n= 0 \mbox{ or } m= s, n= t
\end{cases}
\]

\item For $\alpha \in \RMC^{1}$ with an endpoint at the bottom vertex $w$,  $m = \mathrm{min}\{i\;|\; \alpha(a_{i}) = 0\}$ and $n = \mathrm{min}\{i \;|\; \alpha(b_{i}) = -\frac{1}{2}\}$, 
\[ \pi(P \alpha) = \pi(\alpha) + 
\begin{cases} 
(0,1) 			& \mbox{ if } m \ne 0, s ,  n \ne 0, t,\\
(-\tfrac{1}{2}, 1)	 & \mbox{ if } m \ne 0, s,   n = 0, 
			 \mbox{ or  } m = s,   n \ne 0, t  \\
(0, \tfrac{1}{2})	& \mbox{ if } m \ne 0, s,  n = t, 
			 \mbox{ or  } m = 0,  n \ne 0, t,\\
(-\tfrac{1}{2}, \tfrac{1}{2})	& \mbox{ if } m = n= 0 \mbox{ or } m= s, n= t
\end{cases}
\]

\item For $\alpha \in \RMC^{2}$,  $m = \mathrm{min}\{i\;|\; \alpha(a_{i}) = -\frac{1}{2}\}$ and $n = \mathrm{min}\{i \;|\; \alpha(b_{i}) = -\frac{1}{2}\}$, 
\[ \pi(P \alpha) = \pi(\alpha) + 
\begin{cases} 
(1,1) 			& \mbox{ if } m \ne 0, s ,  n \ne 0, t,\\
(\tfrac{1}{2}, 1)	 & \mbox{ if } m \ne 0, s,   n = 0, 
			 \mbox{ or  } m = s,   n \ne 0, t  \\
(1, \tfrac{1}{2})	& \mbox{ if } m \ne 0, s,  n = t, 
			 \mbox{ or  } m = 0,  n \ne 0, t,\\
(\tfrac{1}{2}, \tfrac{1}{2})	& \mbox{ if } m = n= 0 \mbox{ or } m= s, n= t
\end{cases}
\]
\end{enumerate}
\end{lemma}

Fix $d \in \NN$.  For reduced multicurves of degree $d$, we visualize the change of coordinates from Lemma~\ref{lem:locdegcomputation} as in Figures \ref{fig:degchangeRMC0}, \ref{fig:degchangeRMC1}, and \ref{fig:degchangeRMC2}.  The black arrows point from $\pi(\alpha)$ to $\pi(P\alpha)$, while the red arrows point from $\pi(\alpha)$ to $\pi(N\alpha)$.  

\begin{figure}[!ht]

\begin{subfigure}[b]{0.45\textwidth}
       \begin{tikzpicture}[scale=1]
	\draw[step=0.5,gray, dashed, very thin] (0, 0) grid (6.5, 6.5);
	\draw[step=0.5,gray, dashed, very thin] (-0.5, -0.5) grid (0, 6.5);
	\draw[step=0.5,gray, dashed, very thin] (0, -0.5) grid (6.5, 0);
	\draw[line width = 1pt, gray] (0, 0) -- (0, 6);
	\draw[line width = 1pt, gray] (0, 0) -- (6, 0);
	\draw[line width = 1pt, gray] (6, 0) -- (6, 6);
	\draw[line width = 1pt, gray] (0, 6) -- (6, 6);

	\draw[thick, ->] (6, 0) -- (5.5, 0);	
	\draw[thick, ->] (6, 2) -- (5.5, 2);	
	\draw[thick, ->] (6, 4) -- (5.5, 4);	
	\draw[thick, ->] (6, 6) -- (5.5, 5.5);	
	\draw[thick, ->] (4, 6) -- (4, 5.5);
	\draw[thick, ->] (2, 6) -- (2, 5.5);
	\draw[thick, ->] (0, 6) -- (0, 5.5);
	\draw[thick, ->] (0, 4) -- (-0.5, 4);
	\draw[thick, ->] (0, 2) -- (-0.5, 2);
	\draw[thick, ->] (0, 0) -- (-0.5, 0);
	\draw[thick, ->] (0, 0) -- (-0.5, -0.5);
	\draw[thick, ->] (0, 0) -- (0, -0.5);
	\draw[thick, ->] (2, 0) -- (2, -0.5);
	\draw[thick, ->] (4, 0) -- (4, -0.5);

	\fill[thick] (0, 0) circle (2pt);
	\fill[thick] (0, 2) circle (2pt);
	\fill[thick] (0, 4) circle (2pt);
	\fill[thick] (2, 0) circle (2pt);
	\fill[thick] (4, 0) circle (2pt);
	
	\fill (4, 4) circle (2pt);
	\fill (2, 4) circle (2pt);
	\fill (4, 2) circle (2pt);
	\fill (2, 2) circle (2pt);
	
	\node at (1, -0.8) {$1$};
	\node at (2, -0.8) {$2$};
	\node at (6, -0.5) {$d$};
	\node at (-0.8, 1) {$1$};
	\node at (-0.8, 2) {$2$};
	\node at (-0.5, 6) {$d$};
	
	\node at (3, -1) {$a_{s}$};
	\node at (-1, 3) {$b_{t}$};
\end{tikzpicture}
        \caption{Black arrows go from $\pi(\alpha)$ to $\pi(P\alpha)$} \label{fig:PdegchangeRMC0}
        \end{subfigure}
 \qquad
\begin{subfigure}[b]{0.45\textwidth}
\begin{tikzpicture}[scale=1]
	\draw[step=0.5,gray, dashed, very thin] (0, 0) grid (6.5, 6.5);
	\draw[step=0.5,gray, dashed, very thin] (-0.5, -0.5) grid (0, 6.5);
	\draw[step=0.5,gray, dashed, very thin] (0, -0.5) grid (6.5, 0);
	\draw[line width = 1pt, gray] (0, 0) -- (0, 6);
	\draw[line width = 1pt, gray] (0, 0) -- (6, 0);
	\draw[line width = 1pt, gray] (6, 0) -- (6, 6);
	\draw[line width = 1pt, gray] (0, 6) -- (6, 6);

	\draw[thick, ->, red] (6, 6) to[bend right=15] (5, 5);
	
	\draw[thick, ->, red] (2, 2) -- (1, 1);
	\draw[thick, ->, red] (2, 4) -- (1, 3);
	\draw[thick, ->, red] (4, 2) -- (3, 1);
	\draw[thick, ->, red] (4, 4) -- (3, 3);
	\draw[thick, ->, red] (6, 4) -- (5, 3);
	\draw[thick, ->, red] (6, 2) -- (5, 1);
	\draw[thick, ->, red] (6, 4) -- (5.5, 3);
	\draw[thick, ->, red] (6, 2) -- (5.5, 1);
	\draw[thick, ->, red] (6, 6) -- (5.5, 5);
	\draw[thick, ->, red] (6, 6) -- (5, 5.5);
	\draw[thick, ->, red] (6, 6) -- (5.5, 5.5);
	\draw[thick, ->, red] (6, 0) -- (5, -0.5);
	\draw[thick, ->, red] (4, 0) -- (3, -0.5);
	\draw[thick, ->, red] (2, 0) -- (1, -0.5);
	\draw[thick, ->, red] (-0.05, 0) -- (-0.55, -0.5);
	\draw[thick, ->, red] (0, 2) -- (-0.5, 1);
	\draw[thick, ->, red] (0, 4) -- (-0.5, 3);
	\draw[thick, ->, red] (0, 6) -- (-0.5, 5);
	\draw[thick, ->, red] (4, 6) -- (3, 5);
	\draw[thick, ->, red] (2, 6) -- (1, 5);
	\draw[thick, ->, red] (4, 6) -- (3, 5.5);
	\draw[thick, ->, red] (2, 6) -- (1, 5.5);
	
	\node at (1, -0.8) {$1$};
	\node at (2, -0.8) {$2$};
	\node at (6, -0.5) {$d$};
	\node at (-0.8, 1) {$1$};
	\node at (-0.8, 2) {$2$};
	\node at (-0.5, 6) {$d$};
	
	\node at (3, -1) {$a_{s}$};
	\node at (-1, 3) {$b_{t}$};
\end{tikzpicture}
     \caption{Red arrows go from $\pi(\alpha)$ to $\pi(N\alpha)$}
\end{subfigure}
\caption{Change of local coordinates for $\alpha \in \RMC^{0}$.}
\label{fig:degchangeRMC0}
\end{figure}
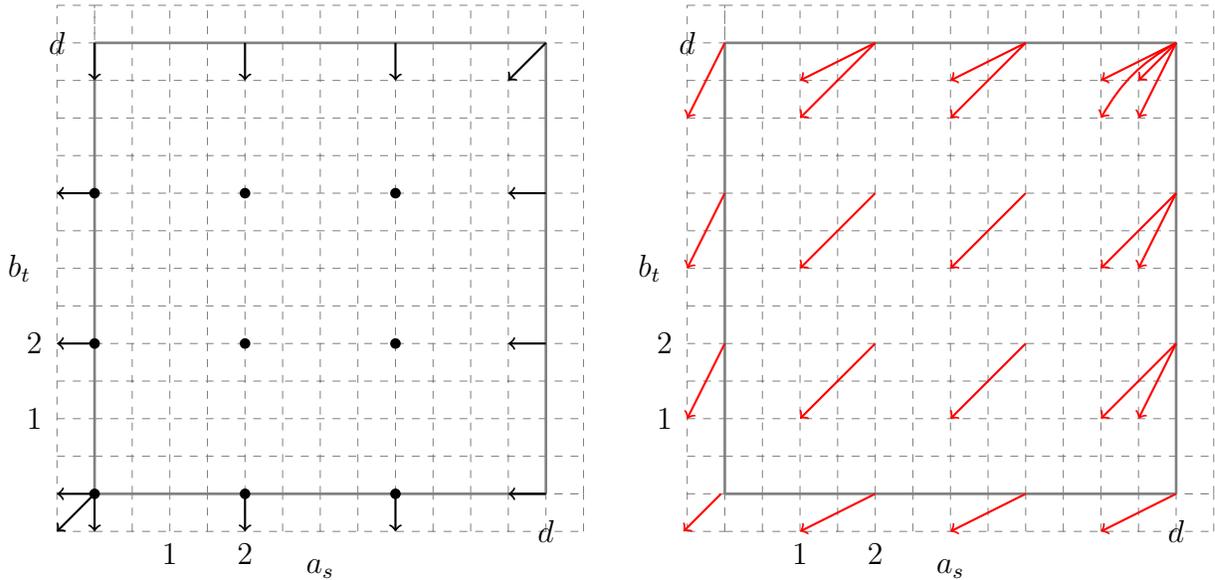

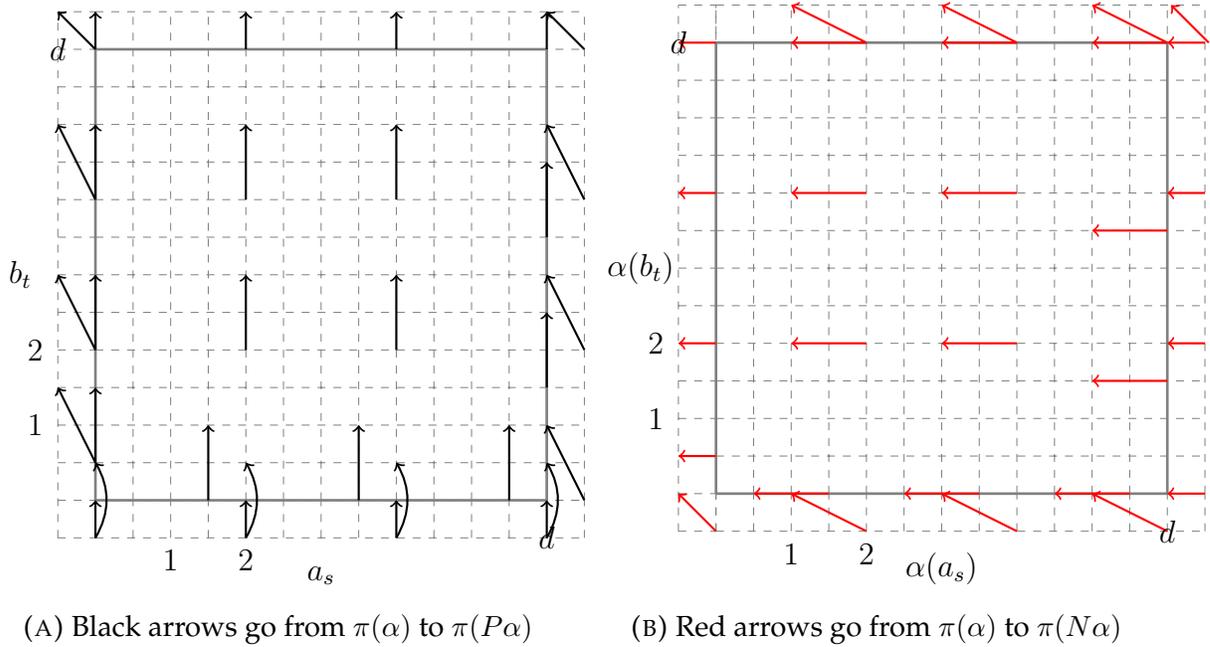
\begin{figure}[!ht]
\begin{subfigure}[b]{0.45\textwidth}
\begin{tikzpicture}[scale=1]
	\draw[step=0.5,gray, dashed, very thin] (0, 0) grid (6.5, 6.5);
	\draw[step=0.5,gray, dashed, very thin] (-0.5, -0.5) grid (0, 6.5);
	\draw[step=0.5,gray, dashed, very thin] (0, -0.5) grid (6.5, 0);
	\draw[line width = 1pt, gray] (0, 0) -- (0, 6);
	\draw[line width = 1pt, gray] (0, 0) -- (6, 0);
	\draw[line width = 1pt, gray] (6, 0) -- (6, 6);
	\draw[line width = 1pt, gray] (0, 6) -- (6, 6);

	\draw[thick, ->] (6, 1.5) -- (6, 2.5);	
	\draw[thick, ->] (6, 3.5) -- (6, 4.5);	
	\draw[thick, ->] (6, 6) -- (6, 6.5);	
	\draw[thick, ->] (4, 6) -- (4, 6.5);
	\draw[thick, ->] (2, 6) -- (2, 6.5);
	\draw[thick, ->] (0, 6) -- (0, 6.5);
	\draw[thick, ->] (0, 6) -- (-0.5, 6.5);
	\draw[thick, ->] (0, 4) -- (-0.5, 5);
	\draw[thick, ->] (0, 4) -- (0, 5);
	\draw[thick, ->] (0, 2) -- (-0.5, 3);
	\draw[thick, ->] (0, 2) -- (0, 3);

	\draw[thick, ->] (0, 0.5) -- (-0.5, 1.5);
	\draw[thick, ->] (0, 0.5) -- (0, 1.5);
	\draw[thick, ->] (0, -0.5) -- (0, 0);
	\draw[thick, ->] (0, -0.5) to[bend right] (0, 0.5);
	\draw[thick, ->] (1.5, 0) -- (1.5, 1);
	\draw[thick, ->] (2, -0.5) -- (2, 0);
	\draw[thick, ->] (2, -0.5) to[bend right] (2, 0.5);
	\draw[thick, ->] (3.5, 0) -- (3.5, 1);
	\draw[thick, ->] (4, -0.5) -- (4, 0);
	\draw[thick, ->] (4, -0.5) to[bend right] (4, 0.5);
	\draw[thick, ->] (5.5, 0) -- (5.5, 1);
	
	\draw[thick, ->] (4, 4) -- (4, 5);
	\draw[thick, ->] (2, 4) -- (2, 5);
	\draw[thick, ->] (2, 2) -- (2, 3);
	\draw[thick, ->] (4, 2) -- (4, 3);
	\draw[thick, ->] (6, -0.5) -- (6, 0);
	\draw[thick, ->] (6, -0.5) to[bend right] (6, 0.5);

	\draw[thick, ->] (6.5, 0) -- (6, 1);
	\draw[thick, ->] (6.5, 2) -- (6, 3);
	\draw[thick, ->] (6.5, 4) -- (6, 5);
	\draw[thick, ->] (6.5, 6) -- (6, 6.5);
	
	\node at (1, -0.8) {$1$};
	\node at (2, -0.8) {$2$};
	\node at (6, -0.5) {$d$};
	\node at (-0.8, 1) {$1$};
	\node at (-0.8, 2) {$2$};
	\node at (-0.5, 6) {$d$};
	
	\node at (3, -1) {$a_{s}$};
	\node at (-1, 3) {$b_{t}$};
\end{tikzpicture}
  \caption{Black arrows go from $\pi(\alpha)$ to $\pi(P\alpha)$}
\end{subfigure}
\quad
\begin{subfigure}[b]{0.45\textwidth}
\begin{tikzpicture}[scale=1]
	\draw[step=0.5,gray, dashed, very thin] (0, 0) grid (6.5, 6.5);
	\draw[step=0.5,gray, dashed, very thin] (-0.5, -0.5) grid (0, 6.5);
	\draw[step=0.5,gray, dashed, very thin] (0, -0.5) grid (6.5, 0);
	\draw[line width = 1pt, gray] (0, 0) -- (0, 6);
	\draw[line width = 1pt, gray] (0, 0) -- (6, 0);
	\draw[line width = 1pt, gray] (6, 0) -- (6, 6);
	\draw[line width = 1pt, gray] (0, 6) -- (6, 6);
	
	\draw[thick, ->, red] (6, 6) -- (5, 6);
	\draw[thick, ->, red] (6, 6) -- (5, 6.5);
	\draw[thick, ->, red] (4, 6) -- (3, 6);
	\draw[thick, ->, red] (4, 6) -- (3, 6.5);
	\draw[thick, ->, red] (2, 6) -- (1, 6);
	\draw[thick, ->, red] (2, 6) -- (1, 6.5);
	\draw[thick, ->, red] (0, 6) -- (-0.5, 6);
	\draw[thick, ->, red] (6.5, 6) -- (6, 6);
	\draw[thick, ->, red] (6.55, 6) -- (6.05, 6.5);

	\draw[thick, ->, red] (6, 3.5) -- (5, 3.5);
	\draw[thick, ->, red] (4, 4) -- (3, 4);
	\draw[thick, ->, red] (2, 4) -- (1, 4);
	\draw[thick, ->, red] (0, 4) -- (-0.5, 4);
	\draw[thick, ->, red] (6.5, 4) -- (6, 4);

	\draw[thick, ->, red] (6, 1.5) -- (5, 1.5);
	\draw[thick, ->, red] (4, 2) -- (3, 2);
	\draw[thick, ->, red] (2, 2) -- (1, 2);
	\draw[thick, ->, red] (0, 2) -- (-0.5, 2);
	\draw[thick, ->, red] (6.5, 2) -- (6, 2);

	\draw[thick, ->, red] (6.5, 0) -- (6, 0);
	\draw[thick, ->, red] (5.5, 0) -- (4.5, 0);
	\draw[thick, ->, red] (3.5, 0) -- (2.5, 0);
	\draw[thick, ->, red] (1.5, 0) -- (0.5, 0);
	\draw[thick, ->, red] (0, 0.5) -- (-0.5, 0.5);

	\draw[thick, ->, red] (6, -0.5) -- (5, 0);
	\draw[thick, ->, red] (4, -0.5) -- (3, 0);
	\draw[thick, ->, red] (2, -0.5) -- (1, 0);
	\draw[thick, ->, red] (0, -0.5) -- (-0.5, 0);
	
	\node at (1, -0.8) {$1$};
	\node at (2, -0.8) {$2$};
	\node at (6, -0.5) {$d$};
	\node at (-0.8, 1) {$1$};
	\node at (-0.8, 2) {$2$};
	\node at (-0.5, 6) {$d$};
	
	\node at (3, -1) {$\alpha(a_{s})$};
	\node at (-1, 3) {$\alpha(b_{t})$};
\end{tikzpicture}
  \caption{Red arrows go from $\pi(\alpha)$ to $\pi(N\alpha)$}
\end{subfigure}
\caption{Change of local coordinates for $\alpha \in \RMC^{1}$ with an endpoint at the bottom vertex $w$}
\label{fig:degchangeRMC1}
\end{figure}

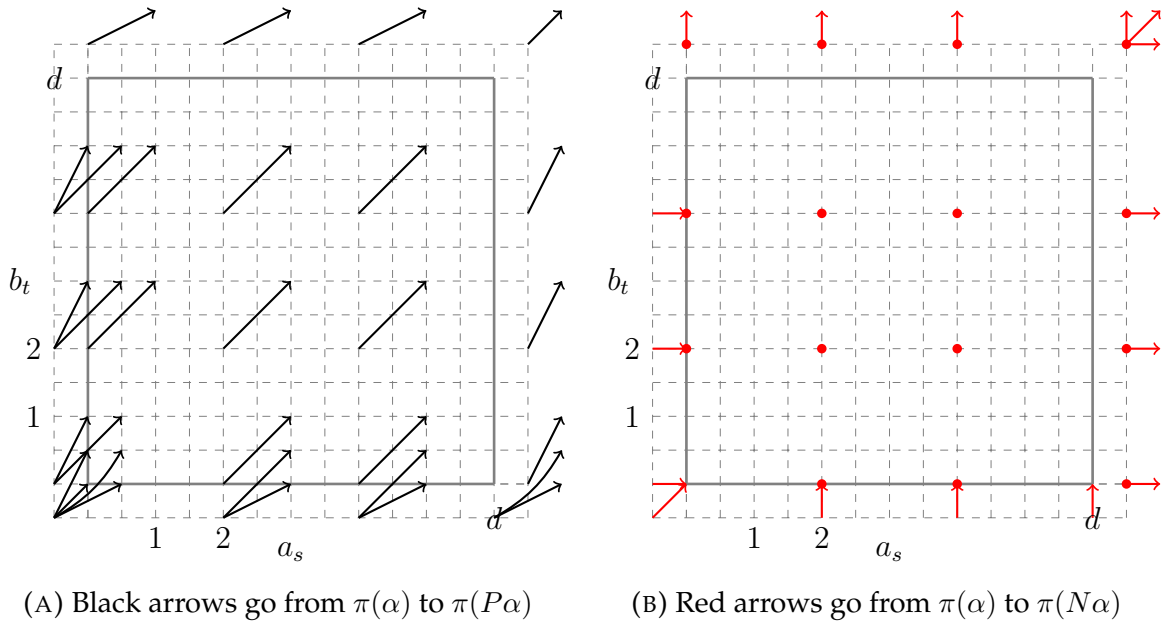
\begin{figure}[!ht]
\begin{subfigure}[b]{0.45\textwidth}
\begin{tikzpicture}[scale=0.9]
	\draw[step=0.5,gray, dashed, very thin] (0, 0) grid (6.5, 6.5);
	\draw[step=0.5,gray, dashed, very thin] (-0.5, -0.5) grid (0, 6.5);
	\draw[step=0.5,gray, dashed, very thin] (0, -0.5) grid (6.5, 0);
	\draw[line width = 1pt, gray] (0, 0) -- (0, 6);
	\draw[line width = 1pt, gray] (0, 0) -- (6, 0);
	\draw[line width = 1pt, gray] (6, 0) -- (6, 6);
	\draw[line width = 1pt, gray] (0, 6) -- (6, 6);

	\draw[thick, ->] (6.5, 6.5) -- (7, 7);
	\draw[thick, ->] (4, 6.5) -- (5, 7);
	\draw[thick, ->] (2, 6.5) -- (3, 7);
	\draw[thick, ->] (0, 6.5) -- (1, 7);

	\draw[thick, ->] (6.5, 4) -- (7, 5);
	\draw[thick, ->] (4, 4) -- (5, 5);
	\draw[thick, ->] (2, 4) -- (3, 5);
	\draw[thick, ->] (0, 4) -- (1, 5);
	\draw[thick, ->] (-0.5, 4) -- (0.5, 5);
	\draw[thick, ->] (-0.5, 4) -- (0, 5);

	\draw[thick, ->] (6.5, 2) -- (7, 3);
	\draw[thick, ->] (4, 2) -- (5, 3);
	\draw[thick, ->] (2, 2) -- (3, 3);
	\draw[thick, ->] (0, 2) -- (1, 3);
	\draw[thick, ->] (-0.5, 2) -- (0.5, 3);
	\draw[thick, ->] (-0.5, 2) -- (0, 3);

	\draw[thick, ->] (6.5, 0) -- (7, 1);
	\draw[thick, ->] (4, 0) -- (5, 1);
	\draw[thick, ->] (2, 0) -- (3, 1);
	\draw[thick, ->] (-0.5, 0) -- (0.5, 1);
	\draw[thick, ->] (-0.5, 0) -- (0, 1);

	\draw[thick, ->] (6, -0.5) -- (7, 0);
	\draw[thick, ->] (6, -0.5) to[bend right=15] (7, 0.5);
	\draw[thick, ->] (4, -0.5) -- (5, 0);
	\draw[thick, ->] (4, -0.5) -- (5, 0.5);
	\draw[thick, ->] (2, -0.5) -- (3, 0);
	\draw[thick, ->] (2, -0.5) -- (3, 0.5);
	\draw[thick, ->] (-0.5, -0.5) -- (0.5, 0);
	\draw[thick, ->] (-0.5, -0.5) to[bend right=15] (0.5, 0.5);
	\draw[thick, ->] (-0.5, -0.5) -- (0, 0.5);
	\draw[thick, ->] (-0.5, -0.5) -- (0, 0);

	\node at (1, -0.8) {$1$};
	\node at (2, -0.8) {$2$};
	\node at (6, -0.5) {$d$};
	\node at (-0.8, 1) {$1$};
	\node at (-0.8, 2) {$2$};
	\node at (-0.5, 6) {$d$};
	
	\node at (3, -1) {$a_{s}$};
	\node at (-1, 3) {$b_{t}$};
\end{tikzpicture}
  \caption{Black arrows go from $\pi(\alpha)$ to $\pi(P\alpha)$}
  \end{subfigure}
  \quad
\begin{subfigure}[b]{0.45\textwidth}  
\begin{tikzpicture}[scale=0.9]
	\draw[step=0.5,gray, dashed, very thin] (0, 0) grid (6.5, 6.5);
	\draw[step=0.5,gray, dashed, very thin] (-0.5, -0.5) grid (0, 6.5);
	\draw[step=0.5,gray, dashed, very thin] (0, -0.5) grid (6.5, 0);
	\draw[line width = 1pt, gray] (0, 0) -- (0, 6);
	\draw[line width = 1pt, gray] (0, 0) -- (6, 0);
	\draw[line width = 1pt, gray] (6, 0) -- (6, 6);
	\draw[line width = 1pt, gray] (0, 6) -- (6, 6);

	\draw[thick, ->, red] (6.5, 6.5) -- (7, 7);
	\draw[thick, ->, red] (6.5, 6.5) -- (7, 6.5);
	\draw[thick, ->, red] (6.5, 6.5) -- (6.5, 7);
	\fill[thick, red] (6.5, 6.5) circle (2pt);
	\draw[thick, ->, red] (4, 6.5) -- (4, 7);
	\fill[thick, red] (4, 6.5) circle (2pt);
	\draw[thick, ->, red] (2, 6.5) -- (2, 7);
	\fill[thick, red] (2, 6.5) circle (2pt);
	\draw[thick, ->, red] (0, 6.5) -- (0, 7);
	\fill[thick, red] (0, 6.5) circle (2pt);
	
	\draw[thick, ->, red] (6.5, 4) -- (7, 4);
	\fill[thick, red] (6.5, 4) circle (2pt);
	\fill[red] (4, 4) circle (2pt);
	\fill[red] (2, 4) circle (2pt);
	\fill[red] (0, 4) circle (2pt);
	\draw[thick, ->, red] (-0.5, 4) -- (0, 4);

	\draw[thick, ->, red] (6.5, 2) -- (7, 2);
	\fill[thick, red] (6.5, 2) circle (2pt);
	\fill[red] (4, 2) circle (2pt);
	\fill[red] (2, 2) circle (2pt);
	\fill[red] (0, 2) circle (2pt);
	\draw[thick, ->, red] (-0.5, 2) -- (0, 2);

	\draw[thick, ->, red] (6.5, 0) -- (7, 0);
	\fill[thick, red] (6.5, 0) circle (2pt);
	\fill[red] (4, 0) circle (2pt);
	\fill[red] (2, 0) circle (2pt);
	\draw[thick, ->, red] (-0.5, 0) -- (0, 0);

	\draw[thick, ->, red] (6, -0.5) -- (6, 0);
	\draw[thick, ->, red] (4, -0.5) -- (4, 0);
	\draw[thick, ->, red] (2, -0.5) -- (2, 0);
	\draw[thick, ->, red] (-0.5, -0.5) -- (0, 0);
	
	\node at (1, -0.8) {$1$};
	\node at (2, -0.8) {$2$};
	\node at (6, -0.5) {$d$};
	\node at (-0.8, 1) {$1$};
	\node at (-0.8, 2) {$2$};
	\node at (-0.5, 6) {$d$};
	
	\node at (3, -1) {$a_{s}$};
	\node at (-1, 3) {$b_{t}$};
\end{tikzpicture}
 \caption{Red arrows go from $\pi(\alpha)$ to $\pi(N\alpha)$}
\end{subfigure}
\caption{Change of local coordinates for $\RMC^{2}$.}
\label{fig:degchangeRMC2}
\end{figure}

First,  observe that each diagram is supported inside the  square $\left([-\frac{1}{2},d +\frac{1}{2}] \cap \frac{1}{2}\ZZ\right)^{2}$.  This is because $\alpha(a_{s}), \alpha(b_{t}) \geq -\frac{1}{2}$, and $d = \alpha(a_{s}) +\alpha(b_{t}) $.  

Moreover, some points inside the square $\left([-\frac{1}{2}, d+\frac{1}{2}] \cap \frac{1}{2}\ZZ\right)^{2}$ do not correspond to a reduced multicurve.  In the $\RMC^{0}$ case depicted in Figure \ref{fig:degchangeRMC0}, $\alpha$ cannot have negative corner coordintes.  So there are no arrows based along the $\alpha(a_{s}) = -\frac{1}{2}$ and $\alpha(b_{t})= - \frac{1}{2}$ lines.  In the  $\RMC^{1}$ case depicted in Figure \ref{fig:degchangeRMC1}, we assume that $\alpha$ has an endpoint at the bottom vertex $w$, and thus there cannot be arrows based along the $\alpha(a_{s}) =-\frac{1}{2}$ line.   

Lastly, notice that when one of the two coordinates is at most zero or at least~$d$ (that is, along the edges of the square in Figures \ref{fig:degchangeRMC0}, \ref{fig:degchangeRMC1}, and \ref{fig:degchangeRMC2}), there can be multiple arrows based at $\pi(\alpha)$.  In particular, the change of corner coordinates at those extremal points does not depend only on $\pi(\alpha)$, but also on the actual curve class $\alpha$.  Figure \ref{fig:localdegreechange} illustrates examples where $\pi(\alpha) = \pi(\alpha')$, but $\pi(P \alpha) \neq \pi(P\alpha')$.  

\begin{figure}[!ht]
\includegraphics[width=0.57\textwidth]{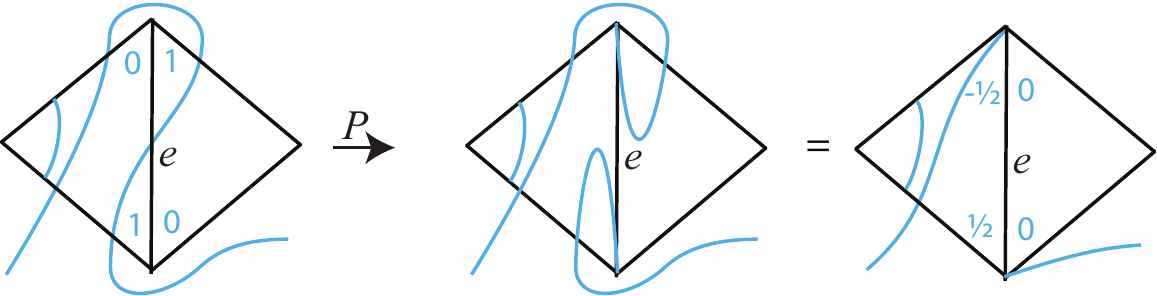}
\smallskip

\includegraphics[width=0.57\textwidth]{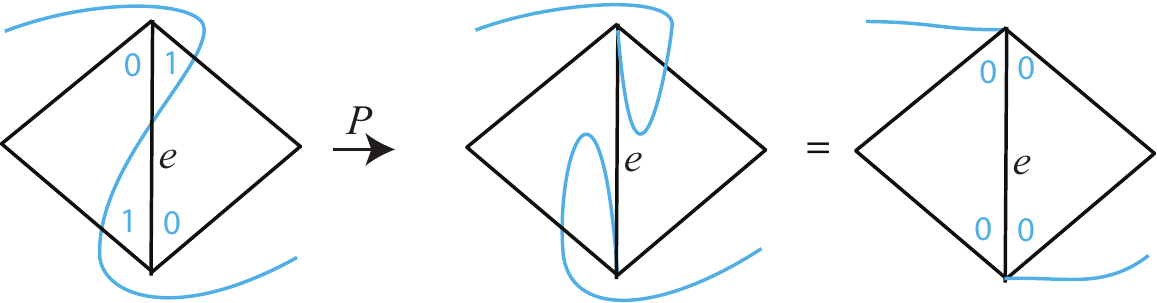}
\caption{edge degree can depend on the curve class.  Pictured are examples such that $\pi(\alpha) = \pi(\alpha')$ but $\pi(P\alpha) \neq \pi(P\alpha')$}\label{fig:localdegreechange}
\end{figure}

\subsection{Positive and Negative Resolutions are Injective}\label{ssec:positiveresolution}

We are now ready to prove the main result of this section. 
\begin{proposition}\label{prop:Pisinjective}
When $\Sigma$ is locally planar, the positive resolution map $P : \RMC^{j} \to \RMC^{2-j}$ and the negative resolution map $N : \RMC^{j} \to \RMC^{2-j}$ are injective for $j = 0, 1, 2$. 
\end{proposition}

\begin{proof}
We prove the case of $P$ only, as the proof for $N$ is identical. The proofs for $j = 0, 1, 2$ are slightly different. 

Let  $\alpha, \alpha'$ such that $P\alpha = P\alpha'$. The positive resolution map $P$ affects only $\Star(e)$.  Indeed, $P \alpha$ is completely determined by the coordinates of $\alpha$ at the corners around $v$ and $w$,which are $a_0, \ldots, a_s$ and $b_0, \ldots, b_t$ respectively. So $\alpha$ and $\alpha'$ must agree at all corners, if they agree at $a_{0}, a_{1}, \cdots, a_{s}, b_{0}, b_{1}, \cdots, b_{t}$.

\noindent \underline{Case $j = 0$, with $P: \RMC^{0} \to \RMC^{2}$} 

Let $\alpha, \alpha' \in \RMC^{0}$. We start with the four corners $a_0$, $ a_s$, $b_0$, and $ b_t$ nearest $e$.    If $\deg(\alpha) \neq \deg(\alpha')$, then Lemma \ref{lem:Pisorderpreserving} implies $\deg(P\alpha) \neq \deg(P\alpha')$, so that $P\alpha \ne P\alpha'$.  Thus let $d = \deg(\alpha) = \deg(\alpha')$.

Suppose that $\alpha(a_{s}) \neq \alpha'(a_{s})$, and without loss of generality say $\alpha(a_{s}) > \alpha'(a_{s})$.  For the corner $a_{s}$ in the $\RMC^{0}$ case, $P$ decreases the corner coordinate at $a_{s}$ by $0$ or $\frac{1}{2}$ (Figure \ref{fig:degchangeRMC0}).   From $P\alpha = P\alpha'$, it follows that $P$ decreases $\alpha$ at $a_s$ by $\frac{1}{2}$, but leaves  $\alpha'$ at $a_s$ unchanged.  

Because $P$ decreases $\alpha$ at $a_s$ by $\frac{1}{2}$, one of two scenarios are possible. Either the decrease is caused by the top change when  $m = s$, or it is caused by the bottom change when $n = 0 $ and  $b_0^{R} = a_s$.  In both scenarios, $\alpha$ is zero at one of $a_{s}$ or $b_{0}$.  So $\alpha$ is integral, and $P \alpha$ is half-integral, in the triangle containing those two corners.  
On the other hand, because $P\alpha = P \alpha'$ and $P$ leaves the corner number of $\alpha'$ at $a_s$ unchanged, $\alpha'$ must be half-integral in the triangle containing $a_{s}$ and $b_{0}$.  Type III is ruled out, since $\alpha' \in \RMC^{0}$ means $\alpha'$ cannot have an endpoint at either $v$ or $w$.     Thus $\alpha'$ is type II.  After applying positive resolutions, we see that $P \alpha$ must be type II, and $P \alpha'$ type III in that triangle.  Again this contradicts $P \alpha = P \alpha'$. We thus deduce that $\alpha(a_{s}) = \alpha'(a_{s})$.

 We take a moment to remark that the first part of our case analysis above can be seen visually using Figure \ref{fig:PdegchangeRMC0}.  Each arrow in the figure depicts how the positive resolution map $P$ affects the corner coordinates at $a_s$ and $b_t$.   Since  $P$ decreases $\alpha$ at $a_s$ by $\frac{1}{2}$, the $x$-coordinates of the arrows in Figure \ref{fig:PdegchangeRMC0} are either the same or go down by $\frac{1}{2}$.   Furthermore, identifying $\alpha$ and $\alpha'$ satisfying $\alpha(a_{s}) > \alpha'(a_{s})$ and $P \alpha(a_{s}) = P\alpha'(a_{s})$ corresponds to finding pairs of arrows which start at different $x$-coordinates and land at the same spot.  Such pairs of arrows occur only on the far right of the figure,  where $\alpha(a_{s}) = d$, $\alpha(b_{0}) = 0$, and  $\alpha'(a_{s}) = d- \frac{1}{2}$, $\alpha'(b_{0}) = \frac{1}{2}$.  The proof can be finished as before, by analyzing triangle type.

A nearly identical argument  proves that $\alpha(b_{t}) = \alpha'(b_{t})$. And because $\deg(\alpha) = \deg(\alpha')$, the identity in Equation~\ref{eqn:locdeg=sum} implies that  $\alpha$ and $\alpha'$ agree on all four corners.  

We now consider the remaining corners, beginning with the ones around $v$.   Suppose that $\alpha$ and $\alpha'$ disagree at some corner around  $v$ which is not $a_0$ or $a_s$.  Let  $k$ be the smallest index such that $\alpha(a_{k}) \ne \alpha'(a_{k})$, and without loss of generality, say $\alpha(a_{k}) > \alpha'(a_{k})$.  Let $m$ and $m'$ be as in Algorithm \ref{alg:PRMC0}; namely, let $m := \min\{i\;|\; \alpha(a_{i}) = 0\}$ and $m' := \min\{i \;|\; \alpha'(a_{i}) = 0\}$.  Taking the  positive resolution  $P$ can cause the corner coordinate at $a_k$ to decrease by $0, \frac{1}{2}$ or $1$, and so there are three cases to consider.

In the first case,  $P$ decreases $\alpha$ at $a_{k}$ by 1, and $P$ decreases $\alpha'$ at $a_{k}$ by $\frac{1}{2}$.  This implies (respectively) that  $m > k$, and that $k = m'$ with $\alpha'(a_{k}) = 0$ and $P \alpha'(a_{k}) = - \frac{1}{2}$.  Since $P \alpha = P \alpha'$, it follows that $\alpha(a_{k}) = \frac{1}{2}$.   Notice that $\alpha'$ is integral in the triangle containing $a_{k}$, and hence is type I.  On the other hand,  $\alpha$ is half-integral in the triangle containing $a_{k}$.  Type III is ruled out, since $\alpha(a_{k}) = \frac{1}{2}$ would force the previous angle to have $\alpha(a_{k-1}) = 0$, contradicting that $m >k$.  
Indeed,  $\alpha$ must be type II, with one end at the corner to the left of $a_{k}$.   After applying positive resolutions,  we see that $P \alpha$ must be type III, whereas  $P\alpha'$ is type II. This contradicts $P \alpha = P \alpha'$.  

In the second case,  $P$ decreases $\alpha$ at $a_k$ by 1, and $P$ leaves the corner number of $\alpha'$ at $a_{k}$ unchanged.   This occurs only if  $m'< k<m$.   Since $k$ was the smallest index where $\alpha$ and $\alpha'$ disagreed, we have that $\alpha(a_{m'}) = \alpha'(a_{m'}) = 0$.   But this contradicts that $a_m$ was the first corner where $\alpha$ is zero. 

In the third case, $P$ decreases $\alpha$ at $a_k$ by $\frac{1}{2}$, and $P$ leaves the corner number of $\alpha'$ at $a_{k}$ unchanged.    The first can occur only if $k = m$ with $\alpha(a_{k}) = 0$ and $P \alpha(a_{k}) = - \frac{1}{2}$.  From $P \alpha = P \alpha'$, it follows that $\alpha'(a_{k}) = - \frac{1}{2}$.  But  $\alpha'$ cannot have negative corner coordinate, since it is in $\RMC^{0}$ and cannot have an endpoint at $v$.  

By repeating the argument above for the corners around $w$, we see that $\alpha$ and $\alpha'$ must agree at the corners around $w$.   We thus conclude that $\alpha= \alpha'$.  

\noindent \underline{Case $j =1$, with $P: \RMC^{1} \to \RMC^{1}$}

Let $\alpha, \alpha' \in \RMC^{1}$ such that $P\alpha = P\alpha'$.   Both $\alpha, \alpha'$ must have an endpoint at the same vertex, so which we assume without loss of generality to be $w$.   As in the previous case, we have $d = \deg(\alpha) = \deg(\alpha')$.   By Lemma \ref{lem:Pisorderpreserving}, $d = \deg(P\alpha) = \deg(P\alpha')$ as well. 

Suppose that $\alpha(a_{s}) > \alpha'(a_{s})$.  Since $P \alpha = P \alpha'$, and $P$ can decrease the corner coordinate at $a_{s}$ only by 0 or $\frac{1}{2}$, it follows that  $P$ decreases $\alpha$ at $a_{s}$ by $\frac{1}{2}$, but leaves $\alpha'$ at $a_{s}$ unchanged. We look in  Figure \ref{fig:degchangeRMC1} for pairs of arrows that start at different $x$-coordinates, but land at the same spot.  These are found only on the far right, when $\alpha(a_{s}) = d + \frac{1}{2}$, $\alpha(b_{0}) = - \frac{1}{2}$ and $\alpha'(a_{s}) = d$, $\alpha'(b_{0}) = 0$.  Figure~\ref{fig:localdegreechange} illustrates such a scenario for $d = 0$.  Notice that because $\alpha'$ has an endpoint at $w$ and $\alpha'(b_{0}) = 0$, $\alpha'(b_{j}) = -\frac{1}{2}$ for some $0 < j \le t$. After taking $P$, $P\alpha'(b_{j}) = 0$. On the other hand, $\alpha(b_{0}) = - \frac{1}{2}$ implies that $P\alpha(b_{i}) > 0$ for all $0 < i \le t$. Thus, $\alpha(a_{s}) > \alpha'(a_{s})$ implies $P\alpha \neq P \alpha'$. 

Next suppose $\alpha(b_{t}) > \alpha'(b_{t})$. $P$ can increase the corner coordinate at $b_{t}$ by $\frac{1}{2}$ or 1.  As before, we can look in Figure \ref{fig:degchangeRMC1}, to find that necessarily
 $\alpha(b_{t}) = d$ and $\alpha'(b_{t})  = d - \frac{1}{2}$. Suppose $d > 0$. Since $d$ is always an integer, $\type_{\alpha}(b_{t}) = \mbox{I}$ and $\type_{\alpha'}(b_{t}) = \mbox{II}$. If $d = 0$, $\type_{\alpha'}(b_{t})$ can be either II or III. But in any case, after the positive resolution, $\type_{P\alpha}(b_{t}) = \mbox{II}$ and $\type_{P\alpha'}(b_{t}) = \mbox{III}$, implying $P\alpha \ne P\alpha'$. 

We may now assume that $\alpha$ and $\alpha'$ agree at $a_{0}, a_{s}, b_{0}, b_{t}$.  The proof that $\alpha(a_{i}) = \alpha'(a_{i})$ for $1 \le i \le s-1$ is the same as in the case of $P: \RMC^{0} \to \RMC^{2}$ above.  
For the remaining corners, let $n= \min\{i \;|\; \alpha(b_{i}) =  - \frac{1}{2}\}$. Then by Algorithm \ref{alg:PNRMC1}, $P\alpha (b_{n}) = 0$, and all $P\alpha(b_{i}) >0$ for $i >n$.  Thus also $n = \max \{i \;|\; P\alpha(b_{i}) =  0\}$.  For $\alpha'$, we similarly have $n' = \min\{i \;|\; \alpha'(b_{i}) =  - \frac{1}{2}\} = \max\{i \;|\; P\alpha'(b_{i}) =  0\}$.  Then $P \alpha = P \alpha'$ implies $n= n'$.   It follows from Algorithm \ref{alg:PNRMC1} for $\RMC^{1}_{w}$ that $\alpha(b_{j}) = \alpha'(b_{j})$ for $1 \le j \le t-1$.   Since all their corner numbers are the same,  we may conclude that $\alpha = \alpha'$.

\noindent \underline{Case $j = 2$, with $P: \RMC^{2} \to \RMC^{0}$} 

Suppose that $\alpha, \alpha' \in \RMC^{2}$ with $P\alpha = P\alpha'$.  As before, $\deg(\alpha) = \deg(\alpha') = d$, and arguments like in the $\RMC^{2}$ case shows that $\alpha$ and $\alpha'$ must agree along the corners $a_{0}, a_{s}, b_{0}$, and $b_{t}$. 

If $\alpha$ and $\alpha'$ disagree at some corner around  $v$ which is not $a_0$ or $a_s$, then let $k$ be the smallest index such that $\alpha(a_k) \neq \alpha'(a_k)$.  Without loss of generality, say $\alpha(a_k) < \alpha'(a_k)$.  Let $m$ and $m'$ be as in Algorithm \ref{alg:PRMC2}.  Taking the  positive resolution  $P$ can cause the corner coordinate at $a_k$ to increase by $0, \frac{1}{2}$ or $1$. If $P$ increases $\alpha$ at $a_k$ by 1 and $\alpha'$ at $a_k$ by $0$ or $\frac{1}{2}$, then $m <k$ and $k \leq m'$. Since $a_k$ is the first corner where $\alpha$ and $\alpha'$ disagree, $\alpha(a_m) = \alpha'(a_m) = - \frac{1}{2}$.  But this contradicts the minimality of $m'$. If $P$ increases $\alpha$ at $a_k$ by $\frac{1}{2}$ but leaves $\alpha'$ at $a_k$ unchanged, then $k = m$ and $k< m'$.  Focusing on the corner $a_{m'}$, from $ m'>m$ it follows that $\alpha$ increases by 1, whereas $\alpha'$ increases from $-\frac{1}{2}$ to 0.   However, $ 0 = P \alpha (a_{m'}) = \alpha(a_{m'}) + 1$ is impossible.    So all the corner coordinates agree, and $\alpha = \alpha'$. \end{proof}

\section{Edges are not zero-divisors}\label{sec:nonzerodivisor}

The goal of this section is to prove Theorem \ref{thm:locallyplanarzerodivisor}, by showing $e \beta \neq 0$ for all $\beta \in \Curve$.  In Section \ref{ssec:multiplication}, we first prove that multiplication by $e\beta \ne 0$ is when $\beta$ is a nontrivial linear combination of reduced multicurves with coefficients in $\CC$. In Section \ref{ssec:deltanozerdivisor} we shows the result of Section \ref{ssec:multiplication} implies the general case where the coefficients of $\beta$ are in $\CC[v_i^{\pm1}]$. Section \ref{ssec:multiplication} requires that $\Sigma$ be locally planar, whereas Section \ref{ssec:deltanozerdivisor} does not. Both proofs are reductive, and make use of Lemma~\ref{lem:split} as its first step. It will allow us to split our analysis according to membership in $\RMC^{j}$.

Recall that $\alpha \in \RMC^{1}_{v}$ if $v$ meets an end of $\alpha$ (see Section \ref{sec:RMC} for the formal definition). 

\begin{lemma}\label{lem:split}
Let $e$ be an edge in a triangulation of $\Sigma$. Let $\gamma \in \Curve$ , and write it as $\gamma = \gamma^{0} + \gamma^{1}+ \gamma^{2}$, where $\gamma^{j}$ is a $\CC[v_{i}^{\pm}]$-linear combination of reduced multicurves in $\RMC^{j}$.  Then $e\gamma = 0$ if and only if $e\gamma^{j} = 0$ for $j= 0, 1, 2$.

Suppose that the edge $e$ has distinct endpoints $v$ and $w$, we may further write $\gamma^{1} = \gamma^{1}_{v} + \gamma^{1}_{w}$ where $\gamma^{1}_{v}$ is a $\CC[v_{i}^{\pm}]$-linear combination of reduced multicurves in $\RMC^{1}_{v}$ and $\gamma_{w}^{1}$ is that of reduced multicurves in $\RMC^{1}_{w}$. Then $e \gamma^{1}  = 0 $ if and only if and $e \gamma^{1}_{v} = e \gamma^{1}_{w} = 0 $. 
\end{lemma}

\begin{proof}
By Proposition~\ref{prop:generatorsofcurvealgebra}, $\RMC$ is the set of generators in the free $\CC[v_{i}^{\pm}]$-module $\Curve$, and by Lemma~\ref{lem:jto2-j} $e\gamma^{j}$ is a $\CC[v_{i}^{\pm}]$-linear combination of elements in $\RMC^{2-j}$. Thus $0 = e\gamma = e\gamma^{2} + e\gamma^{1} + e\gamma^{0}$ implies $e\gamma^{j} = 0$ for each $j$. The converse is clear.  The proof for $\gamma^{1} = \gamma^{1}_{v} + \gamma^{1}_{w}$ is similar. 
\end{proof}

The next simple computational lemma will be used in the proof of Proposition \ref{prop:multiplyingeisinjective}. 

\begin{lemma}\label{lem:extremaldegchange}
Let $\alpha \in \RMC$ and $\deg_{e}(\alpha) = d$. Suppose that one of the conditions below holds:
\begin{enumerate}
\item $\alpha \in \RMC^{0}$, $d > 0$, and $\pi(\alpha) = (d, d)$;
\item $\alpha \in \RMC^{1}_{w}$ and $\pi(\alpha) = (d+\frac{1}{2}, d)$;
\item $\alpha \in \RMC^{2}$ and $\pi(\alpha) = (d+\frac{1}{2}, d+\frac{1}{2})$.
\end{enumerate}
Then $P\alpha(b_{t-1}) > N\alpha(b_{t-1})$.

If either:
\begin{enumerate}
\item $\alpha \in \RMC^{0}$, $d > 0$, and $\pi(\alpha) = (0, 0)$;
\item $\alpha \in \RMC^{1}_{w}$ and $\pi(\alpha) = (0, -\frac{1}{2})$;
\item $\alpha \in \RMC^{2}$ and $\pi(\alpha) = (-\frac{1}{2}, -\frac{1}{2})$,
\end{enumerate}
then $P\alpha(b_{1}) < N\alpha(b_{1})$. 
\end{lemma}

\begin{proof}
The proof is immediate if one apply Algorithms in Section \ref{sec:algorithms}. Here we give the proof of the very first statement to describe how the proof goes. 

Suppose that $\alpha \in \RMC^{0}$, $d > 0$, and $\pi(\alpha) = (d, d)$. Then $\deg_{e}(\alpha) = d$ and $\alpha(b_{t}) = 0$, so we have $\alpha(b_{0}) = 0$. Thus $n = \mathrm{min}\{i\;|\; \alpha(b_{i})\} = 0$. Therefore $P\alpha(b_{t-1}) = \alpha(b_{t-1})$ by Algorithm \ref{alg:PRMC0}. On the other hand, because $\alpha(b_{t}) = d > 0$, $\mathrm{max}\{i\;|\; \alpha(b_{i}\} \le t-1$. Thus $N\alpha(b_{t-1}) < \alpha(b_{t-1})$ and we obtain the result.

The second half of the statement is obtained by symmetry. 
\end{proof}

\subsection{First step - complex coefficients}\label{ssec:multiplication}

\begin{proposition}\label{prop:multiplyingeisinjective}
Suppose that $e$ is an edge of a locally planar triangulation of $\Sigma$.  For any $\beta$ that is a nonzero $\CC$-linear combination of reduced multicurves, the product $e\beta \neq 0$ in $\Curve$. 
\end{proposition}

\begin{proof}
To begin, we split the components of $\beta$ according to their membership in $\RMC^{j}$. Applying Lemma~\ref{lem:split}, we may thus fix $j$ and assume that $\beta = \sum_{i \in I} c_{i}\alpha_{i}$, where $ c_{i} \in \CC$ and $\alpha_{i}\in \RMC^{j}$.   From here on out, the general strategy is to consider the leading terms of $\beta$ and $e \beta$.  We show that some leading term of $e \beta$ is nonzero, and thus $e \beta \neq 0 $.  

Let $d = \max \{ \deg_e(\alpha_{i}) \;|\; i \in I  \}$ and $J = \{ i \;| \; \deg_e(\alpha_{i}) = d \}$.  Then $S = \{\alpha_{i}\}_{i \in J}$ consists of the leading terms in $\beta$, and $PS = \{P\alpha_{i}\}_{i \in J}$ and $NS = \{P\alpha_{i}\}_{i \in J}$ consist of the positive and negative resolutions of the leading terms. By Lemmas \ref{lem:PandNareleadingterms} and \ref{lem:Pisorderpreserving}, $P\alpha_{i}$ and $N\alpha_{i}$ are the only possible leading terms in $e\alpha_{i}$, and their degree is $d + j -1$. So the set of leading terms of $e \beta$ is a subset of $PS \cup NS$. 

However, as we will see, the leading terms of $e \beta$ can be a proper subset of $PS \cup NS$, meaning there can be cancellations among the possible leading terms when computing $e \beta$. Because the resolution maps are injective by Proposition \ref{prop:Pisinjective}, cancellations cannot occur amongst the positive resolutions, and the same is true of the negative resolutions. But $P \alpha_{i} = N \alpha_{k}$ for $i, k \in J$ may occur. See Figures~\ref{fig:cancellationRMC1} and \ref{fig:Pai=Naj} for examples. Our goal is to show that some member of $PS \cup NS$ survives to be a leading term of $e \beta$. 

Based on our discussion above, from now on we thus replace $I$ with $J$ and show $e \beta \neq 0$ when $\beta$ is $\deg_{e}$-homogeneous.  Proof of the following lemma then finishes the proof of Proposition~\ref{prop:multiplyingeisinjective}.

\begin{lemma}\label{lem:homogeneouscase}
Fix $j=0, 1, 2$, and let $\beta = \sum_{i \in J} c_{i}\alpha_{i}$  with $c_{i} \in \CC$ and all the $\alpha_{i}\in \RMC^{j}$ having edge degree $d$.  Then leading terms of $e \beta$ have degree $d+ j -1$, and each is a positive or negative resolution of a leading term of $\beta$.  
\end{lemma}

\begin{proof} 
To distinguish between the possible leading terms, we analyze their projection onto two coordinates, with $\pi(\alpha) = (\alpha(a_{s}), \alpha(b_{t}))$ for any reduced multicurve $\alpha$. See Section~\ref{ssec:changelc}.   Let us denote $\pi(S)= \{\pi(\alpha_{i})\}_{i \in J}$, and similarly $\pi(PS)= \{\pi(P\alpha_{i})\}_{i \in J}$ and $\pi(NS)= \{\pi(N\alpha_{i})\}_{i \in J}$.  We order the projected coordinates using lexicographical ordering $\succ$; that is,  $(x, y) \succ (x', y')$ if $x > x'$ or $x = x'$ and $y > y'$. There is a maximal $\pi_{\max}$ in $\pi(S)$.  We begin with the cases $j = 1, 2$, as they are simpler. 

For the case $j= 1$,  compare the action of the $P$ and $N$ maps, as depicted in Figure~\ref{fig:degchangeRMC1}.  Observe that $\pi(P\alpha) \succ \pi(N\alpha)$ for all $\alpha \in \RMC^{1}$ except those with $\pi(\alpha) = (d+\frac{1}{2}, d)$. 
If $\pi_{\max} \neq  (d+\frac{1}{2}, d)$,  pick any $\alpha_{\max} \in \{\alpha_{i}\}_{i \in J}$ with $\pi( \alpha_{\max}) = \pi_{\max}$.  Since $\pi_{\max} \in \pi(PS) \setminus \pi(NS)$,  $P \alpha_{\max} \neq N \alpha_{i}$ for every $i \in J$.  Moreover, because of  the injectivity of the positive resolution map, $P \alpha_{\max}$ is distinct from every other $P \alpha_{i}$ in $PS$ as well.  Thus $P \alpha_{\max}$ will be a leading term of $e \beta$. 

However, if $\pi_{\max} =  (d+\frac{1}{2}, d)$, there may exist $i, k \in J$ such that $\pi(\alpha_{i}) = \pi(\alpha_{k})  = \pi_{\max}$ and $\pi(P\alpha_{i}) = \pi(N\alpha_{k})$. See Figure \ref{fig:cancellationRMC1} for an illustrated example. In this case, the projected coordinates are not enough to determine how to locate a suitable $\alpha_{\max}$. 

In this case, pick $\alpha_{k} \in \{\alpha_{i}\;|\; \pi(\alpha_{i}) = \pi_{\max}\}$ with the maximum $P\alpha_{k}(b_{t-1})$. We claim that $P\alpha_{k}$ will survive after the cancellation with other terms. Suppose not. Then $P\alpha_{k} = N\alpha_{\ell}$ for some $\alpha_{\ell} \in \{\alpha_{i}\;|\; \pi(\alpha_{i}) = \pi_{\max}\}$. Then by Lemma \ref{lem:extremaldegchange}, $P\alpha_{k}(b_{t-1}) = N\alpha_{\ell}(b_{t-1}) < P\alpha_{\ell}(b_{t-1})$. It violates the maximality of $P\alpha_{k}(b_{t-1})$. Therefore such $\ell$ does not exist and $P\alpha_{k}$ is a nonzero leading term of $e\beta$. 

\begin{figure}[!ht]
\includegraphics[width=0.4\textwidth]{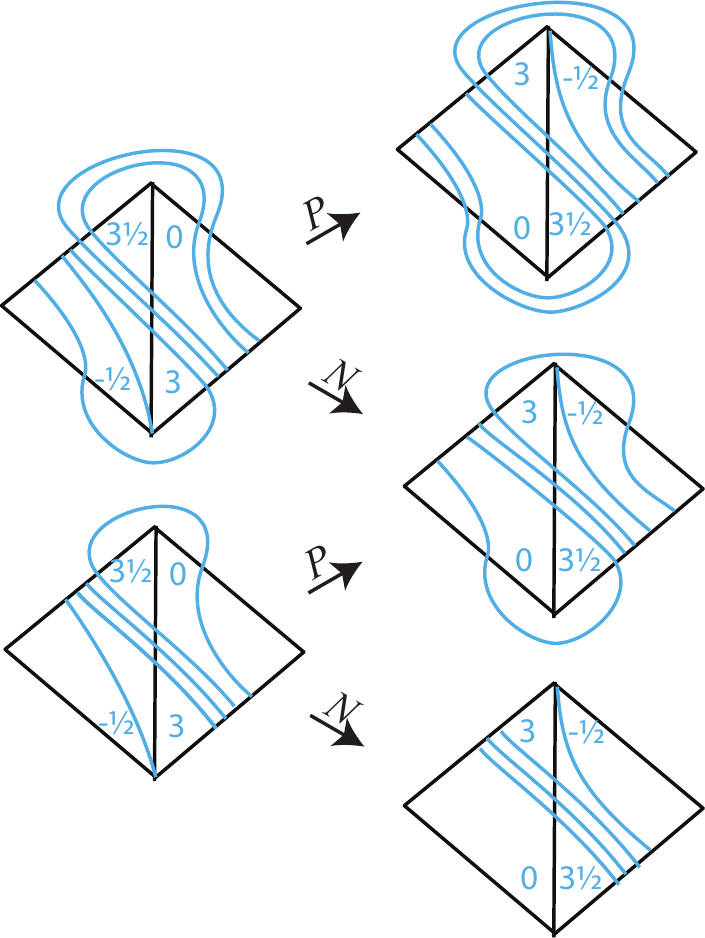}
\caption{Examples of $\alpha_{i}, \alpha_{k} \in \RMC^{1}$ such that $\pi(P\alpha_{i}) = \pi(N\alpha_{k})$. Notice that $N\alpha_{i}(b_{t-1}) < P\alpha_{k}(b_{t-1})$.}
\label{fig:cancellationRMC1} 
\end{figure}

The case $j = 2$ is very similar.  From Figure \ref{fig:degchangeRMC2},  we see that $\pi(P\alpha) \succ \pi(N\alpha)$ for all $\alpha \in \RMC^{2}$, except those with $\pi(\alpha) = (d+\frac{1}{2}, d+\frac{1}{2})$ or $ (-\frac{1}{2}, -\frac{1}{2})$.   If  $\pi_{\max} \neq (d+\frac{1}{2}, d+\frac{1}{2}),  (-\frac{1}{2}, -\frac{1}{2})$,  pick any $\alpha_{\max}$ with $\pi( \alpha_{\max}) = \pi_{\max}$.    If $\pi_{\max} = (d+\frac{1}{2}, d+\frac{1}{2})$ or $\pi_{\max} = (-\frac{1}{2}, -\frac{1}{2})$, cancellations are possible. 

Suppose that $\pi_{\max} = (d+\frac{1}{2}, d+\frac{1}{2})$. Pick $\alpha_{k} \in \{\alpha_{i}\;|\; \pi(\alpha_{i}) = \pi_{max}\}$ with the maximum $P\alpha_{k}(b_{t-1})$. If $P\alpha_{k} = N\alpha_{\ell}$ for some other $\alpha_{\ell} \in \{\alpha_{i}\;|\; \pi(\alpha_{i}) = \pi_{max}\}$, then $P\alpha_{k}(b_{t-1}) = N\alpha_{\ell}(b_{t-1}) < P\alpha_{\ell}(b_{t-1})$ by Lemma \ref{lem:extremaldegchange}. Thus such $\alpha_{\ell}$ does not exist and $P\alpha_{k}$ is a leading term in $e\beta$. When $\pi_{\max} = (-\frac{1}{2}, -\frac{1}{2})$, then one can show in a similar way by using $N$ and $b_{1}$ instead of $P$ and $b_{t-1}$ by symmetry. 

The case $j = 0$ also proceeds along the same lines, but we are first required to subdivide $\RMC^{0}$ based on whether the triangles on either side of $e$ are integral of type I,  or fractional of type II. (Type III is not possible here.)  Let $\RMC^{0}_{\ri \ri}$ be the set of $\alpha \in \RMC^{0}$ with $ \type_{\alpha}(a_{s}) = \mbox{I}$ and $\type_{\alpha}(b_{t}) = \mbox{I}$, and let $\RMC^{0}_{\ri \rf}$ be the set where $\type_{\alpha}(a_{s}) = \mbox{I}$ and $\type_{\alpha}(b_{t}) = \mbox{II}$.  Similarly define $\RMC^{0}_{\rf \ri}$ and $\RMC^{0}_{\rf \rf}$.  Clearly, $\RMC^{0} = \RMC^{0}_{\ri \ri} \sqcup \RMC^{0}_{\ri \rf} \sqcup \RMC^{0}_{\rf \ri} \sqcup \RMC^{0}_{\rf \rf}$.

\begin{lemma}\label{lem:RMC0ii}
Let $\beta = \sum_{i \in I}c_{i}\alpha_{i}$, with $c_{i} \in \CC$ and $\alpha_{i} \in \RMC^{0}$. Let $T_{\ri \ri} := \{P\alpha_{i}, N\alpha_{i}\;|\; \alpha_{i} \in \RMC_{\ri \ri}^{0}\}$ and define $T_{\ri \rf}$, $T_{\rf \ri}$, $T_{\rf \rf}$ in a similar way. Then $T_{\ri \ri}$, $T_{\ri \rf}$, $T_{\rf \ri}$, and $T_{\rf \rf}$ are mutually disjoint. 
\end{lemma}

\begin{proof}
We show that $T_{\ri \ri}$ and $T_{\rf \ri}$ are disjoint.  The other cases are similar. 

When integral on the left, we have $\type_{\alpha}(a_{s}) = \mbox{I}$, and we show that $P\alpha(a_{s}^{R})  \geq 0 $.  There are two cases, either $\type_{P\alpha}(a_{s})$ is I or II.  The type~I case is clear.  If $\type_{P\alpha}(a_{s})$ is type II, one of $P\alpha(a_{s})$ and $P\alpha(b_{0})$ is $-\frac{1}{2}$,  and so $P\alpha(a_{s}^{R}) > 0$.   Similarly, $N\alpha(a_{s}^{R}) \geq 0$.

On the other hand, when fractional on the left, we have $\type_{\alpha}(a_{s}) = \mbox{II}$, and then $\type_{P\alpha}(a_{s})$ and $\type_{N\alpha}(a_{s})$ are II or III. Furthermore, $P\alpha(a_{s}^{R}) = N\alpha(a_{s}^{R}) = -\frac{1}{2}$. 
\end{proof}

In our setting, Lemma \ref{lem:RMC0ii} implies that cancellations are possible only when the types of the triangles on either side of $e$ agree.  
Let us further assume that $\beta = \sum_{i \in J}c_{i}\alpha_{i}$ is a linear combination of $\alpha_{i}$ that are in one of $\RMC^{0}_{\ri \ri}$, $\RMC^{0}_{\ri \rf}$, $\RMC^{0}_{\rf \ri}$, or $\RMC^{0}_{\rf \rf}$. 

Lexicographically order the projected coordinates $ \pi(\alpha_{i})$ for $i\in J$, and let $\pi_{\max}$ be the maximal coordinates with respect to $\succ$.  From Figure \ref{fig:degchangeRMC0}, one can check that if we restrict the domain to one of $\RMC^{0}_{\ri \ri}$, $\RMC^{0}_{\ri \rf}$, $\RMC^{0}_{\rf \ri}$, and $\RMC^{0}_{\rf \rf}$, then $P$ and $N$ are $\succ$-preserving maps.  We remark that this is not true for $\RMC^{0}$ without the subdivision.  

\begin{figure}[!ht]
\includegraphics[width=0.4\textwidth]{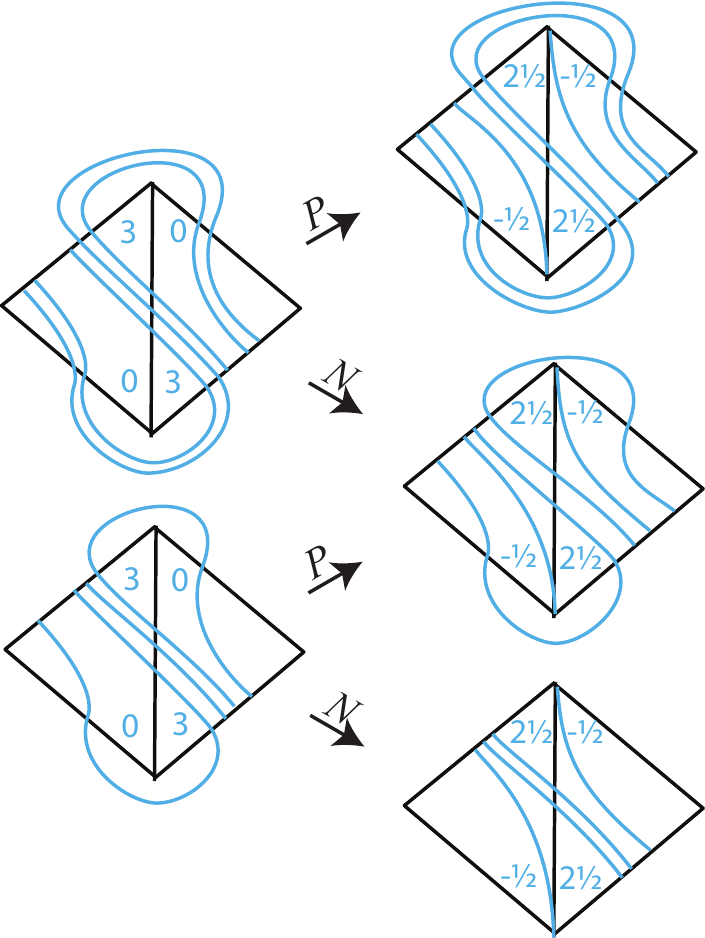}
\caption{Examples of $\alpha_{i}, \alpha_{k} \in \RMC^{0}_{\ri \ri}$ such that $\pi(P\alpha_{i}) = \pi(N\alpha_{k})$. Notice that $N\alpha_{i}(b_{t-1}) < P\alpha_{k}(b_{t-1})$.}\label{fig:Pai=Naj}
\end{figure}

One can verify that $\pi(P\alpha) \succ \pi(N\alpha)$ except possibly when $\pi(\alpha) = (0,0), (d,d)$.  If $ \pi_{\max} \ne (0,0), (d, d)$, pick any $\alpha_{\max}$ such that $\pi(\alpha_{\max} )= \pi_{\max}$. If $ \pi_{\max} = (d, d)$, pick $\alpha_{\max}$ so that $P\alpha_{\max}(b_{t-1})$ is the maximum. Arguing like in the $j= 1, 2$ cases, we see that $P\alpha_{\max}$ will survive in $e \beta$. The case of $\pi(\alpha) = (0, 0)$ is obtained in a similar way. 
\end{proof}

We conclude our proof of Proposition~\ref{prop:multiplyingeisinjective}.  If $\beta$ has a leading term of degree $d$, Lemma~\ref{lem:homogeneouscase} shows that $e \beta$ has a nonzero leading term with the expected degree $d + j -1$.  Hence $e \beta \neq 0$.
\end{proof}

\subsection{Second step - general coefficients} \label{ssec:deltanozerdivisor}
Unlike Proposition~\ref{prop:multiplyingeisinjective}, Proposition \ref{prop:deltanozerodivisor} below is valid for arbitrary surfaces, and does not require $\Sigma$ to be locally planar.

\begin{proposition}\label{prop:deltanozerodivisor}
Let $e$ be an edge of a triangulation of $\Sigma$, and suppose that $e \beta \neq 0$ for any  $\beta$ that is a non-zero $\CC$-linear combination of reduced multicurves. Then $e$ is not a zero divisor in $\Curve$. 
\end{proposition}

\begin{proof}[Proof of Proposition \ref{prop:deltanozerodivisor}]
Let $\gamma \in \Curve$ be nonzero and $e \gamma = 0$.   First consider the case where the edge $e$ has two distinct vertices, $v$ and $w$. We may assume that in the vector of vertices $\bv = (v_{1}, v_{2}, \cdots, v_{n})$, $v = v_{1}$ and $w = v_{2}$. 

By Lemma~\ref{lem:split}, we may fix $j=0,1,2$ and assume that $\gamma =  \sum_{k \in I}f_{k}(v_{i}^{\pm})\alpha_{k}$, where $f_{k}(v_{i}^{\pm}) \in \CC[v_{i}^{\pm}]$ and $\alpha_{k} \in \RMC^{j}$.   In the case that $j = 1$, we may further assume that $\alpha_{k}$ belongs to $\RMC^{1}_{v}$ or $\RMC^{1}_{w}$, and without loss of generality, let us assume $\RMC^{1}_{v}$.  
 
Because of   Proposition~\ref{prop:generatorsofcurvealgebra}, we now rewrite $\gamma$ as a linear combination of vertex classes as
\[
	\gamma = \sum_{\bm \in \ZZ^{C}}\gamma_{\bm}\bv^{\bm}, 
\]
and all $\gamma_{\bm}$ are $\CC$-linear combinations of elements of one of $\RMC^{0}$, $\RMC^{1}_{v}$, or $\RMC^{2}$.

In the first case, all $\gamma_{\bm}$ are $\CC$-linear combinations of reduced multicurves in $\RMC^{0}$. Then for any resolution of $e\gamma_{\bm}$, there is no resolution at a vertex (the second relation in Definition \ref{def:curvealgebra}). Thus it does not produce any extra vertex class, so as a linear combination of the vertex classes, 
\[
	e\gamma = \sum_{\bm \in \ZZ^{C}}e\gamma_{\bm}\bv^{\bm}.
\]
Therefore $e\gamma = 0$ implies that $e\gamma_{\bm} = 0$ for all $\bm$. Then by the assumption on $e$, $\gamma_{\bm} = 0$ for all $\bm$. Therefore $\gamma = 0$. 

If $\gamma_{\bm}$ is a $\CC$-linear combination of elements elements in $\RMC^{1}_{v}$, then for every resolution of $e\gamma_{\bm}$, there is only one resolution at a vertex $v$. Thus $ve\gamma_{\bm}$ is a $\CC$-linear combination of reduced multicurves, and we have the unique decomposition 
\begin{equation}\label{eqn:formulaegamma}
	e\gamma = \sum_{\bm \in \ZZ^{C}}ve\gamma_{\bm}\bv^{\bm - \be_{1}}
\end{equation}
where $\be_{1}$ is the first standard coordinate vector. Now $e\gamma = 0$ implies $ve\gamma_{\bm} = 0$ for all $\bm$, since the $\bv^{\bm}$ are linearly independent in $\Curve$. Since $v$ is a unit in $\Curve$, it now follows that $e\gamma_{\bm} = 0$. Our assumption $e \beta \neq 0$ for any $\beta$ that is a non-zero $\CC$-linear combination of reduced multicurves means that $\gamma_{\bm} = 0$ for all $\bm$.  Therefore $\gamma = 0$. 

The cases of $\RMC^{2}$ are similar and we can obtain the same conclusion.  The only difference is that instead of \eqref{eqn:formulaegamma}, we have 
\begin{equation}\label{eqn:formulaegamma2}
	e\gamma = \sum_{\bm \in \ZZ^{C}}vwe\gamma_{\bm}\bv^{\bm - \be_{1} - \be_{2}}.
\end{equation}

When $e$ is an edge whose ends are both $v$, then we have a decomposition $\RMC = \RMC^{0}\sqcup \RMC^{2}$.  We also argue in the same way. The only difference here is that in $\RMC^{2}$ case, we have \eqref{eqn:formulaegamma} instead of \eqref{eqn:formulaegamma2} because there is only one endpoint resolution at $v$.
\end{proof}

Putting together Propositions \ref{prop:multiplyingeisinjective} and \ref{prop:deltanozerodivisor}, we immediately arrive at the statement of 
Theorem \ref{thm:locallyplanarzerodivisor}, which states that when $e$ is an edge of a locally planar triangulation of $\Sigma$, then $e$ is not a zero divisor in $\Curve$.  This also completes the proof in Section~\ref{sec:outline} of Theorem \ref{thm:mainthmintro}, which states that when $\Sigma$ is locally planar, the Poisson algebra homomorhpism  $\Phi: \Curve \to C^{\infty}( \Teich)$ is injective.


\section{The Roger-Yang skein algebra $\Arc$}\label{sec:arcalgebra}

Having finished the proofs of Theorems~\ref{thm:nonzerodivisorimpliesinjectivityintro} and  \ref{thm:mainthmintro} about the commutative curve algebra $\Curve$, we now turn to the quantum setting.  Let us now consider the skein algebra $\Arc$ as defined by Roger and Yang \cite{RogerYang14}.  

\subsection{Framed knots and arcs in a thickened punctured surface}
Previously, we considered only loops and arcs in the 2-dimensional punctured surface $\Sigma$.   We now go up a dimension, to the 3-dimensional product $\Sigma \times [0,1]$.  In particular we define framed knots, arcs, and generalized framed links in $\Sigma \times [0,1]$ as analogies of, respectively, the loops, arcs, and generalized multicurves in the 2-dimensional punctured surface $\Sigma$.   Recall that $P$ are the punctures of $\Sigma$.  

A \emph{framed knot in $\Sigma \times [0,1]$} is an embedding of an oriented annulus into $\overline \Sigma \times[0,1]$ that is disjoint from $P \times [0,1]$.  A \emph{framed arc in $\Sigma \times [0,1]$} is a map of a strip $[0,1] \times[0,1]$ into $\overline \Sigma \times [0,1]$ so that on the set $(0,1) \times [0,1]$ it is an embedding into $\overline \Sigma \times [0,1]$ that is  disjoint from $P \times [0,1]$, and on each of the sets $\{0\} \times [0,1]$ and $\{1 \} \times [0,1]$, it is an embedding into $P \times [0,1]$ that is increasing in the second coordinate. A \emph{generalized framed link in $\Sigma \times [0,1]$} is a disjoint union of finitely many framed knots and framed arcs.  Thus, although more than one component of a generalized framed link may end above a particular puncture  $p_i$, the components must do so at different heights above $p_i$.  
 
 We consider generalized framed links up to a suitable notion of regular isotopy which is described in detail in \cite{RogerYang14}.  In particular, regular isotopy of generalized framed links can be described using four moves on their diagrams (the usual  three Reidemeister moves on the interior and one more move for ends of arcs meeting at a vertex).  In this paper, we will assume that diagrams are obtained from representatives with vertical framing, so that the restriction of the embedding from the definition of a framed knot or arc is always increasing in the second coordinate.  Breaks in the diagrams are enough to show crossing information at double points in the interior or at a vertex, but further numbering according to height will be necessary when more than two ends of arcs meet at a vertex.  
We say that a generalized framed link in $\Sigma \times [0,1]$ is \emph{simple} or \emph{reduced} when its diagram is a reduced generalized multicurve in $\Sigma$.   In particular, the empty set $\emptyset$ is a reduced generalized framed link.

There is a natural \emph{stacking operation} for two generalized framed links $\alpha, \beta$  in $\Sigma \times [0,1]$.  In particular,  $\alpha$ stacked on top of $\beta$ is the union of the framed curve $\alpha' \subset \overline \Sigma \times [0, \frac12]$ (obtained by rescaling $\alpha$ in $ \overline \Sigma \times [0, 1]$ vertically by half) and of  the framed curve $\beta' \subset \overline \Sigma \times [\frac12, 1]$ (obtained by rescaling $\beta$ in $\overline \Sigma \times [0, 1]$ vertically by half).   We denote the framed link obtained from $\alpha$ stacked on top of $\beta$ as $\alpha \cdot \beta$.

\subsection{Roger-Yang skein algebra}\label{ssec:RYarcalgebra} 
Suppose that $h$ is some indeterminate.  Then  the ring of power series in $h$, equipped with a natural $h$-adic topology, will be denoted by $\CC[[h]]$. Furthermore, in this ring, we distinguish a certain power series $q= e^{h/4}\in \CC[[h]]$.  In addition, let there be an indeterminate $v_i$ associated to each puncture $p_i \in P$, such that a formal inverse $v_i^{-1}$ exists. Let $\CC[[h]][v_i^{\pm1}]$ denote the commutative $\CC[[h]]$-algebra generated by $\{v_{i}^{\pm 1}\}$.

\begin{definition}\label{def:RYalgebra}
Let $\Sigma$ be a surface with punctures. Let $h$ be some indeterminate, and associate a variable $v_i$ to every punctures $p_i$.  Then the \emph{Roger-Yang skein algebra $\Arc$} is the $\CC[[h]][v_i^{\pm1}]$-algebra freely generated by by the generalized framed links on $\Sigma$ modded out by the following relations:
\begin{align*}
&1)
\quad
\begin{minipage}{.5in}\includegraphics[width=\textwidth]{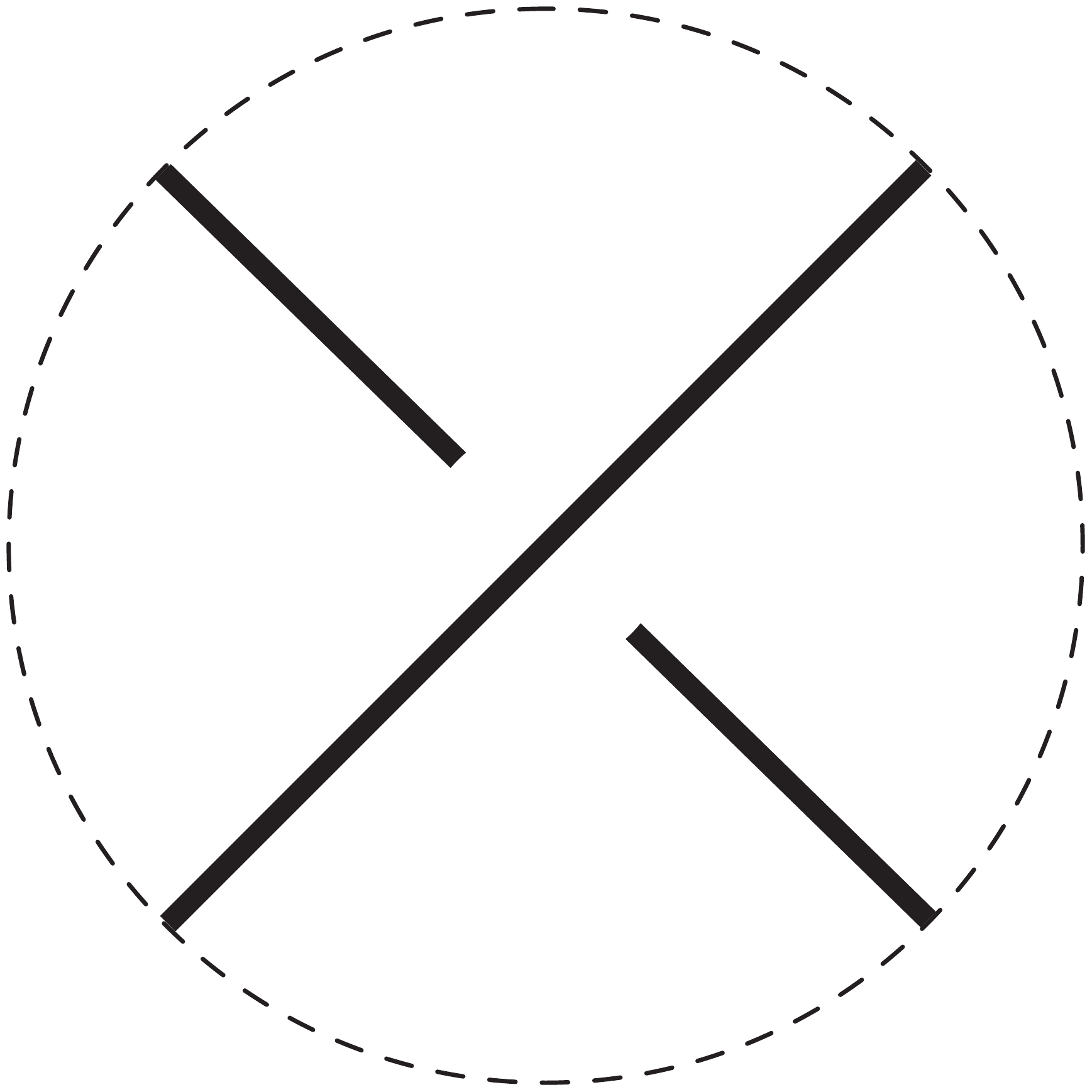}\end{minipage} 
-  \left( q\begin{minipage}{.5in}\includegraphics[width=\textwidth]{rel-skein2.pdf}\end{minipage} 
+q^{-1}\begin{minipage}{.5in}\includegraphics[width=\textwidth]{rel-skein3.pdf}\end{minipage}  \right)\\
&2)
\quad 
v_i \begin{minipage}{.5in}\includegraphics[width=\textwidth]{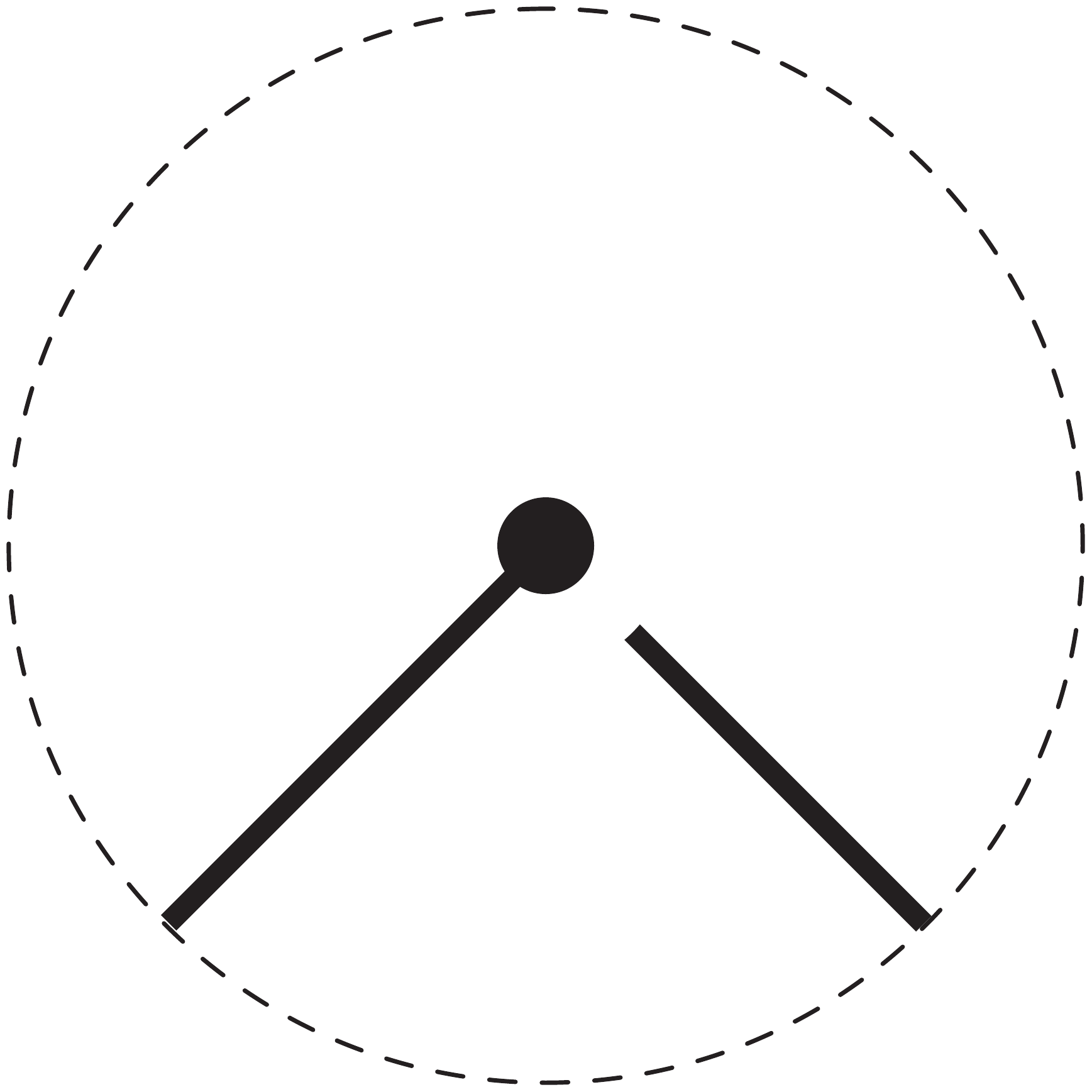}\end{minipage} 
- \left(  q^\frac{1}{2}\begin{minipage}{.5in}\includegraphics[width=\textwidth]{rel-punctureskein2.pdf}\end{minipage} 
+ q^{-\frac{1}{2}}\begin{minipage}{.5in}\includegraphics[width=\textwidth]{rel-punctureskein3.pdf}\end{minipage}  \right)\\
&3)
\quad 
\begin{minipage}{.5in}\includegraphics[width=\textwidth]{rel-framing.pdf} \end{minipage} 
- ( - q^2 - q^{-2} )\\
&4)
\quad 
\begin{minipage}{.5in}\includegraphics[width=\textwidth]{rel-puncture.pdf} \end{minipage} 
-( q + q^{-1} ),\\ 
\end{align*}
where we use $q= e^{h/4}$, and where the diagrams in the relations are assumed to be identical outside of the small balls depicted.  Multiplication of elements in $\Arc$ is the one induced by the stacking operation for generalized framed links. 
\end{definition}

It is a simple exercise to prove that $\Arc$ is well-defined.  Observe that in the absence of punctures on $\Sigma$, the Roger-Yang $\Arc$ and the Kauffman bracket skein algebra coincide. Hence, $\Arc$ can be regarded as an extension of the Kauffman bracket skein algebra.  

From comparing the definitions of the curve algebra $\Curve$ and the Roger-Yang skein algebra $\Arc$, we see that the Roger-Yang algebra is some non-commutative version of the curve algebra. Formally, we have the following theorem.

\begin{theorem}[\protect{\cite[Proposition 2.10]{RogerYang14}}]\label{thm:quantization}
Let $p: \Arc \to \Curve$ be the map which associates a generalized framed link in $\Sigma \times [0,1]$ with its projection to a generalized multicurve in $\Sigma$.  Then  $p$ induces an isomorphism between the $\CC$-algebras $\Arc/ (h \cdot \Arc)$ and $\Curve$.  

Hence,  $\Arc$ is a deformation quantization of $\Curve$.  
\end{theorem}

The above theorem generalizes the analogous statements about the Kauffman bracket skein algebra \cite{Turaev91,  HostePrz90, BullockFrohmanJKB99}.   For the definition and details about deformation quantizations, see \cite{Kontsevich03}.

Although one is commutative whereas the other is usually not, the underlying module structure of $\Arc$ is no more complicated than $\Curve$. We say that a $\CC[[h]]$-module $M$ is \emph{topologically free} if $M \cong V\otimes \CC[[h]]$ for some vector space $V$ in the category of $\CC[[h]]$-modules. 

\begin{theorem}[\protect{\cite[Theorem 2.4]{RogerYang14}}]\label{thm:topologicallyfree}
The algebra $\Arc$ is topologically free. Furthermore, $\Arc \cong \Curve[[h]]$ as $\CC[[h]]$-modules. 
\end{theorem}

\begin{remark}
Variations in the definition of $\Arc$ exist in the literature. In particular, let $\cA^{A}(\Sigma)$ be the  $\ZZ[A][v_{i}^{\pm}]$-algebra generated by $\RMC$ on $\Sigma$ and with the same four relations as in Definition~\ref{def:RYalgebra}. Observe that $\cA^{A}(\Sigma)$ can be regarded as a coordinate restriction of $\Arc$, by mapping $A \to q = e^{h/4}$.  Thus statements about $\cA^{A}(\Sigma)$ apply also to $\Curve$ and $\Arc$.  In particular, $\Arc$ and $\Curve$ are also finitely generated, with an explicit generating set \cite{BKPW16JKTR},  and a presentation is known for certain small surfaces including the three-punctured sphere and the one-punctured torus \cite{BKPW16Involve}.   
\end{remark}

\subsection{Integrality of the Roger-Yang algebras}\label{sec:consequences}

As we mentioned in the introduction,  $\Arc$ seems a likely candidate to be a quantization of the decorated Teichm\"uller space. In the case of the Kauffman bracket skein algebra, its integrality was an important step towards showing that it is a quantization of the decorated Teichm\"uller space \cite{Bullock97, PrzytyckiSikora00, PrzytyckiSikora19}.   Analogously, we also have integrality for the Roger-Yang algebras. 

\begin{theorem} \label{thm:nozerodivisors}
Suppose that Conjecture \ref{conj:injectivity} is true for a punctured surface $\Sigma$. Then  $\Arc$, $\cA^{A}(\Sigma)$, and $\Curve$ are all domains. In particular, if $\Sigma$ is locally planar, then $\Arc$,  $\cA^{A}(\Sigma)$, and $\Curve$ are domains.
\end{theorem}

\begin{proof}
By the proof of Theorem \ref{thm:nonzerodivisorimpliesinjectivity}, $\Curve$ is a subalgebra of $\CC[\lambda_{i}^{\pm}]$. The latter is an integral domain, thus $\Curve$ is, too.

Recall from Theorem~\ref{thm:topologicallyfree} that $\Arc$ is topologically free. Thus as a $\CC[[h]]$-module, $\Arc \cong \Curve[[h]]$. Let $\alpha, \beta \in \Arc$ be two nonzero elements. Then $\alpha$ (resp. $\beta$) can be written as $\sum_{i \ge m}\alpha_{i}h^{i}$ (resp. $\sum_{i \ge n}\beta_{i}h^{i}$) with $\alpha_{i}, \beta_{i} \in \Curve$. Now
\[
	\alpha \beta = \alpha_{m}\beta_{n}h^{m+n} + O(h^{m+n+1}).
\]
Since the smallest degree term is nonzero by the classical case, $\alpha\beta \ne 0$ in $\Arc$. 

The algebra $\cA^{A}(\Sigma)$ is a subalgebra of $\Arc$ by sending $A \mapsto q = e^{h/4}$. 
\end{proof}

\bibliographystyle{alpha}

\begin{thebibliography}{BKPW16b}

\bibitem[BZ05]{BerensteinZelevinsky05}
A. Berenstein and A. Zelevinsky.
\newblock Quantum cluster algebras. 
\newblock {\em Adv. Math.} 195 (2005), no. 2, 405--455.

\bibitem[BKPW16a]{BKPW16JKTR}
M. Bobb, S. Kennedy, H. Wong, and D. Peifer.
\newblock Roger and Yang's Kauffman bracket arc algebra is finitely generated. 
\newblock {\em J. Knot Theory Ramifications} 25 (2016), no. 6, 1650034, 14 pp.

\bibitem[BKPW16b]{BKPW16Involve}
M. Bobb, D. Peifer, S. Kennedy, and H. Wong.
\newblock Presentations of Roger and Yang's Kauffman bracket arc algebra. 
\newblock {\em Involve} 9 (2016), no. 4, 689--698.

\bibitem[Bul97]{Bullock97}
D. Bullock.
\newblock Rings of ${\rm SL}_2({\bf C})$-characters and the Kauffman bracket skein module
\newblock {\em Commentarii Mathematici Helvetici} 72 (1997), no. 4, 521--542. 


\bibitem[BFKB99]{BullockFrohmanJKB99}
D. Bullock, C. Frohman, and J. Kania-Bartoszy\'nska. 
\newblock Understanding the Kauffman bracket skein module. 
\newblock {\em J. Knot Theory Ramifications} 8 (1999), no. 3, 265--277. 


\bibitem[BFKB02]{BullockFrohmanJKB02}
D. Bullock, C. Frohman, and J. Kania-Bartoszy\'nska. 
\newblock The Kauffman bracket skein as an
algebra of observables. 
\newblock {\em Proc. Amer. Math. Soc.} 130 (2002), 2479–2485.

\bibitem[CM12]{CharlesMarche12}
L. Charles and J. March\'e.
\newblock Multicurves and regular functions on the representation variety of a surface in $SU(2)$. 
\newblock {\em Comment. Math. Helv.} 87 (2012), no. 2, 409--431. 

\bibitem[FG06]{FockGoncharov06}
V. Fock and A. Goncharov.
\newblock Moduli spaces of local systems and higher Teichm\"uller theory. 
\newblock {\em Publ. Math. Inst. Hautes \'Etudes Sci.} 103 (2006), 1--211.

\bibitem[FST08]{FominShapiroThurston08}
S. Fomin, M. Shapiro, and D. Thurston.
\newblock Cluster algebras and triangulated surfaces. I. Cluster complexes.
\newblock {\em Acta Math.} 201 (2008), no. 1, 83--146.

\bibitem[FKBL19]{FrohmanJKBLe19}
C. Frohman, J. Kania-Bartoszynska, T. L\^e.
\newblock Unicity for representations of the Kauffman bracket skein algebra. 
\newblock {\em Invent. Math.} 215 (2019), no. 2, 609--650.


\bibitem[GSV05]{GekhtmanShapiroVainshtein05}
M. Gekhtman, M. Shapiro, M., A. Vainshtein, A. 
\newblock Cluster algebras and Weil-Petersson forms. 
\newblock {\em Duke Mathematical Journal}, 127 (2005), no. 2, 291--311.

\bibitem[HP90]{HostePrz90}
J. Hoste, J. Przytycki.
\newblock Homotopy skein modules of orientable 3-manifolds. 
\newblock {\em Math. Proc. Cambridge Philos. Soc.} 108 (1990), no. 3, 475–488. 

\bibitem[Kon03]{Kontsevich03}
M. Kontsevich.
\newblock Deformation quantization of Poisson manifolds. 
\newblock {\em Lett. Math. Phys.} 66 (2003), no. 3, 157--216.

\bibitem[Mat07]{Matveev07}
S. Matveev.
\newblock {\em Algorithmic topology and classification of 3-manifolds. Second edition.} 
\newblock Algorithms and Computation in Mathematics, 9. Springer, Berlin, 2007. xiv+492 pp.

\bibitem[Mon09]{Mondello09}
G. Mondello.
\newblock Triangulated Riemann surfaces with boundary and the Weil-Petersson Poisson structure. 
\newblock {\em J. Differential Geom.} 81 (2009), no. 2, 391--436.

\bibitem[Mul16]{Muller16}
G. Muller. 
\newblock Skein and cluster algebras of marked surfaces. 
\newblock {\em Quantum Topol.} 7 (2016), no. 3, 435--503.

\bibitem[Pen87]{Penner87}
R. C. Penner.
\newblock The decorated Teichm\"uller space of punctured surfaces. 
\newblock {\em Comm. Math. Phys.} 113 (1987), 299--339.

\bibitem[Pen92]{Penner92}
R. C. Penner. 
\newblock Weil-Petersson volumes. 
\newblock {\em J. Differential Geom.} 35 (1992), no. 3, 559--608.

\bibitem[PS00]{PrzytyckiSikora00}
J. H. Przytycki and A. Sikora.
\newblock On skein algebras and $Sl_{2}(\CC)$-character varieties.
\newblock {\em Topology} 39 (2000), no. 1, 115--148.

\bibitem[PS19]{PrzytyckiSikora19}
J. H. Przytycki and A. Sikora.
\newblock Skein algebras of surfaces. 
\newblock {\em Trans. Amer. Math. Soc.} 371 (2019), no. 2, 1309--1332.

\bibitem[RY14]{RogerYang14}
J. Roger and T. Yang.
\newblock The skein algebra of arcs and links and the decorated Teichm\"uller space.
\newblock {\em J. Differential Geom.} 96 (2014), no. 1, 95--140. 

\bibitem[Tur91]{Turaev91}
V. G. Turaev.
\newblock Skein quantization of {P}oisson algebras of loops on surfaces.
\newblock {\em Ann. Sci. \'Ecole Norm. Sup. (4)}, 24(6):635--704, 1991.

\bibitem[Whi37]{Whitney37}
H. Whitney. 
\newblock On regular closed curves in the plane. 
\newblock {\em Compositio Math.} 4 (1937), 276--284.

\end{thebibliography}

\end{document}